\documentclass{article}
\usepackage{amssymb,amsmath,amsthm,tikz-cd,graphicx,mathtools,framed,stmaryrd}
\usepackage{xcolor,float,hhline,multirow,colortbl,bbm,tabularx,array}
\usepackage{titlesec}
\usepackage{lipsum}
\usepackage[top=2cm,bottom=2cm,left=3cm,right=3cm,includefoot]{geometry}
\usepackage[nottoc]{tocbibind}
\usepackage{hyperref}
\usepackage[title]{appendix}
\hypersetup{
 colorlinks=true,
 citecolor=red,
 linkcolor=blue,
 urlcolor=blue,
}
\usepackage{tikz,tikz-3dplot}
\tdplotsetmaincoords{60}{110}
\pgfmathsetmacro{\rvec}{2}
\pgfmathsetmacro{\thetavec}{45}
\pgfmathsetmacro{\phivec}{45}

\theoremstyle{plain}
\newtheorem{thm}{Theorem}[section]
\newtheorem{lem}[thm]{Lemma}
\newtheorem{cor}[thm]{Corollary}
\newtheorem{prop}[thm]{Proposition}

\theoremstyle{definition}
\newtheorem{defn}[thm]{Definition}
\newtheorem{exa}[thm]{Example}
\newtheorem{rem}[thm]{Remark}
\newtheorem{nota}[thm]{Notation}
\newtheorem{const}[thm]{Construction}
\newtheorem{obs}[thm]{Observation}
\newtheorem{conv}[thm]{Convention}
\newtheorem{ques}[thm]{Question}

\newcommand{\RNum}[1]{\uppercase\expandafter{\romannumeral #1\relax}}
\newcommand{\rNum}[1]{\lowercase\expandafter{\romannumeral #1\relax}}
\newcommand{\salash}{\mathbin{/\mkern-3mu/}}

\allowdisplaybreaks

\begin{document}
\baselineskip=0.7cm

% Title page
\begin{titlepage}
\begin{center}
\vspace*{1cm}
{\Huge
\textbf{Categorical Hopf map}}
\\[1.5cm]
{\Large by \\ Ali Khalili Samani \\
\href{https://orcid.org/0009-0008-1906-5097}{ORCID: 0009-0008-1906-5097}}
\vfill
{\Large A thesis submitted in total fulfilment for the \\
degree of Master of Philosophy}
\\[1cm]
{\Large In the \\
School of Mathematics and Statistics \\
\textbf{The University of Melbourne}}
\\[1cm]
{\Large May 2026}
\end{center}
\end{titlepage}

\newpage
\pagenumbering{roman}
\section*{\centering Abstract}
\addcontentsline{toc}{section}{Abstract}
\begin{center}
{\Large The University of Melbourne \\
School of Mathematics and Statistics \\[1cm]
by Ali Khalili Samani \\
\href{https://orcid.org/0009-0008-1906-5097}{ORCID: 0009-0008-1906-5097}} \\[1.5cm]
\end{center}
We introduce the categorical Hopf map as a categorical principal bundle over the two-dimensional sphere with fibre the categorical circle $\mathcal{U}(1)$ of \cite{MR3764535}. We investigate its connection to the Hopf map. We present a factorisation of the categorical Hopf map through the Hopf map and the basic bundle gerbe over the three-dimensional sphere. We discuss three equivalent constructions for the basic bundle gerbe over the three-dimensional sphere. Finally, we conjecture that the categorical group $String(3)$ is equivalent to the categorical group of symmetries of the categorical Hopf map.

\newpage
\section*{\centering{\huge Declaration of authorship}}
\addcontentsline{toc}{section}{Declaration of authorship}
{\large I, Ali Khalili Samani, declare that this thesis titled 'Categorical Hopf map' and the work presented in it are my own. I confirm that \\
$\bullet$ this thesis comprises only my original work towards a Master of Philosophy except where indicated. \\
$\bullet$ acknowledgement has been made in the text to all other material used. \\
$\bullet$ this thesis is fewer than 50,000 words exclusive of table of contents, introduction, figures, and appendices. \\[1.5cm]
Signature: \\
Date:}

\newpage
\section*{\centering \huge Acknowledgements}
\addcontentsline{toc}{section}{Acknowledgements}
I would like to express my sincere gratitude to my supervisor, Nora Ganter, for her continuous support, guidance, and encouragement throughout this project. I am also grateful to my committee members, Christian Haesemeyer and David Ridout, for their support. I particularly thank Christian for the insightful meetings.

I also wish to thank Daniel Murfet, my former committee chair, for his support during the earlier stages of my work. I am thankful for the faculty members and friends in the School of Mathematics and Statistics, particularly Arun Ram, with whom I had stimulating and meaningful discussions.

I gratefully acknowledge the support of the University of Melbourne Graduate Research Scholarship, which made this work possible.

Finally, I extend my heartfelt thanks to my parents and my wife Fatemeh, whose constant support, patience, and encouragement have been invaluable throughout my academic journey.

\newpage
% Table of contents
\tableofcontents

\newpage
\section{Introduction}
Categorical principal bundles generalise the notion of both principal bundles and bundle gerbes. For a Lie group $G$ and a manifold $M$, principal $G$-bundles over $M$ can be classified by homotopy classes of maps $M \longrightarrow BG$, where $BG$ denotes the classifying space of $G$. Equivalently, principal $G$-bundles over $M$ can be understood through the \v{C}ech cohomology group $\check{H}^1(M,G)$. A classical example is the Hopf map $S^3 \longrightarrow S^2$, introduced by Hopf in \cite{MR1512691}, which represents the first known non-trivial map between spheres of different dimensions. The Hopf map remains a fundamental example of a non-trivial principal $U(1)$-bundle, and plays a central role in topology, geometry, and mathematical physics. \\
\indent
At a higher categorical level, one encounters bundle gerbes, which provide geometric models classified for degree-three cohomology classes. Specifically bundle gerbes over $M$ are classified by $H^3(M;\mathbb{Z})$, with the Dixmier-Douady class serving as their primary invariant. Bundle gerbes were introduced by Murray \cite{MR1669206} as a geometric model for degree-three cohomology. Relatedly, Brylinski \cite{MR1197353} developed the notion of gerbes in the context of Deligne cohomology, providing a sheaf-theoretic and categorical perspective on similar objects. \\
\indent
Such structures are especially relevant in the study of topological groups and their representations, which have far-reaching applications in geometry, topology, and quantum field theory (see for example \cite{MR198828}). Among these groups, the Lie groups $SO(n)$ and $Spin(n)$, along with the infinite-dimensional group $String(n)$ introduced by Stolz in \cite{MR1380455} play a pivotal role. They arise as successive covering groups in the Whitehead tower associated with the orthogonal group $O(n)$:
\[
\begin{tikzcd}
\cdots \arrow[r] & String(n) \arrow[r] & Spin(n) \arrow[r] & SO(n) \arrow[r] & O(n).
\end{tikzcd}
\]
Here, $Spin(n)$ is the double cover of $SO(n)$, and $String(n)$ is the $3$-connected cover of $Spin(n)$. \\ 
\indent
A significant development arose when mathematicians and physicists began to investigate the geometric and topological properties of manifolds whose frame bundle admits a lift of its structure group from $SO(n)$ to $Spin(n)$ or further to $String(n)$. Such refinements are intimately related to topics such as spin cobordism, the Witten genus, and generalised cohomology theories. For a comprehensive historical overview, see Sati and Schreiber \cite{MR2588823}.\\
\indent
Traditionally, $String(n)$ has been constructed as an extension of $Spin(n)$ by the Eilenberg-MacLane space $K(\mathbb{Z},2)$:
\[
0 \longrightarrow K(\mathbb{Z},2) \longrightarrow String(n) \longrightarrow Spin(n) \longrightarrow 0,
\]
yielding an infinite-dimensional topological group. In contrast, Schommer-Pries \cite{MR2800361} developed a finite-dimensional categorical model of $String(n)$ by realising it as a non-trivial extension of $Spin(n)$ by the categorical group $\mathbb{B}U(1)$:
\[
0 \longrightarrow \mathbb{B}U(1) \longrightarrow String(n) \longrightarrow Spin(n) \longrightarrow 0.
\]
This construction yields a model for $String(n)$ as a categorical group. It does not, however , offer an explicit description of $String(n)$ as symmetries of some mathematical objects. Related perspective may be found in the work of Fiorenza, Rogers , Schreiber \cite[Section 4.1]{MR3535115}, as well as in the work of Bunk, Müller,, Szabo \cite{MR4268834}. \\
\indent
In this thesis, we focus on the case $n = 3$. We expect that the categorical Lie group $String(3)$ will play a foundational role in the representation theory of categorical groups, much like $SU(2)$, or equivalently $Spin(3)$, serves as a cornerstone in the representation theory of compact Lie groups. \\
\indent
Our approach is motivated by the work of Bartels \cite{MR2709030}, who introduced the concept of categorical principal bundles. Building on his framework, together with the formalism for constructing categorical bundles developed in \cite{MR2805195}, we construct the categorical Hopf map, a categorical principal bundle over $S^2$ whose fibre is the categorical circle $\mathcal{U}(1)$ as defined in \cite{MR3764535}. Furthermore, we conjecture a construction of $String(3)$ that is both explicit and comparatively elementary. \\
The structure of this work is as follows: \\
$\bullet$ Section $2$:  We begin by recalling three classical perspectives on the Hopf map. We then review the notion of an Ehresmann connection on a principal bundle, which we use implicitly to identify the group of symmetries of the Hopf map with the group $Spin^c(3)$ in section $6$. The remainder of this section collects the categorical preliminaries; we review the definition of categorical bundles and the notion of a connection on categorical bundles. \\
$\bullet$ Section $3$: We revisit the definition of bundle gerbes and discuss the correspondence between bundle gerbes and principal $\mathbb{B}U(1)$-bundles. We also present an equivalent construction of the basic bundle gerbe over $S^3$. The other equivalent construction is given in section $5$. \\
$\bullet$ Section $4$:  We develop tools for categorifying the transition function of the Hopf map. \\
$\bullet$ Section $5$: We construct the categorical Hopf map in detail. We describe a factorisation of the categorical Hopf map through the classical Hopf map. \\
$\bullet$ Section $6$: We conjecture how $String(3)$ can be realised as the categorical group of symmetries of the categorical Hopf map. \\
$\bullet$ Appendix: In this section, we provide supporting calculations. We also construct explicit open covers of $S^2$ adapted to some finite subgroups of $SO(3)$. Specifically, we use a four-open cover compatible with the action of the tetrahedral group on $S^2$, while in Section $5$ we employed a six-open cover compatible with the action of the octahedral group on $S^2$.
\begin{conv}
Throughout, we adopt the following conventions: \\
$\bullet$ $U(1)$ denotes the group of $1 \times 1$ unitary matrices, i.e., the complex unit circle. \\
$\bullet$ We denote by $\left\{\mathbbm{1}\right\}$ the trivial group consisting of a single element $\mathbbm{1}$. \\
$\bullet$ $S^n$ denotes the unit $n$-sphere for $n \in \mathbb{Z}^+$. \\
$\bullet$ We use the term "categorical" rather than the prefix "$2$". For example, we write categorical group or categorical space in place of weak $2$-group or $2$-space. \\
$\bullet$ When expressing a point $(x,y,z,w) \in S^3$ as a unit quaternion, we write it as
\[
x\mathsf{i}+y\mathsf{j}+z\mathsf{k}+w.
\]
Accordingly, a point $(X,Y,Z) \in S^2$ is written as
\[
X\mathsf{i}+Y\mathsf{j}+Z\mathsf{k}.
\]
\end{conv}

\newpage
\pagenumbering{arabic}
\section{Background}
In this section, we revisit three perspectives on the Hopf map $$\eta:S^3 \longrightarrow S^2,$$
For detailed discussions and the equivalences between these constructions, see \cite[on page 15]{MR3012377}. A key feature of the Hopf map is that it defines a principal $U(1)$-bundle over $S^2$.\\
$1)$ It is well known that $S^2$ is homeomorphic to the projective line $\mathbb{CP}^1$ via the composition of the maps
\[
f_1:S^2 \longrightarrow \mathbb{C} \cup \left\{\infty\right\}
\]
and
\[
f_2:\mathbb{C} \cup \left\{\infty\right\} \longrightarrow \mathbb{CP}^1,
\]
where
\[
f_1(X,Y,Z) = \frac{X}{1-Z} + \frac{Y}{1-Z}\mathsf{i}
\]
for $(X,Y,Z) \neq (0,0,1)$ and $f_1(0,0,1)=\infty$, and
\[
f_2(z) = [z:1]
\]
for $z \neq \infty$ and $f_2(\infty) = [1:0]$. We define $f = f_2 \circ f_1$. \\
We identify $S^3$ with
\[
\left\{(z_1,z_2) \in \mathbb{C}^2 ~\big\vert~ \lvert z_1 \rvert^2 + \lvert z_2 \rvert^2=1\right\}.
\]
In complex coordinates, the Hopf map is given by
\[
\eta:S^3 \longrightarrow S^2 \simeq \mathbb{CP}^1
\]
which sends a point $(z_1,z_2) \in S^3$ to its projective class $[z_1:z_2] \in \mathbb{CP}^1$. \\
$2)$ To construct $\eta$ using local trivialisation data, consider the cover
$\left\{V_N,V_S\right\}$ of $S^2$, where
\begin{align*}
V_N & = \left\{[z_1:z_2] \in S^2 ~ \big\vert z_1 \neq 0\right\} \\
V_S & = \left\{[z_1:z_2] \in S^2 ~ \big\vert z_2 \neq 0\right\}
\end{align*}
On the overlap $V_N \cap V_S$, the transition function is defined by the composition
\[
V_N \cap V_S \longrightarrow \mathbb{C}^\ast \longrightarrow U(1),
\]
which sends $[z_1:z_2] \in V_N \cap V_S$ first to $\frac{z_1}{z_2} \in \mathbb{C}^\ast$, and then to $\frac{z_1}{z_2} \cdot \big\vert \frac{z_2}{z_1}\big\vert \in U(1)$. Following \cite[on page 583]{MR2805195}, one may construct the total space of the principal bundle associated with this transition data. \\
The Hopf map corresponds to the generator of
\[
H^2(S^2;\mathbb{Z}) \cong \mathbb{Z}.
\]
Throughout this thesis, we adopt the positive generator $1$. \\
$3)$ In real coordinates, the Hopf map is given by\footnote{See Remark \ref{kernelofhopf}.}
\begin{align*}
\eta:S^3 \subset \mathbb{R}^4 &\longrightarrow S^2 \subset \mathbb{R}^3 \\
(x,y,z,w) &\longmapsto \left(2(xz+yw),2(yz-xw),-x^2-y^2+z^2+w^2\right).
\end{align*}
In Construction \ref{symmofhopf}, we require the definition of a connection on a principal bundle. For this reason, we recall the definition below:
\begin{defn}\cite{MR42768}\cite[Definition 6.1 on page 254]{MR1698234}\label{Ehresconn}
Let $G$ be a Lie group, and let $\pi:P \longrightarrow M$ be a principal $G$-bundle over a smooth manifold $M$. An \textbf{Ehresmann connection} $\zeta$ on $\pi$ assigns to every point $p$ in $P$ a vector subspace $H_p P$ of $T_p P$ such that: \\
$(1)$ The tangent space $T_p P$ decomposes as $$T_p P = H_p P \oplus V_p P$$
where $V_p P$ is the kernel of the map $$\pi_*:T_p P \longrightarrow T_{\pi(p)} M.$$
$(2)$ For every $g$ in $G$, if we write $$g:P \longrightarrow P,$$ for the action of $G$ on $P$, then $$g^*(H_p P) = H_{g \cdot p} P.$$
The subspaces $H_{p}P$ and $V_{p}P$ of $T_{p}P$ are referred to as the \textbf{horizontal} and \textbf{vertical} subspaces, respectively.
\end{defn}
Below, we provide an interpretation of an Ehresmann connection which we mainly use throughout this thesis.
\begin{rem}\cite{MR42768}\cite[Comments on page 254]{MR1698234}\label{equivconnpri}
An Ehresmann connection on a principal $G$-bundle $\pi:P \longrightarrow M$ can be considered as a $\mathfrak{g}$-valued one-form on $P$, where $\mathfrak{g}$ is the Lie algebra of $G$. For $p$ in $P$, define the map
\begin{align*}
\xi:\mathfrak{g} & \longrightarrow T_p P \\
A & \mapsto X^A_p
\end{align*}
defined by $$A \mapsto X^A_p(f)=\frac{d}{dt}f\bigl(p \cdot exp(tA)\bigr) \big\vert_{t = 0}$$
for $f$ in $C^\infty(P)$. By the property $(1)$ in Definition \ref{Ehresconn}, for each vector field $X$ in $\Gamma(TP)$ we get two vector fields $\mathsf{Hor}(X) \in $ and $\mathsf{Ver}(X)$ in $\Gamma(TP)$, where $\mathsf{Hor}(X)_p \in H_p P$ and $\mathsf{Ver}(X)_p \in V_p P$, for $p$ in $P$. Define the one-form $\tau$ on $P$ to be given by
$$\tau_p(X_p) = \xi^{-1}\bigl(\mathsf{Ver}(X_p)\bigr)$$
for $X_p$ in $T_p P$. So, we have the following properties: \\
$(1)$ $\tau_p(X^A_p) = A$ for $p$ in $P$ and $A$ in $\mathfrak{g}$, \\
$(2)$ $g^*(\tau) = \mathsf{Ad}_{g^{-1}} \tau$, \\
$(3)$ the tangent vector $X_p$ is in $H_p P$ if and only if $\tau_p(X_p) = 0$.
\end{rem}
\begin{defn}\cite{MR116360}\cite[Section 1.1 on page 206]{MR1950948}\label{groupoid1}
A \textbf{groupoid} $G$ is a category in which every morphism is invertible. We denote the set of objects of $G$ and the set of morphisms of $G$ by $G_0$ and $G_1$, respectively. \\ In \cite{MR1950948}, Moerdijk uses $G$ for denoting a groupoid and $G_1 \rightrightarrows G_0$ for denoting its source and target maps. Throughout this thesis, we use $[G_1 \rightrightarrows G_0]$ to denote a groupoid $G$ with objects $G_0$ and morphisms $G_1$.
\end{defn}
\begin{nota}\label{groupoid}
\cite[Section 1.1 on page 206]{MR1950948} Associated to a groupoid $[G_1 \rightrightarrows G_0]$, we have the following five structure maps: \\
(\rNum{1}) source map $\mathsf{s}:G_1 \longrightarrow G_0$ which sends a morphism to its source, \\
(\rNum{2}) target map $\mathsf{t}:G_1 \longrightarrow G_0$ which sends a morphism to its target, \\
(\rNum{3}) multiplication (composition) map $\mathsf{mult}:G_1 \times_{G_0} G_1 \longrightarrow G_1$ which is the composition of composable morphisms in $G$, where $G_1 \times_{G_0} G_1$ is given by the following pullback square:
\[
\begin{tikzcd}
G_1 \times_{G_0} G_1 \arrow[r,"pr_2"] \arrow[d,"pr_1"'] & G_1 \arrow[d,"t"] \\
G_1 \arrow[r,"s"'] & G_0
\end{tikzcd}
\]
(\rNum{4}) unit map $\mathsf{unit}:G_0 \longrightarrow G_1$ which sends an object $x$ to its identity morphism $\mathsf{1}_x$, \\
(\rNum{5}) inverse map $\mathsf{inv}:G_1 \longrightarrow G_1$ which sends a morphism $h$ to its inverse $h^{-1}$.
\end{nota}
\begin{defn}\cite{MR217087}\cite[Definition 2.2 on page 429]{MR2068521}
A \textbf{weak categorical group} $\mathcal{G}$ is a weak monoidal category $(\mathcal{M},\otimes,a,1,l,r)$ in which every morphism is invertible and each object $x$ is weakly invertible, that is, there exists an object $\bar{x}$ together with natural isomorphisms $x \otimes \bar{x} \overset{\cong}\longrightarrow 1$ and $\bar{x} \otimes x \overset{\cong}\longrightarrow 1$. Here, $1$ denotes the unit object, $a$ the associator, and $l$ and $r$ the left and right unitors, respectively. \\
If all natural isomorphisms are replaced with identities in the definition of weak monoidal categories and weak categorical groups, one obtains \textbf{strict monoidal categories} and \textbf{strict categorical groups} respectively. Using Definition \ref{groupoid1}, a strict categorical group may be defined as a strict monoidal groupoid. 
\end{defn}
\begin{defn}\cite{MR30760}\cite[on page 8]{MR2597732}
A \textbf{crossed module} is a $4$-tuple $(G,H,\alpha,\beta)$, where \\
$(1)$ $G$ and $H$ are groups, \\
$(2)$ $\beta:H \longrightarrow G$ is a group homomorphism, and \\
$(3)$ $\alpha:G \times H \longrightarrow H$ denotes an action of $G$ on $H$.
These data are required to satisfy the following conditions: \\
(\rNum{1}) (Equivariance) $\beta\bigl(\alpha(g,h)\bigr) = g \cdot \beta(h) \cdot g^{-1}$, \\
(\rNum{2}) (Peiffer identity) $\alpha\bigl(\beta(h),h^\prime\bigr) = h \cdot h^\prime \cdot h^{-1}$.
\end{defn}
\begin{rem}
The notion of a crossed module extends naturally to topological crossed modules, in which the groups involved are topological groups and the associated maps are continuous.
\end{rem}
\begin{rem}\label{2grpcross}
\cite{MR419643}
\cite[Remark 2.4 on page 574]{MR2805195} Given a strict categorical group $\mathcal{G}$, one can associate a crossed module $(G,H,\alpha,\beta)$ defined as follows: \\
$(1)$ $G \coloneq \mathsf{Obj}(\mathcal{G})$, \\
$(2)$ $H \coloneq$ the kernel of the source map $s:\mathsf{Mor}(\mathcal{G}) \longrightarrow \mathsf{Obj}(\mathcal{G})$, \\
$(3)$ $\alpha(g,h) \coloneq id_g \otimes h \otimes id_{g^{-1}}$, \\
$(4)$ $\beta$ is the restriction of the target map $\mathsf{t}:\mathsf{Mor}(\mathcal{G}) \longrightarrow \mathsf{Obj}(\mathcal{G})$ to $H$. \\
Conversely, from a crossed module $(G,H,\alpha,\beta)$ one can construct a strict categorical group $\mathcal{G}$ with: \\
(\rNum{1}) $\mathsf{Obj}(\mathcal{G}) \coloneq G$, \\
(\rNum{2}) $\mathsf{Mor}(\mathcal{G}) \coloneq H \rtimes G$, where the source and target of a morphism $(h,g)$ are given by
\begin{align*}
\mathsf{s}(h,g) & = g \\
\mathsf{t}(h,g) & = \beta(h) \cdot g.
\end{align*}
The composition of two composable arrows
\[
\begin{tikzcd}[row sep=huge,column sep=huge]
g \arrow[r,"{(h,g)}"] & \beta(h) \cdot g
\end{tikzcd}
\]
and
\[
\begin{tikzcd}[row sep=huge,column sep=huge]
\beta(h) \cdot g \arrow[r,"{\bigl(h^\prime,\beta(h) \cdot g\bigr)}"] & \beta(h^\prime \cdot h) \cdot g
\end{tikzcd}
\]
is given by:
\[
\begin{tikzcd}[row sep=huge,column sep=huge]
g \arrow[r,"{(h^\prime \cdot h,g)}"] & \beta(h^\prime \cdot h)  \cdot g.
\end{tikzcd}
\]
(\rNum{3}) Tensor multiplication of two objects $g$ and $g^\prime$ is defined by
$$g \otimes g^\prime \coloneq g \cdot g^\prime,$$
and the tensor multiplication of two morphisms $(h,g)$ and $(h^\prime,g^\prime)$ is given by
$$(h,g) \otimes (h^\prime,g^\prime) \coloneq \bigl(h \cdot \alpha(g,h^\prime),g \cdot g^\prime\bigr).$$
\end{rem}
\begin{exa}
In \cite[Section 3 on page 5]{MR3764535}, Ganter introduces the categorical circle $\mathcal{U}(1)$. The strict categorical group $\mathcal{U}(1)$ is associated to the crossed module $(\mathbb{R},\mathbb{Z} \times \mathbb{R}/\mathbb{Z},\alpha,\beta)$, where $\alpha$ is the action of $\mathbb{R}$ on $\mathbb{Z} \times \mathbb{R}/\mathbb{Z}$:
\begin{align*}
\alpha:\mathbb{R} \times (\mathbb{Z} \times \mathbb{R}/\mathbb{Z}) & \longrightarrow \mathbb{Z} \times \mathbb{R}/\mathbb{Z} \\
\bigl(r,(m,[a])\bigr) ~~ & \mapsto (m,[a+rm]),
\end{align*}
and $\beta$ is the group homomorphism
\begin{align*}
\mathbb{Z} \times \mathbb{R}/\mathbb{Z} & \longrightarrow \mathbb{R} \\
(m,[a]) & \mapsto m.
\end{align*}
So, by Remark \ref{2grpcross}, $\mathcal{U}(1)$ has the following data: \\
(\rNum{1}) objects: $\mathbb{R}$, \\
(\rNum{2}) morphisms: $(\mathbb{Z} \times \mathbb{R}/\mathbb{Z}) \rtimes \mathbb{R}$, \\
(\rNum{3}) The source and target of a morphism $\bigl((m,[a]),r\bigr)$ given by:
\begin{align*}
\mathsf{s}\bigl((m,[a]),r\bigr) & = r \\
\mathsf{t}\bigl((m,[a]),r\bigr) & = r + m.
\end{align*}
So, the morphism $\bigl((m,[a]),r\bigr)$ can be depicted as
\[
\begin{tikzcd}
r \arrow[rr,"{\bigl((m,[a]),r\bigr)}"] & & r + m.
\end{tikzcd}
\]
The composition of two composable morphisms
\[
\begin{tikzcd}
r \arrow[rr,"{\bigl((m,[a]),r\bigr)}"] & & r + m
\end{tikzcd}
\]
and 
\[
\begin{tikzcd}
r + m \arrow[rrr,"{\bigl((m^\prime,[a^\prime]),r+m\bigr)}"] & & & r + m + m^\prime
\end{tikzcd}
\]
in $\mathcal{U}(1)$ is given by
\[
\begin{tikzcd}
r \arrow[rrrr,"{\bigl((m+m^\prime,[a+a^\prime]),r\bigr)}"] & & & & r + m + m^\prime.
\end{tikzcd}
\]
(\rNum{5}) The tensor multiplication of two objects $r$ and $r^\prime$ is defined by $r + r^\prime$, and the tensor multiplication of two
morphisms $\bigl((m,[a]),r\bigr)$ and $\bigl((n,[b]),s\bigr)$ is given by
\begin{equation*}
\bigl(\begin{tikzcd}
r \arrow[rr,"{\bigl((m,[a]),r\bigr)}"] & & r + m
\end{tikzcd}\bigr) \otimes
\bigl(\begin{tikzcd}
s \arrow[rr,"{\bigl((n,[b]),s\bigr)}"] & & s + n
\end{tikzcd}\bigr) =
\bigl(\begin{tikzcd}
r + s \arrow[rrrr,"{\bigl((m+n,[a+b+rn]),r+s\bigr)}"] & & & & r + s + m + n
\end{tikzcd}\bigr).
\end{equation*}
\end{exa}
\begin{defn}\cite{MR116360}\cite[Definition 1.4 on page 569]{MR2805195}
A \textbf{categorical space} $\mathcal{M}$ is a category for which $\mathsf{Obj}(\mathcal{M})$ and $\mathsf{Mor}(\mathcal{M})$ are topological spaces, and all structure maps are continuous. We write $\mathcal{M}_0$ and $\mathcal{M}_1$ for $\mathsf{Obj}(\mathcal{M})$ and $\mathsf{Mor}(\mathcal{M})$, respectively.
\end{defn}
\begin{exa}
A topological space $M$ may be regarded as a categorical space in which the objects are elements of $M$, and each object $m$ admits only the identity morphism $m \longrightarrow m$. In this case, $M$ becomes a categorical space $\mathcal{M}$ with $\mathsf{Obj}(\mathcal{M}) = M$ and $\mathsf{Mor}(\mathcal{M}) = M$.
\end{exa}
\begin{defn}\cite{MR116360}\cite[Definition 1.4 on page 569]{MR2805195}
Let $\mathcal{M}$ and $\mathcal{M}^\prime$ be categorical spaces. A \textbf{continuous categorical map} $F:\mathcal{M} \longrightarrow \mathcal{M}^\prime$ is a functor such that $F_0:\mathcal{M}_0 \longrightarrow \mathcal{M}^\prime_0$ and $F_1:\mathcal{M}_1 \longrightarrow \mathcal{M}^\prime_1$ are continuous maps. \\
If there exists a functor $\bar{F}:\mathcal{M}^\prime \longrightarrow \mathcal{M}$ together with continuous natural transformations $F \circ \bar{F} \Longrightarrow id_{\mathcal{M}^\prime}$ and $\bar{F} \circ F \Longrightarrow id_{\mathcal{M}}$, we call $F$ an \textbf{equivalence} of categorical spaces and $\bar{F}$ a \textbf{weak inverse} of $F$.
\end{defn}
\begin{defn}\cite{MR217087}\cite[Definition 1.7 on page 570]{MR2805195}
Let $M$ be a topological space regarded as a categorical space, and let $\mathcal{G}$ be a strict categorical group. The space $M$ is called a \textbf{categorical right $\mathcal{G}$-space} if there exists a continuous categorical map $$\epsilon:M \times \mathcal{G} \longrightarrow M$$ such that the following diagram commutates:
\[
\begin{tikzcd}
& \mathcal{M} \times \mathcal{G} \times \mathcal{G} \arrow[dl,"\epsilon \times id_{\mathcal{G}}"'] \arrow[dr,"id_{\mathcal{M}} \times m_{\mathcal{G}}"] & \\
\mathcal{M} \times \mathcal{G} \arrow[dr,"\epsilon"'] & & \mathcal{M} \times \mathcal{G} \arrow[dl,"\epsilon"] \\
& \mathcal{M} &
\end{tikzcd}
\]
where $m_{\mathcal{G}}$ is the tensor multiplication of $\mathcal{G}$.
\end{defn}
\begin{defn}\label{prin2bun}\cite{MR2709030}\cite[Definition 6.1.5 on page 389]{MR3089401}
Let $\mathcal{G}$ be a strict categorical group, and let $M$ be a smooth manifold. A \textbf{principal $\mathcal{G}$-bundle} over $M$ is a categorical right $\mathcal{G}$-space $\mathcal{P}$ together with a  $\mathcal{G}$-equivariant map
$$\rho:\mathcal{P} \longrightarrow M$$
such that the smooth functor
\[
(pr_1,R):\mathcal{P} \times \mathcal{G} \longrightarrow \mathcal{P} \times_M \mathcal{P}
\]
is a weak equivalence. Here $R$ denotes the right action of $\mathcal{G}$ on $\mathcal{P}$.
\end{defn}
\begin{rem}
As noted in \cite[Remark 6.2.7 on page 391]{MR3089401}, Definition \cite[Definition 1.8 on page 571]{MR2805195} is, in general, restrictive, since it requires local trivialisations to be given by strong equivalences. However, the use of this stronger notion does not affect the essential features of the constructions considered here. In particular, the resulting categorical total space and the action of the categorical group remain unchanged. The distinction lies only in the nature of local trivialisations and, more broadly, in the class of morphisms between categorical spaces: Waldorf’s framework admits anafunctors in place of strict functors. \\
For the class of examples considered in Construction \ref{cathopf}, strong local trivialisations are available, and the framework of \cite{MR2805195} provides a convenient and sufficiently precise setting. Where greater generality is required, we pass to Waldorf’s formulation without altering the underlying geometric or categorical content, but only the level of flexibility in the morphisms and trivialisations.
\end{rem}
\begin{defn}\cite{MR2342821}\cite[Definition 2.12 on page 578]{MR2805195}\label{gvaluedcocycle} Let $X$ be a topological space, and let $\mathcal{G}$ be a strict categorical group. Let $(G,H,\alpha,\beta)$ be the associated crossed module to $\mathcal{G}$. A \textbf{$\mathcal{G}$-valued cocycle} on $X$ is a triple $(W_i,g_{ij},h_{ijk})$ where $\left\{W_i\right\}_{i \in I}$ is an open cover of $X$ and for each $i,j,k \in I$, the maps $$g_{ij}:W_i \cap W_j \longrightarrow G$$ and $$h_{ijk}:W_i \cap W_j \cap W_k \longrightarrow H$$ are smooth. These maps are required to satisfy the following conditions pointwise:
\begin{align}
g_{ii} & = e_G \label{gii} \\
\beta\bigl(h_{ijk}\bigr) \cdot g_{ij} \cdot g_{jk} & = g_{ik} \label{cocycleg} \\
h_{ijj} = h_{jji} & = e_H \label{hijj} \\
h_{ijk} \cdot h_{ikl} & = h_{ijl} \cdot \alpha(g_{ij},h_{jkl}) \label{cocycleh}
\end{align}
\end{defn}
The $2$-category of $\mathcal{G}$-valued cocycles and the $2$-cateogry of $\mathcal{G}$-bundles are equivalent \cite[Sections 2.5.2 and 2.5.3]{MR2709030}. In what follows, we do not need the $2$-categorical structure; instead, we simply recall how a principal $\mathcal{G}$-bundle may be obtained from a $\mathcal{G}$-valued cocycle. For the converse direction, see \cite[Remark 2.14 on page 580]{MR2805195}.
\begin{rem}\cite{MR2342821}\cite[Remark 2.16 on page 583]{MR2805195}\label{objandmorcatbun}
Let $X$ be a topological space , and let $\mathcal{G}$ be a strict categorical group. Let $(G,H,\alpha,\beta)$ be the associated crossed module to $\mathcal{G}$. Given a $\mathcal{G}$-valued cocycle $(W_i,g_{ij},h_{ijk})$ on $X$, one constructs a principal $\mathcal{G}$-bundle $\rho:\mathcal{P} \longrightarrow X$ as follows: \\
(\rNum{1}) The space of objects of $\mathcal{P}$ is
\begin{equation}\label{objs}
\mathsf{Obj}(\mathcal{P}) = \sqcup_i (W_i \times G).
\end{equation}
The space of morphisms of $\mathcal{P}$ is
\begin{equation}\label{mors}
\mathsf{Mor}(\mathcal{P}) = \sqcup_{i,j} \bigl((W_i \cap W_j) \times H \times G\bigr).
\end{equation}
The source and target of a morphism $(w,h,g)$ in $(W_i \cap W_j) \times H \times G$ are given by:
\begin{align}
\mathsf{s}\bigl(w,h,g\bigr) & \coloneq (w,g) \\
\mathsf{t}\bigl(w,h,g\bigr) & \coloneq \bigl(w,g_{ij}(w)^{-1} \cdot \beta(h) \cdot g\bigr)
\end{align}
The composition of two composable arrows
\[
\begin{tikzcd}
\bigl(w,g\bigr) \arrow[rr,"{\bigl(w,h,g\bigr)}"] & & \bigl(w,g^\prime\bigr)
\end{tikzcd}
\]
in $(W_i \cap W_j) \times H \times G$ and
\[
\begin{tikzcd}
\bigl(w,g^\prime\bigr) \arrow[rr,"{(w,h^\prime,g^\prime)}"] & & \bigl(w,g^{\prime\prime}\bigr)
\end{tikzcd}
\]
in $(W_j \cap W_k) \times H \times G$, where
\begin{align*}
g^\prime & = g_{ij}(w)^{-1} \cdot \beta(h) \cdot g, \\
g^{\prime\prime} & = g_{jk}(w)^{-1} \cdot \beta(h^\prime) \cdot g^\prime,
\end{align*}
is defined as
\[
\begin{tikzcd}[row sep=huge,column sep=huge]
\bigl(w,g\bigr) \arrow[rrr,"{\bigl(w,g,h_{ijk} \cdot \alpha(g_{ij},h^\prime) \cdot h\bigr)}"] & & & \bigl(w,g^{\prime\prime}\bigr)
\end{tikzcd}
\]
Here, we need to note that
\[
g^{\prime\prime} = g_{ik}^{-1}(w) \cdot \beta\bigl(h_{ijk} \cdot \alpha(g_{ij},h^\prime) \cdot h\bigr) \cdot g
\]
by Equation (\ref{cocycleg}) in Definition \ref{gvaluedcocycle}.  \\
(\rNum{2}) The action of $\mathcal{G}$ on $\mathcal{P}$ is defined by:
\[
(w,g) \cdot g^\prime \coloneq (w,g \cdot g^\prime)
\]
on objects, and by
\[
(w,h,g) \cdot (h^\prime,g^\prime) \coloneq \bigr(w,h \cdot \alpha(g,h^\prime),g \cdot g^\prime\bigr)
\]
on morphisms. \\
(\rNum{3}) The $\mathcal{G}$-equivalence $\bar{\phi}_i:W_i \times \mathcal{G} \longrightarrow \rho^{-1}(W_i)$ can be defined using inclusion and $\phi_i:\rho^{-1}(W_i) \longrightarrow W_i \times \mathcal{G}$ is
given by:
\[
\phi_i(w,g) = (w,g_{ij} \cdot g)
\]
on objects and
\[
\phi_i(w,h,g) = \Bigl(w,\bigl(h_{ijk} \cdot \alpha(g_{ij},h),g_{ij} \cdot g\bigr)\Bigr)
\]
on morphisms.
\end{rem}
\begin{nota}\label{objmornewform}
At times, we express the disjoint unions in Equations (\ref{objs}) and (\ref{mors}) in the following form:
\begin{align*}
\mathsf{Obj}(\mathcal{P}) & = \cup_i \bigl(\left\{i\right\} \times W_i \times G\bigr), \\
\mathsf{Mor}(\mathcal{P}) & = \cup_{i,j} \bigl(\left\{(i,j)\right\} \times (W_i \cap W_j) \times H \times G\bigr).
\end{align*}
Thus, an object of $\mathcal{P}$ may be represented as $(i,w,g)$, and a morphism of $\mathcal{P}$ may be expressed as $(i,j,w,h,g)$.
\end{nota}
In the following, we recall the definition of a connection on a categorical bundle. We employ this notion in formulating the conjectural interpretation of the categorical group $String(3)$ as categorical group of symmetries of the categorical Hopf map in Section $6$:
\begin{defn}\cite[Definition 4.1.1 on page 15]{MR3894086}
Let $\mathcal{P}$ be a categorical space, and let $\Gamma$ be a strict categorical Lie group. Denote by $(G,H,\alpha,\beta)$ the crossed module associated to $\Gamma$. Considering the categorical Lie algebra $(\mathfrak{g},\mathfrak{h},\alpha_\ast,\beta_\ast)$ of $\Gamma$ as defined in \cite[Section 2.1 on page 3]{MR3894086}, a \textbf{$\gamma$-valued $k$-form} $\Psi$ on $\mathcal{P}$ consists of: \\
$\bullet$ a $\mathfrak{g}$-valued $k$-form $\Psi^a \in \Omega^k(\mathcal{P}_0,\mathfrak{g})$, \\
$\bullet$ a $\mathfrak{h}$-valued $k$-from $\Psi^b \in \Omega^k(\mathcal{P}_1,\mathfrak{h})$, \\
$\bullet$ a $\mathfrak{h}$-valued $(k+1)$-form $\Psi^c \in \Omega^{k+1}(\mathcal{P}_0,\mathfrak{h})$. \\
The differential $D:\Omega^k(\mathcal{P},\gamma) \longrightarrow \Omega^{k+1}(\mathcal{P},\gamma)$ maps $\Psi \in \Omega^k(\mathcal{P},\gamma)$ to $D(\Psi)$ with components: \\
$\bullet$ $(D\Psi)^a = d\Psi^a - (-1)^k \mathsf{t}_\ast(\Psi^c)$, \\
$\bullet$ $(D\Psi)^b = d\Psi^b - (-1)^k \Delta(\Psi^c)$, \\
$\bullet$ $(D\Psi)^c = d\Psi^c$. \\
The map $\Delta:\Omega^k(\mathcal{P}_0) \longrightarrow \Omega^k(\mathcal{P}_1)$ is defined as in \cite[Section 4.1 on page 14]{MR3894086}.
\end{defn}
\begin{defn}\cite[Definition 5.1.1 on page 18]{MR3894086}\label{connectiononcatbundle}
Let $\Gamma$ be a categorical Lie group with its categorical Lie algebra $\gamma$, and let $\rho:\mathcal{P} \longrightarrow M$ be a principal $\Gamma$-bundle. Consider the action $$R:\mathcal{P} \times \Gamma \longrightarrow \mathcal{P}$$ of $\Gamma$ on $\mathcal{P}$, and let $(G,H,\alpha,\beta)$ be the crossed module associated to $\Gamma$. \\
A connection $\Omega$ on $\rho$ s defined to be a $\gamma$-valued $1$-form on $\mathcal{P}$ satisfying:
\begin{equation}\label{connontwobun}
R^\ast \Omega = \mathsf{Ad}^{-1}_{pr_\Gamma}(pr^\ast_\mathcal{P} \Omega) + pr^\ast_\Gamma \theta,
\end{equation}
where $pr_\Gamma$ and $pr_\mathcal{P}$ are the projection maps
\[
pr_\Gamma:\mathcal{P} \times \Gamma \longrightarrow \Gamma, ~~~~~
pr_\mathcal{P}:\mathcal{P} \times \Gamma \longrightarrow \mathcal{P}.
\]
Here, the functor $\mathsf{Ad}_{pr_\Gamma}$ is defined as in \cite[Section 4.2 on page 16]{MR3894086}, and $\theta$ denotes the Maurer-Cartan form on $\Gamma$. The curvature $\mathsf{curv}(\Omega)$ of $\Omega$ is a $2$-form in $\Omega^2(\mathcal{P},\gamma)$ given as in \cite[Definition 5.1.3 on page 19]{MR3894086}. \\
Since we will not write down the explicit connection form for the categorical Hopf map, we instead refer the reader to \cite[Section 5.1 on page 19]{MR3894086} to revisit the expansion of Equation (\ref{connontwobun}) in terms of components $\Omega^a$, $\Omega^b$ and $\Omega^c$.
\end{defn}
\begin{rem}
There are two extreme cases for a connection on a categorical bundle: \\
$(1)$ $\Gamma = G$ where $G$ is an ordinary group, \\
$(2)$ $\Gamma = \mathbb{B}U(1)$. \\
In case ($1$), one recovers the usual notion of a connection on a principal bundle, while in case ($2$), the resulting structure is a connection on a bundle gerbe. For further details, see \cite[Examples 5.1.9 and 5.1.10 on page 20]{MR3894086}.
\end{rem}

\section{Equivalent constructions for the basic bundle gerbe over \texorpdfstring{$S^3$}{S3}}
In this section, we recall the definition of a bundle gerbe and its relation to the categorical group $\mathbb{B}U(1)$. The strict categorical group $\mathbb{B}U(1)$ is defined as follows: \\
$\bullet$ its set of objects is the trivial group $\left\{\mathbbm{1}\right\}$, and \\ 
$\bullet$ its set of morphisms is $U(1)$. \\
The associated crossed module for $\mathbb{B}U(1)$ is given by 
$$(\left\{\mathbbm{1}\right\},\mathbb{R}/\mathbb{Z},\alpha,\beta),$$
where both $\alpha$ and $\beta$ are trivial choices. \\
There is a well-known correspondence between bundle gerbes and principal $\mathbb{B}U(1)$-bundles which we recall in Remark \ref{butobunandconv}. More precisely, there exists an equivalence between the $2$-category of bundle gerbes and the $2$-category of principal $\mathbb{B}U(1)$-bundles. As we do not use the $2$-categorical structure explicitly in this thesis, we refer the interested reader to \cite[Section 7 on page 392]{MR3089401}. \\
Bundles gerbes over a manifold $M$ are classified, up to stable isomorphism, by degree-three integral cohomology $H^3(M;\mathbb{Z})$. we review this classification in Remark \ref{clager}. The class in $H^3(M;\mathbb{Z})$ associated to a bundle gerbe is known as its \textbf{Dixmier-Douady class}.  \\
An important example for our purposes is the basic bundle gerbe over $S^3$, described in Example \ref{murbun}. This gerbe is classified by the positive generator of $$H^3(S^3;\mathbb{Z}) \cong \mathbb{Z}.$$
In Example \ref{ourbun} and Observation \ref{mapH}, we present two equivalent constructions of this gerbe, and we establish their equivalence in Observations \ref{equiofbungerbes} and \ref{mapH}.
\begin{rem}\cite[page 68]{MR1197353}\label{linebuncirbun}
There is a natural correspondence between complex line bundles and principal $U(1)$-bundles. Accordingly, we use these notions interchangeably without any confusion.
\end{rem}
\begin{nota}
Let $M$ and $Y$ be topological spaces, and let $\pi:Y \longrightarrow M$ be a continuous map. When no ambiguity arises, we write $Y^{[k]}$ for the $k$-fold fibre product 
$$\underbrace{Y \times_M Y \times_M \cdots \times_M Y}_{k ~ \text{times}}.$$
We also denote by $\pi_{ij}:Y^{[3]} \longrightarrow Y^{[2]}$ the projection onto the $i$-th and $j$-th factors.
\end{nota}
\begin{defn}\cite[Definition 12.1 on page 243]{MR2681698}\label{gerbe}
Let $M$ be a smooth manifold. A \textbf{bundle gerbe over $M$} is a triple $(L,Y,M)$ consisting of: \\
$\bullet$ a surjective submersion $\pi:Y \longrightarrow M$, \\
$\bullet$ a line bundle $\lambda:L \longrightarrow Y^{[2]}$, and \\
$\bullet$ an isomorphism
$$m:\pi_{12}^*L \otimes \pi_{23}^*L \longrightarrow \pi_{13}^*L$$
called the \textbf{bundle gerbe product} such that the diagram
\begin{equation}\label{gerbeprod}
\begin{tikzcd}[row sep=huge,column sep=huge]
L_{(y_1,y_2)} \otimes L_{(y_2,y_3)} \otimes L_{(y_3,y_4)} \arrow[r,"m \otimes id"] \arrow[d,"id \otimes m"'] & L_{(y_1,y_3)} \otimes L_{(y_3,y_4)} \arrow[d,"m"] \\
L_{(y_1,y_2)} \otimes L_{(y_2,y_4)} \arrow[r,"m"'] & L_{(y_1,y_4)}
\end{tikzcd}
\end{equation}
commutes for each $(y_1,y_2,y_3,y_4) \in Y^{[4]}$. Here, $L_{(y_i,y_j)}$ denotes the fibre of the line bundle $L$ over $(y_j,y_j) \in Y^{[2]}$.
\end{defn}
\begin{exa}[Construction one for the basic bundle gerbe over $S^3$]\cite[Example 12.8 on page 244]{MR2681698}\label{murbun}
Let $$W = S^3 \setminus \left\{(0,0,0,-1)\right\}$$ and $$W^\prime = S^3 \setminus \left\{(0,0,0,1)\right\}$$ be the neighbourhoods of points $(0,0,0,1)$ and $(0,0,0,-1)$, respectively. The intersection $W \cap W^\prime$ is homotopy equivalent to $S^2$, with homotopy equivalence chosen to be given by the canonical maps
\begin{align*}
f &: W \cap W' \longrightarrow S^2, \qquad
(x,y,z,w) \longmapsto
\left(
\frac{x}{\sqrt{x^2+y^2+z^2}},
\frac{y}{\sqrt{x^2+y^2+z^2}},
\frac{z}{\sqrt{x^2+y^2+z^2}}
\right), \\
g &: S^2 \longrightarrow W \cap W', \qquad
(X,Y,Z) \longmapsto (X,Y,Z,0).
\end{align*}
It follows that complex line bundles over $S^2$ are in bijective correspondence with complex line bundles over $W \cap W^\prime$, since line bundles are classified up to isomorphism by homotopy classes of maps into the classifying space. Equivalently, every complex line bundle over $W \cap W^\prime$ is isomorphic to the pullback along the map $f$ of a complex line bundle over $S^2$. \\
The bundle gerbe $(L,Y,S^3)$ is specified by the following data: \\
(\rNum{1}) $Y = W \sqcup W^\prime$, \\
(\rNum{2}) A line bundle $L \longrightarrow Y^{[2]} = W \sqcup (W \cap W^\prime) \sqcup (W^\prime \cap W) \sqcup W^\prime$. \\
Over $W$ and $W^\prime$, the line bundles $L_{W}$ and $L_{W^\prime}$ are trivial since these spaces are contractible. Over $W \cap W^\prime$, we take the line bundle $L_{W \cap W^\prime}$ which is the pullback of the line bundle associated with the Hopf map $\eta:S^3 \longrightarrow S^2$ along the canonical map $f$. By our convention set in the background section, the Hopf map represents the class $1$ in $H^2(S^2;\mathbb{Z})$; hence the line bundle $L_{W \cap W^\prime}$ is non-trivial. \\
We note the following isomorphisms:
\begin{align*}
\check{H}^1\bigl(S^2,\underline{U(1)}\bigr) & \cong \check{H}^2\bigl(S^2,\underline{\mathbb{Z}}\bigr) \\
& \cong H^2(S^2;\mathbb{Z}) \\
& \cong H^2(W \cap W^\prime;\mathbb{Z}),
\end{align*}
arising from the exact sequence of constant sheaves
\[
\begin{tikzcd}
1 \arrow[r] & \underline{\mathbb{Z}} \arrow[r] & \underline{\mathbb{R}} \arrow[r] & \underline{U(1)} \arrow[r] & 1
\end{tikzcd},
\]
where the second isomorphism is given by \cite[Theorem 3.1.1 on page 184]{MR1481706}. \\
(\rNum{3}) The bundle gerbe product
\[
m:\pi_{12}^*L \otimes \pi_{23}^*L \longrightarrow \pi_{13}^*L
\]
is defined using the equality $Y^{[3]} = Y^{[2]} \times_Y Y^{[2]}$. In $Y^{[3]}$, we have the following intersections:
\[
W \cap W \cap W, \qquad W \cap W \cap W^\prime, \qquad W \cap W^\prime \cap W, \qquad W \cap W^\prime \cap W^\prime,\] \[W^\prime \cap W \cap W, \qquad W^\prime \cap W \cap W^\prime, \qquad W^\prime \cap W^\prime \cap W, \qquad W^\prime \cap W^\prime \cap W^\prime.
\]
To illustrate, we consider the intersections $W \cap W \cap W$ and $W \cap W^\prime \cap W$. The construction is similar for
other intersections. For $W \cap W \cap W$ in $Y^{[3]}$, we have three maps
\begin{align*}
pr_{12} & : W \cap W \cap W \longrightarrow W \cap W\\
pr_{23} & : W \cap W \cap W \longrightarrow W \cap W \\
pr_{13} & : W \cap W \cap W \longrightarrow W \cap W
\end{align*}
with values in $Y^{[2]}$. We notice that we are viewing $W \cap W \cap W$ as
\[
(W \cap W) \cap (W \cap W)
\]
In this case, the bundle gerbe product reduces to multiplication in $\mathbb{C}$, as the fibres consist of complex numbers; in other words, this yields three trivial line bundles
\begin{align*}
pr_{12}^*L_{W} & \longrightarrow W \cap W \cap W, \\
pr_{23}^*L_{W} & \longrightarrow W \cap W \cap W, \\
pr_{13}^*L_{W} & \longrightarrow W \cap W \cap W.
\end{align*}
For $W \cap W^\prime \cap W$ in $Y^{[3]}$, we have three maps
\begin{align*}
pr_{12} & : W \cap W^\prime \cap W \longrightarrow W \cap W^\prime \\
pr_{23} & : W \cap W^\prime \cap W \longrightarrow W^\prime \cap W \\
pr_{13} & : W \cap W^\prime \cap W \longrightarrow W \cap W
\end{align*}
with values in $Y^{[2]}$. Considering $L^{-1}_{W \cap W^\prime}$, the line bundle over $W \cap W^\prime$ having transition function inverse to the transition function of $L_{W \cap W^\prime}$, we get three line bundles
\begin{align*}
pr_{12}^*L_{W \cap W^\prime} & \longrightarrow W \cap W^\prime \cap W, \\
pr_{23}^*L^{-1}_{W \cap W^\prime} & \longrightarrow W \cap W^\prime \cap W, \\
pr_{13}^*L_{W} & \longrightarrow W \cap W^\prime \cap W.
\end{align*}
We note that the line bundle $L_W$ is trivial. Hence, its pullback along $pr_{13}$ is also a trivial line bundle over $W \cap W^\prime \cap W$. Tensoring the first two line bundles, we obtain a canonically trivial line bundle, since they are mutually inverse. In particular, this defines the multiplication of the bundle gerbe in this case.
\end{exa}
\begin{exa}[Construction two for the basic bundle gerbe over $S^3$]\label{ourbun}
Let $$V_N = S^2 \setminus \left\{(0,0,-1)\right\}$$ and $$V_S = S^2 \setminus \left\{(0,0,1)\right\}$$ be the neighbourhoods of the north pole and south poles of $S^2$, respectively. Define
$$W_N = \eta^{-1}(V_N), ~~~~~ W_S = \eta^{-1}(V_S),$$
the preimages of $U$ and $V$ under the Hopf map. Then $\left\{W_N,W_S\right\}$ forms an open cover of $S^3$. The intersection $V_N \cap V_S$ is homotopy equivalent to the equator $S^1$. Moreover,
\[
W_N \simeq V_N \times S^1, ~~~~~ W_S \simeq V_S \times S^1,
\]
so that $W_N \cap W_S$ is homotopy equivalent to the torus $T^2 = S^1 \times S^1$, with homotopy equivalence chosen to be given by canonical maps
\begin{align*}
f &: W_N \cap W_S \longrightarrow S^1 \times S^1, \qquad
\left((X,Y,Z),(X^\prime,Y^\prime)\right) \longmapsto
\left(
(\frac{X}{\sqrt{X^2+Y^2}},
\frac{Y}{\sqrt{X^2+Y^2}}),
(X^\prime,Y^\prime)
\right), \\
g &: S^1 \times S^1 \longrightarrow W_N \cap W_S, \qquad
\left((X,Y),(X^\prime,Y^\prime)\right) \longmapsto \left((X,Y,0),(X^\prime,Y^\prime)\right).
\end{align*}
The bundle gerbe $(L,Y,S^3)$ is describes as follows: \\
(\rNum{1}) $Y = W_N \sqcup W_S$, \\
(\rNum{2}) A line bundle $L \longrightarrow Y^{[2]} = W_N \sqcup (W_N \cap W_S) \sqcup (W_S \cap W_N) \sqcup W_S$; The line bundles over $W_N$ and $W_S$ are trivial. \\
Over $W_N \cap W_S \simeq T^2$, we have the pullback of the line bundle defined by the quotient construction
\[
\begin{tikzcd}
\mathbb{R}^2 \times \mathbb{C} \arrow[r] \arrow[d] & (\mathbb{R}^2 \times \mathbb{C})/\sim \arrow[d] \\
\mathbb{R}^2 \arrow[r] & \mathbb{R}^2/\mathbb{Z}^2
\end{tikzcd}
\]
along the map $f$, where the equivalence relation is given by
$$\bigl((x,y),\epsilon\bigr) \sim \bigl((x+m,y+n),e^{2\pi \mathsf{i} my} \cdot \epsilon\bigr)$$ for $x,y \in \mathbb{R}$ and $m,n \in \mathbb{Z}$. \\
(\rNum{3}) The bundle gerbe product
\[
m:\pi_{12}^*L \otimes \pi_{23}^*L \longrightarrow \pi_{13}^*L
\]
can be described case by case, as in Example \ref{murbun}.
\end{exa}
\begin{rem}\label{butobunandconv}\cite[Section 7 on page 38]{MR3089401}
Let $M$ be a smooth manifold. \\
$(1)$ Given a principal $\mathbb{B}U(1)$-bundle $\rho:\mathcal{P} \longrightarrow M$, one can construct a bundle gebre $(L,Y,M)$. First, we require a smooth manifold $Y$ together with a surjective submersion $\pi:Y \longrightarrow M$. We take $Y = \mathsf{Obj}(\mathcal{P})$ and define $\pi = \rho_0$. Next, we require a line bundle over $Y^{[2]}$. Let $P = \mathsf{Mor}(\mathcal{P})$, and let $\left\{U_i\right\}_{i \in I}$ be an open cover of $M$ satisfying the properties of Definition \ref{prin2bun}. From Notation \ref{objmornewform}, the objects of $\mathcal{P}$ are of the form $(i,u,e)$, where $u \in U_i$, while the morphisms of $\mathcal{P}$ are of the form $(i,j,u,e,[a])$, where $u \in U_i \cap U_j$ and $[a] \in \mathbb{R}/\mathbb{Z}$. Hence,
\begin{align*}
\pi\bigl(\mathsf{s}(i,j,u,e,[a])\bigr) & = \pi(i,u,e) \\
& = u \\
\pi\bigl(\mathsf{t}(i,j,u,e,[a])\bigr) & = \pi(j,u,e) \\
& = u
\end{align*}
By the universal property of pullback, there exists a unique map $\lambda:P \longrightarrow Y^{[2]}$ making the following diagram commute:
\[
\begin{tikzcd}
P \arrow[drrr,bend left=45,"\mathsf{t}"] \arrow[dr,,dashed,"\lambda = {(\mathsf{s},\mathsf{t})}"] \arrow[dddr,bend right=45,"\mathsf{s}"'] & & & \\
& Y^{[2]} \arrow[rr,"pr_2"] \arrow[dd,"pr_1"'] & & Y \arrow[dd,"\pi"] \\
& & & \\
& Y \arrow[rr,"\pi"'] & & M 
\end{tikzcd}
\]
We now show that $\lambda:P \longrightarrow Y^{[2]}$ defines a principal $U(1)$-bundle. Let $(y_1,y_2) \in Y^{[2]}$, and set $x = \pi(y_1) = \pi(y_2)$. Since $\rho$ is a $\mathbb{B}U(1)$-bundle, there exist a neighbourhood $U_x$ of $x$ and an equivalence
$$\phi_x:\rho^{-1}(U_x) \longrightarrow U_x \times \mathbb{B}U(1)$$
with weak inverse $\bar{\phi}_x$, commutating the diagram in Definition \ref{prin2bun}. Let $U_y = \rho^{-1}(U_x)$, and consider the neighbourhood $U_y \times_{M} U_y$ of $(y_1,y_2)$. Note that we already have an action of $U(1)$ on $P$, induced by the action of $\mathbb{B}U(1)$ on $\mathcal{P}$. Define a homeomorphism 
$$\psi_y:\lambda^{-1}(U_y \times_M U_y) \longrightarrow (U_y \times_M U_y) \times U(1),$$
by
$$\psi_y(i,j,u,e,[a]) = \bigl((i,u,e),(j,u,e),[a]\bigr).$$
Thus, $P$ defines a principal $U(1)$-bundle $P$ over $Y^{[2]}$. By Remark \ref{linebuncirbun}, this yields a line bundle $L$ over $Y^{[2]}$. \\
Finally, we must define an isomorphism $m:\pi_{12}^*L \otimes \pi_{23}^*L \longrightarrow \pi_{13}^*L$, which satisfies the associativity condition in diagram (\ref{gerbeprod}) of Definition \ref{gerbe}. Since
\[
Y^{[3]} = Y^{[2]} \times_Y Y^{[2]},
\]
we define $m$ over the triple intersection $U_{ijk}$ using the transition data of $L \longrightarrow Y^{[2]}$. Writing
\[
U_{ijk} = U_{ij} \cap U_{jk},
\]
Comparison of fibres over $U_{ij}$ and $U_{jk}$ determines the isomorphism to the fibre over $U_{ik}$. \\
$(2)$ Conversely, given a bundle gerbe $(L,Y,M)$, one can construct a $\mathbb{B}U(1)$-bundle $\rho:\mathcal{P} \longrightarrow M$. Set $\mathsf{Obj}(\mathcal{P}) = Y$ and $\mathsf{Mor}(\mathcal{P}) = L$, with the diagram
\[
\begin{tikzcd}
L \arrow[r,"\lambda"] & Y^{[2]} \arrow[r,shift left,"pr_1"] \arrow[r,shift right,"pr_2"'] & Y
\end{tikzcd}.
\]
For $l \in L$, define:
\begin{align*}
\mathsf{s}(l) & = pr_1\bigl(\lambda(l)\bigr) \\
\mathsf{t}(l) & = pr_2\bigl(\lambda(l)\bigr).
\end{align*}
If $l$ and $l^\prime$ are two composable morphisms in $L$, their composition is given by $m(l,l^\prime)$, where $m$ denotes the bundle gerbe product. \\
To define the identity map $\mathsf{id}:\mathsf{Obj}(\mathcal{P}) \longrightarrow \mathsf{Mor}(\mathcal{P})$, first consider :
\begin{align*}
d & : Y \longrightarrow Y^{[2]}, ~~ y \mapsto (y,y), \\
\Delta & : Y \longrightarrow Y^{[3]}, ~~ y \mapsto (y,y,y).
\end{align*}
Pulling back $m$ along $\Delta$ yields
\[
\Delta^*m:(\pi_{12} \circ \Delta)^*L \otimes (\pi_{23} \circ \Delta)^*L \longrightarrow (\pi_{13} \circ \Delta)^*L,
\]
which provides an isomorphism
\begin{equation}
\Delta^*m:d^*L \otimes d^*L \longrightarrow d^*L.
\end{equation}
The map 
$$\mathsf{id}:\mathsf{Obj}(\mathcal{P}) \longrightarrow \mathsf{Mor}(\mathcal{P})$$
is then defined by the composition
\[
\begin{tikzcd}
Y \arrow[r,"s_Y"] & Y \times U(1) \arrow[r,"\cong"] & d^*L \otimes (d^*L)^\vee \arrow[dd,"{(\Delta^*m)^{-1} \otimes id}"] & & \\
& & & & & \\
& & d^*L \otimes d^*L \otimes (d^*L)^\vee \arrow[r,"\cong"] & d^*L \otimes (Y \times U(1)) \cong d^*L \arrow[r,"pr_2"] & L
\end{tikzcd},
\]
where $s_Y$ is the canonical section of the trivial $U(1)$-bundle $Y \times U(1) \longrightarrow Y$, and $(d^*L)^\vee$ denotes the dual bundle of $d^*L$. \\
The action of $\mathbb{B}U(1)$ on $\mathcal{P}$ is given by the trivial action $\left\{\mathbbm{1}\right\}$ on $Y = \mathsf{Obj}(\mathcal{P})$, together with the natural action of $U(1)$ on $L = \mathsf{Mor}(\mathcal{P})$ coming from the associated line bundle. Finally, the projection $\rho:\mathcal{P} \longrightarrow M$ is defined by the surjective submersion $\pi:Y \longrightarrow M$.
\end{rem}
In the following, we recall the definition of a connection and curving on a bundle gerbe. We employ these structures in the study symmetries of our bundle gerbe, where we explain what we mean by "symmetries" at the beginning of section $6$.
\begin{defn}\cite{MR1669206,MR2681698}\cite[Definition 1 on page 34]{MR2318847}\label{gerbeconn}
Let $(L,Y,M)$ be a bundle gerbe. A \textbf{connection} and \textbf{curving} on this bundle gerbe consist of: \\
$(1)$ A connection $\nabla$ on the line bundle $L \longrightarrow Y^{[2]}$. \\
$(2)$ An associative isomorphism
$$m_{\nabla}:\pi^*_{12}(L,\nabla) \otimes \pi^*_{23}(L,\nabla) \longrightarrow \pi^*_{13}(L,\nabla),$$
which means that the bundle gerbe product $m$ is compatible with the connection $\nabla$ on $L$. \\
$(3)$ A $2$-form $C \in \Omega^2(Y)$, called the \textbf{curving}, satisfying
$$\pi^*_2 C - \pi^*_1 C = R_\nabla,$$
where $R_\nabla$ denotes the curvature of $\nabla$. \\
The connection on $(L,Y,M)$ is denoted by $\nabla$, and the associated $2$-form $C$ is called the \textbf{curving}.
Considering these data, there exists a unique $3$-from $D$ on $M$ such that $$\pi^*D = dC.$$
\end{defn}
In what follows, we recall the definition of a cocycle for a bundle gerbe. Such a cocycle determines a class in Deligne cohomology, as noted in Remark \ref{clager}.
\begin{defn}\cite[Section 1.1]{Wal2007phd}
Let $M$ be a manifold, and let $\left\{U_i\right\}_{i \in I}$ be an open cover of $M$, a \textbf{cocycle} for a bundle gerbe $(L,Y,M)$ with respect to this cover is a triple 
$$(g_{ijk},A_{ij},B_i),$$ where: \\
$(1)$ $g_{ijk}$ is a smooth $U(1)$-valued function on the triple intersection $U_i \cap U_j \cap U_k$, \\
$(2)$ $A_{ij}$ is a $1$-form on $U_i \cap U_j$, \\
$(3)$ $B_i$ is a $2$-from on $U_i$. \\
These data are required to satisfy the following conditions:
\begin{align*}
g_{ijk} \cdot g_{ikl} & = g_{jkl} \cdot g_{ijl}, \\
A_{ik} & = A_{ij} + A_{jk} + \frac{1}{\mathsf{i}} g^{-1}_{ijk}dg_{ijk}, \\
dA_{ij} & = B_j - B_i.
\end{align*}
\end{defn}
\begin{rem}\cite[Section 1.1]{Wal2007phd}\label{clager}
Let $(L,Y,M)$ be a bundle gerbe, and let $\left\{U_i\right\}_{i \in I}$ be a good cover of $M$. For each $i \in I$, choose a local section $s_i:U_i \longrightarrow Y$. These determine sections 
$$(s_i,s_j):U_i \cap U_j \longrightarrow Y^{[2]}$$
of the projection $Y^{[2]} \longrightarrow M$ over pairwise overlaps. Pulling back the line bundle $\lambda:L \longrightarrow Y^{[2]}$ along $(s_i,s_j)$ yields
\[
L_{ij} = (U_i \cap U_j) \times_{Y^{2}} L,
\]
as in the commutative diagram
\[
\begin{tikzcd}
L_{ij} \arrow[r,"pr_2"] \arrow[d,"{(s_i,s_j)^*\lambda}"'] & L \arrow[d,"\lambda"] \\
U_i \cap U_j \arrow[r,"{(s_i,s_j)}"'] & Y^{[2]}
\end{tikzcd}.
\]
We may regard the map $(s_i,s_j)^*\lambda$ as projection to the first factor. Choosing unit sections $\delta_{ij}$ of the line bundles $L_{ij} \longrightarrow U_i \cap U_j$ gives maps
\begin{equation}\label{gijkgerbe}
g_{ijk}:U_i \cap U_j \cap U_k \longrightarrow U(1)
\end{equation}
defined by the relation
$$m\bigl(\delta_{ij}(x),\delta_{jk}(x)\bigr) = g_{ijk}(x) \delta_{ik}(x)$$
for $x \in U_i \cap U_j \cap U_k$. If the bundle gerbe carries a connection and curving $(\nabla,C)$, then pulling back $\nabla$ along $(s_i,s_j)$ defines $1$-forms $A_{ij}$ on overlaps $U_i \cap U_j$. Similarly, restricting the curving $C$ along $Y \longrightarrow U_i$ gives $2$-forms $B_i$ on each $U_i$. Altogether, the data $(g_{ijk},A_{ij},B_i)$ form a cocycle associated to the bundle gerbe. \\
There are two cases to consider: \\
(\rNum{1}) without considering a connection and curving on our bundle gerbe, and \\
(\rNum{2}) with connection and curving included. \\
In case (\rNum{1}), Equation (\ref{gijkgerbe}) yields a cohomology class
$$[g_{ijk}] \in \check{H}^2\bigr(M,\underline{U(1)}\bigr).$$
From the long exact sequence
$$\cdots \longrightarrow 0 = \check{H}^2\bigl(M,\underline{\mathbb{R}}\bigr) \longrightarrow \check{H}^2\bigl(M,\underline{U(1)}\bigr) \longrightarrow \check{H}^3(M,\underline{\mathbb{Z}}) \longrightarrow 0 = \check{H}^3(M,\underline{\mathbb{R}}) \longrightarrow \cdots$$
induced by the exact sequence
$$1 \longrightarrow \underline{\mathbb{Z}} \longrightarrow \underline{\mathbb{R}} \longrightarrow \underline{U(1)} \longrightarrow 1,$$
we obtain the Dixmier-Douady class of the bundle gerbe $(P,Y,M)$ in $H^3(M;\mathbb{Z})$. \\
In case (\rNum{2}), the construction instead produces a class in Deligne cohomology. For further details on case (\rNum{1}), see \cite[on page 245]{MR2681698}, and for case (\rNum{2}), see \cite[Section 1.3]{Wal2007phd}.
\end{rem}
\begin{obs}\label{equiofbungerbes}
The bundle gerbes in Examples \ref{murbun} and \ref{ourbun} are equivalent. In Example \ref{murbun}, the class of the line bundle $L_{W \cap W^\prime}$ is $1$ in $H^2\bigl(W \cap W\prime;\mathbb{Z}\bigr)$, since
\[
H^2(W \cap W^\prime;\mathbb{Z}) \cong H^2(S^2;\mathbb{Z})
\]
, and the corresponding line bundle over $S^2$ in the line bundle associated to the Hopf map, which, by the convention fixed in the introduction, represents the generator $1$ in $H^2(S^2;\mathbb{Z})$. It follows the resulting class in $H^3(S^3;\mathbb{Z})$ is $1$, as it is the image of the class of its non-trivial line bundle $L_{W \cap W^\prime}$ under the following associated connecting homomorphism in the Mayer-Vietoris sequence
\begin{equation}\label{mayervie1}
\begin{tikzcd}
\cdots \arrow[r] & 0 = \check{H}^2(W,\mathbb{Z}) \oplus \check{H}^2(W^\prime,\mathbb{Z}) \arrow[r] & \check{H}^2\bigl(W\cap W^\prime,\mathbb{Z}\bigr) \arrow[d] & \\  & & \check{H}^3\bigl(S^3,\mathbb{Z}\bigr) \arrow[r] & 0 = \check{H}^3(W,\mathbb{Z}) \oplus \check{H}^3(W^\prime,\mathbb{Z}) \arrow[d] \\
& & & \vdots
\end{tikzcd}
\end{equation}
since
\[
\check{H}^3(S^3,\mathbb{Z}) \cong H^3(S^3;\mathbb{Z})
\]
and the positive generator of $H^2(W \cap W^\prime;\mathbb{Z})$ maps to the positive generator of $H^3(S^3;\mathbb{Z})$ under the connecting homomorphism in (\ref{mayervie1}). \\
The bundle gerbe in Example \ref{ourbun} has the class $1$ in $H^3(S^3;\mathbb{Z})$. The transition function and the connection on the line bundle of our bundle gerbe are given in the following \v{C}ech-de Rham double complex\footnote{See \cite[Section $8$ in Chapter $2$]{MR658304} for the notion of \v{C}ech-de Rham complexes.}diagram \cite[Remark 3.1 on page 6]{MR3764535}:
\begin{figure}[H]
\centering
\begin{tikzcd}
\draw[thick] (0,0) -- (0,3.5);
\draw[thick] (0,0) -- (6,0);
\draw (1,0) node[below]{\mathbb{R}^2};
\draw (2.5,0) node[below]{\mathbb{R}^2 \times \mathbb{Z}^2};
\draw (4.8,0) node[below]{\mathbb{R}^2 \times \mathbb{Z}^2 \times \mathbb{Z}^2};
\draw (0,1) node[left]{\underline{U(1)}};
\draw (0,2) node[left]{\Omega^1};
\draw (0,3) node[left]{\Omega^2};
\draw (1,2.7) node[above]{dx \wedge dy};
\draw (1,1.7) node[above]{xdy};
\draw[|->] (1,2.3) -- (1,2.8);
\draw[|->] (1.5,1.95) -- (2,1.95);
\draw (2.5,1.7) node[above]{m dy};
\draw (2.6,0.7) node[above]{e^{2\pi my\mathsf{i}}};
\draw[|->] (2.8,1.2) -- (2.8,1.7);
\draw[|->] (3.2,0.9) -- (4.3,0.9);
\draw (4.6,0.7) node[above]{1};
\end{tikzcd}
\caption{\v{C}ech-de Rham complex diagram}
\label{cechderhamcom}
\end{figure}
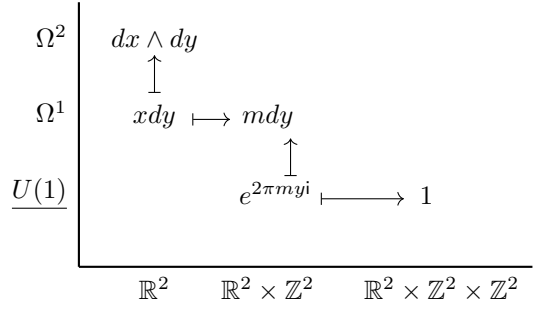
The class of this bundle gerbe is obtained as similar to the Mayer-Vietoris sequence in $(\ref{mayervie1})$.
\end{obs}

\section{Categorified clutching functions}
In this section, we develop the necessary ingredients for categorifying the transition function of the classical Hopf map,
\[
V_N \cap V_S \longrightarrow U(1),
\]
where $\left\{V_N,V_S\right\}$ denotes the standard two-open cover of $S^2$ by the northern and southern hemispheres. Our goal is to define a categorical transition map
$$V_N \cap V_S \longrightarrow \mathcal{U}(1),$$
which will serve as the transition data for the categorical Hopf map. Given a $2$-cocycle as in \cite{MR3894086} and \cite{MR2805195}, one can construct an associated categorical bundle. \\
For the construction in the next section, we work with a six-open cover of $S^2$, as this is better adapted to Wockel’s formalism. By contrast, the two-open cover suffices when working within Waldorf’s framework. In what follows, we define the maps that will be used to specify the transition functions of the categorical Hopf map.
\begin{defn}\label{coverofs1}
Let $\mathcal{U}$ denote the \v{C}ech groupoid associated with the open cover 
$\left\{U_1,U_2,U_3,U_4\right\}$ of $S^1$, where
\begin{align*}
U_1 & = \left\{(X,Y) \in S^1 ~ \big\vert ~ X > 0\right\}, \\
U_2 & = \left\{(X,Y) \in S^1 ~ \big\vert ~ X < 0\right\}, \\
U_3 & = \left\{(X,Y) \in S^1 ~ \big\vert ~ Y > 0\right\}, \\
U_4 & = \left\{(X,Y) \in S^1 ~ \big\vert ~ Y < 0\right\}.
\end{align*}
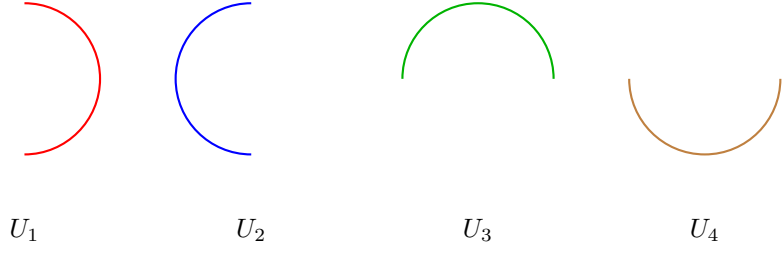
\begin{figure}[H]
\centering
\begin{tikzpicture}
%U_1
\draw[red,thick] (0,-1) arc [start angle=-90,end angle=90,radius=1cm];

%U_2
\draw[blue,thick] (3,1) arc [start angle=90,end angle=270,radius=1cm];

%U_3
\draw[black!30!green,thick] (7,0) arc [start angle=0,end angle=180,radius=1cm];

%U_4
\draw[brown,thick] (8,0) arc [start angle=-180,end angle=0,radius=1cm];

%labels
\node at (0,-2) {$U_1$};
\node at (3,-2) {$U_2$};
\node at (6,-2) {$U_3$};
\node at (9,-2) {$U_4$};
\end{tikzpicture}
\caption{The cover of $S^1$}
\end{figure}
\noindent
Thus, $\mathcal{U}$ has \\
(\rNum{1}) Objects: $$U_1 \sqcup U_2 \sqcup U_3 \sqcup U_4$$
(\rNum{2}) Arrows:
$$(U_1 \cap U_1) \sqcup (U_1 \cap U_3) \sqcup (U_1 \cap U_4) \sqcup (U_2 \cap U_2) \sqcup (U_2 \cap U_3) \sqcup (U_2 \cap U_4) \sqcup (U_3 \cap U_3) \sqcup (U_3 \cap U_1)$$ 
$$\sqcup (U_3 \cap U_2) \sqcup (U_4 \cap U_4) \sqcup (U_4 \cap U_1) \sqcup (U_4 \cap U_2)$$
An object in $\mathcal{U}$ can be written as a pair $\bigl(i,(X,Y)\bigr)$, where $(X,Y) \in U_i$. For $U_i \cap U_j \neq \varnothing$, an arrow is given by $\bigl(i,j,(X,Y)\bigr)$ with $(X,Y) \in U_i \cap U_j$. The source and target of such an arrow are
$$\mathsf{s}\bigl(i,j,(X,Y)\bigr) = \bigl(i,(X,Y)\bigr), ~~~~~ \mathsf{t}\bigl(i,j,(X,Y)\bigr) = \bigl(j,(X,Y)\bigr).$$
In diagrammatic form,
\[
\begin{tikzcd}
\bigl(i,(X,Y)\bigr) \arrow[rr,"{\bigl(i,j,(X,Y)\bigr)}"] & & \bigl(j,(X,Y)\bigr)
\end{tikzcd}
\]
For the general definition of \v{C}ech groupoids, see \cite[Section 3 on page 370]{MR3089401}.
\end{defn}
\begin{defn}
\cite[Section 1.3 on page 207]{MR1950948} Let $G$ be a group acting on a topological space $X$ from the left. The action groupoid $[G \times X \rightrightarrows X]$ is defined to have object set $X$ and morphism set $G \times X$. The source map is given by the projection to the first factor, while the target map is the action $G \times X \longrightarrow X$. The composition composable arrows is induced by the multiplication in $G$. \\
In the literature, this groupoid is sometimes denoted by $X{\salash}G$ or $[X/G]$. For the sake of consistency, throughout this thesis we adopt the notation $[G \times X \rightrightarrows X]$ for the action groupoid associated with a left action of $G$ on $X$.
\end{defn}
\begin{const}\label{phiandkappa}
We aim to define a groupoid homomorphism from $\mathcal{U}$ to $\mathcal{U}(1)$. To begin, consider the groupoid homomorphism $\phi:[\mathbb{Z} \times \mathbb{R} \rightrightarrows \mathbb{R}] \longrightarrow \mathcal{U}(1)$ defined by:
\begin{align*}
\phi_0 & : r \mapsto r \\
\phi_1 & : (m,r) \mapsto \bigl((m,[r]),r\bigr).
\end{align*}
Here, the action of $\mathbb{Z}$ on $\mathbb{R}$ in the groupoid $[\mathbb{Z} \times \mathbb{R} \rightrightarrows \mathbb{R}]$ is given by addition. \\
It remains to define a groupoid homomorphism
$$\kappa:\mathcal{U} \longrightarrow [\mathbb{Z} \times \mathbb{R} \rightrightarrows \mathbb{R}],$$ which, upon composition with $\phi$, yields the desired groupoid homomorphism $\mathcal{U} \longrightarrow \mathcal{U}(1)$. We define $\kappa$ so as to make the following diagram commute:
\begin{equation}\label{kappacomm}
\begin{tikzcd}
\mathcal{U} \arrow[rr,dashed,"\kappa"] \arrow[dr,"p"'] & & {[}\mathbb{Z} \times \mathbb{R} \rightrightarrows \mathbb{R}{]} \arrow[dl,"q"] \\
& {[}S^1 \rightrightarrows S^1{]} &
\end{tikzcd}
\end{equation}
where $p$ and $q$ are groupoid homomorphisms defined by:
\begin{align*}
p_0 & : (X,Y) \in U_i \mapsto (X,Y) \\
p_1 & : (X,Y) \in U_i \cap U_j \mapsto (X,Y)
\end{align*}
and
\begin{align*}
q_0 & : r \mapsto \big(\cos(2\pi r),\sin(2\pi r)\big) \\
q_1 & : (m,r) \mapsto \big(\cos (2\pi r+2\pi m),\sin(2\pi r+2\pi m)\big).
\end{align*}
Accordingly, the functions $\arcsin$ and $\arccos$ appear in the following definition of $\kappa$. \\
We now define $\kappa$ to be given by \\
$\bullet$ at the level of objects:
\begin{align*}
\kappa_0 & : \begin{cases} (X,Y) \in U_1 \mapsto \frac{1}{2\pi} \times \arcsin(Y) \\[0.5cm] (X,Y) \in U_2 \mapsto \frac{1}{2\pi} \times \arcsin(-Y) + \frac{1}{2} \\[0.5cm] (X,Y) \in U_3 \mapsto \frac{1}{2\pi} \times \arccos(X) \\[0.5cm] (X,Y) \in U_4 \mapsto \frac{1}{2\pi} \times \arccos(-X)+\frac{1}{2}\end{cases}
\end{align*}
\begin{figure}[H]
\centering
\begin{tikzpicture}
%U_1
\draw[red,thick] (0,-1) arc [start angle=-90,end angle=90,radius=1cm];

%U_2
\draw[blue,thick] (4.5,1) arc [start angle=90,end angle=270,radius=1cm];

%U_3
\draw[black!30!green,thick] (9,-0.5) arc [start angle=0,end angle=180,radius=1cm];

%U_4
\draw[brown,thick] (11.5,0.5) arc [start angle=180,end angle=360,radius=1cm];

%labels
\node at (0,-2) {$U_1$};
\node at (4.5,-2) {$U_2$};
\node at (8,-2) {$U_3$};
\node at (12.5,-2) {$U_4$};
\node at (2,0) {$\overset{\kappa_0} \longrightarrow (-\frac{1}{4},\frac{1}{4})$};
\node at (5.5,0) {$\overset{\kappa_0} \longrightarrow (\frac{1}{4},\frac{3}{4})$};
\node at (10,0) {$\overset{\kappa_0} \longrightarrow (0,\frac{1}{2})$};
\node at (14.5,0) {$\overset{\kappa_0} \longrightarrow (\frac{1}{2},1)$};

\end{tikzpicture}
\caption{The map $\kappa_0$}
\end{figure}
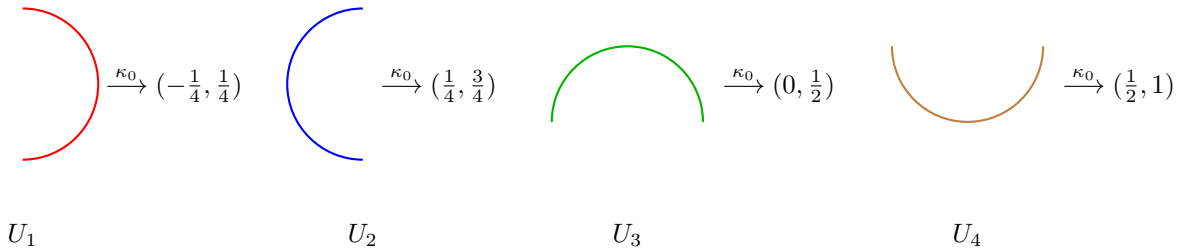
$\bullet$ at the level of morphisms:
\begin{align*}
\kappa_1: & (X,Y) \in U_1 \cap U_1 \mapsto \big(0,\frac{1}{2\pi} \times \arcsin(Y)\big), \\ & (X,Y) \in U_2 \cap U_2 \mapsto \big(0,\frac{1}{2\pi} \times \arcsin(-Y) + \frac{1}{2}\big), \\ & (X,Y) \in U_3 \cap U_3 \mapsto \big(0,\frac{1}{2\pi} \times \arccos(X)\big), \\\  & (X,Y) \in U_4 \cap U_4 \mapsto \big(0,\frac{1}{2\pi} \times \arccos(-X)+\frac{1}{2}\big), \\ & (X,Y) \in U_1 \cap U_3 \mapsto \big(0,\frac{1}{2\pi} \times \arcsin(Y)\big), \qquad (X,Y) \in U_1 \cap U_4 \mapsto \big(1,\frac{1}{2\pi} \times \arcsin(Y)\big), \\ & (X,Y) \in U_2 \cap U_3 \mapsto \big(0,\frac{1}{2\pi} \times \arcsin(-Y)+\frac{1}{2}\big), \\ & (X,Y) \in U_2 \cap U_4 \mapsto \big(0,\frac{1}{2\pi} \times \arcsin(-Y)+\frac{1}{2}\big), \\ & (X,Y) \in U_3 \cap U_1 \mapsto \big(0,\frac{1}{2\pi} \times \arccos(X)\big), \qquad (X,Y) \in U_3 \cap U_2 \mapsto \big(0,\frac{1}{2\pi} \times \arccos(X)\big), \\ & (X,Y) \in U_4 \cap U_1 \mapsto \big(-1,\frac{1}{2\pi} \times \arccos(-X)+\frac{1}{2}\big), \\
& (X,Y) \in U_4 \cap U_2 \mapsto \big(0,\frac{1}{2\pi} \times \arccos(-X) + \frac{1}{2}\big).
\end{align*}
\begin{rem}
Once $\kappa_0$ is defined, the map $\kappa_1$ is determined by $\kappa_0$, since the source and target of each arrow are known, as well as their images under $\kappa_0$. The integers appearing in the definition of $\kappa_1$ thus arise as the differences between the targets and sources of arrows. When verifying the commutativity of the diagram (\ref{kappacomm}) in Construction \ref{phiandkappa}, it is important to note that in our framework the ranges of the functions $\arcsin$ and $\arccos$ are $(-\frac{\pi}{2},\frac{\pi}{2})$ and $(0,\pi)$, respectively.
\end{rem}
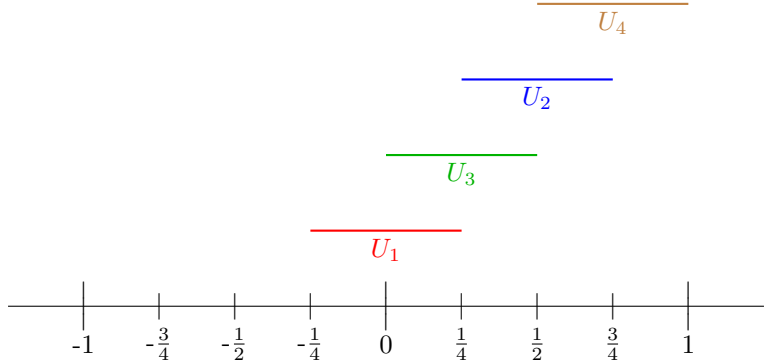
\begin{figure}[H]
\centering
\begin{tikzpicture}

\draw (-5,0) -- (5,0);
\node at (-1,0) {$|$};
\node at (-1,-0.5) {-$\frac{1}{4}$};
\node at (-2,0) {$|$};
\node at (-2,-0.5) {-$\frac{1}{2}$};
\node at (-3,0) {$|$};
\node at (-3,-0.5) {-$\frac{3}{4}$};
\node at (-4,0) {$\Big|$};
\node at (-4,-0.5) {-$1$};
\node at (0,0) {$\Big|$};
\node at (0,-0.5) {$0$};
\node at (1,0) {$|$};
\node at (1,-0.5) {$\frac{1}{4}$};
\node at (2,0) {$|$};
\node at (2,-0.5) {$\frac{1}{2}$};
\node at (3,0) {$|$};
\node at (3,-0.5) {$\frac{3}{4}$};
\node at (4,0) {$\Big|$};
\node at (4,-0.5) {$1$};
\draw[thick,black!30!green,thick] (0,2) -- (2,2);
\draw[thick,red] (-1,1) -- (1,1);
\draw[thick,blue] (1,3) -- (3,3);
\draw[thick,brown] (2,4) -- (4,4);
\node[red] at (0,0.75) {$U_1$};
\node[black!30!green,thick] at (1,1.75) {$U_3$};
\node[blue] at (2,2.75) {$U_2$};
\node[brown] at (3,3.75) {$U_4$};

\end{tikzpicture}
\caption{Range of $\kappa_0$}
\end{figure}
\begin{table}[H]
\begin{center}
\resizebox{0.9\textwidth}{!}{
\begin{tabular}{c || c | c | c | c}
$\cap$ & $U_1$ & $U_2$ & $U_3$ & $U_4$ \\
\hhline{= || = | = | = | =}
& & \cellcolor{gray!20!white} & & \\
$U_1$ & $\big(0,\frac{1}{2\pi} \times \arcsin(Y)\big)$ & \cellcolor{gray!20!white}  & $\big(0,\frac{1}{2\pi} \times \arcsin(Y)\big)$ & $\big(1,\frac{1}{2\pi} \times \arcsin(Y)\big)$ \\[0.5cm]
\hline
& \cellcolor{gray!20!white} & & & \\
$U_2$ & \cellcolor{gray!20!white} & $\big(0,\frac{1}{2\pi} \times \arcsin(-Y) + \frac{1}{2}\big)$ & $\big(0,\frac{1}{2\pi} \times \arcsin(-Y)+\frac{1}{2}\big)$ & $\big(0,\frac{1}{2\pi} \times \arcsin(-Y)+\frac{1}{2}\big)$ \\[0.5cm]
\hline
& & & & \cellcolor{gray!20!white} \\
$U_3$ & $\big(0,\frac{1}{2\pi} \times \arccos(X)\big)$ & $\big(0,\frac{1}{2\pi} \times \arccos(X)\big)$ & $\big(0,\frac{1}{2\pi} \times \arccos(X)\big)$ & \cellcolor{gray!20!white} \\[0.5cm]
\hline
& & & \cellcolor{gray!20!white} & \\
$U_4$ & $\big(-1,\frac{1}{2\pi} \times \arccos(-X)+\frac{1}{2}\big)$ & $\big(0,\frac{1}{2\pi} \times \arccos(-X) + \frac{1}{2}\big)$ & \cellcolor{gray!20!white} & $\big(0,\frac{1}{2\pi} \times \arccos(-X) + \frac{1}{2}\big)$ \\
& & & \cellcolor{gray!20!white} &
\end{tabular}
}
\end{center}
\caption{The map $\kappa_1$ (source and targets are sorted vertically and horizontally, respectively)}
\end{table}
\end{const}
\begin{lem}
The morphism $\phi$, defined in Construction \ref{phiandkappa}, is a groupoid homomorphism.
\end{lem}
\begin{proof}
We need to show that $\phi$ preserves all structure maps $\mathsf{s}$, $\mathsf{t}$, $\mathsf{mult}$, $\mathsf{unit}$ and $\mathsf{inv}$.
\begin{align*}
\phi_0\big(\mathsf{s}(m,r)\big) & = \phi_0(r) \\
& = r \\
\mathsf{s}\big(\phi_1(m,r)\big) & = \mathsf{s}\bigl((m,[r]),r\bigr) \\
& = r
\end{align*}
So, $\phi$ preserves the source map.
\begin{align*}
\phi_0\big(\mathsf{t}(m,r)\big) & = \phi_0(r+m) \\
& = r+m \\
\mathsf{t}\big(\phi_1(m,r)\big) & = \mathsf{t}\bigl((m,[r]),r\bigr) \\
& = r+m
\end{align*}
So, $\phi$ preserves the target map.
\begin{align*}
\phi_1\bigl((n,r+m) \circ (m,r)\bigr) & = \phi_1(m+n,r) \\
& = \bigl((m+n,[r]),r\bigr) \\
\phi_1(n,r+m) \circ \phi_1(m,r) & = \bigl((n,[r+m]),r+m\bigr) \circ \bigl((m,[r]),r\bigr) \\
& = \bigl((m+n,[r]),r\bigr)
\end{align*}
So, $\phi$ preserves the multiplication map.
\begin{align*}
\phi_1(\mathsf{unit}(r)) & = \phi_1(0,r) \\
& = \bigl((0,[r]),r\bigr) \\
\mathsf{unit}\bigl(\phi_0(r)\bigr) & = \mathsf{unit}(r) \\
& = \bigl((0,[r]),r\bigr)
\end{align*}
So, $\phi$ preserves the unit map.
\begin{align*}
\phi_1\big(\mathsf{inv}(m,r)\big) & = \phi_1(-m,r+m) \\
& = \bigl((-m,[r+m]),r+m\bigr) \\
\mathsf{inv}\bigl(\phi_1(m,r)\bigr) & = \mathsf{inv}\bigl((m,[r]),r\bigr) \\
& = \bigl((-m,[r+m]),r+m\bigr)
\end{align*}
So, $\phi$ preserves the inverse map.
\end{proof}
\begin{lem}
The morphism $\kappa$, defined in Construction \ref{phiandkappa}, is a groupoid homomorphism.
\end{lem}
\begin{proof}
Here, it suffices to verify that $\kappa$ preserves all structure maps $\mathsf{s}$, $\mathsf{t}$, $\mathsf{mult}$, $\mathsf{unit}$ and $\mathsf{inv}$. We carry out the check for objects in $U_2$ and $U_3$, and for arrows in $U_2 \cap U_3$; the remaining cases are analogous. Let $(x,y)$ be an arrow in $U_2 \cap U_3$ . Then, its source and target are in $U_2$ and $U_3$, respectively.
\begin{align*}
\kappa_0\big(\mathsf{s}(X,Y)\big) & = \kappa_0(X,Y) \\
& = \frac{1}{2\pi} \times \arcsin(-Y)+\frac{1}{2} \\
\mathsf{s}\big(\kappa_1(X,Y)\big) & = \mathsf{s}(0,\frac{1}{2\pi} \times \arccos(-Y)+\frac{1}{2}) \\
& = \frac{1}{2\pi} \times \arcsin(-Y)+\frac{1}{2}
\end{align*}
So, $\kappa$ preserves the source map.
\begin{align*}
\kappa_0\big(\mathsf{t}(X,Y)\big) & = k_0(X,Y) \\
& = \frac{1}{2\pi} \times \arccos(X) \\
\mathsf{t}\big(\kappa_1(X,Y)\big) & = \mathsf{t}(0,\frac{1}{2\pi} \times \arccos(X)) \\
& = \frac{1}{2\pi} \times \arccos(X)
\end{align*}
So, $\kappa$ preserves the target map.
\begin{align*}
\kappa_1\big((X,Y) \circ (X,Y)\big) & = \kappa_1(X,Y) \\
& = (0,\frac{1}{2\pi} \times \arcsin(-Y)+\frac{1}{2}) \\
\kappa_1(X,Y) \circ \kappa_1(X,Y) & = (0,\frac{1}{2\pi} \times \arcsin(-Y)+\frac{1}{2}) \circ (0,\frac{1}{2\pi} \times \arcsin(-Y)+\frac{1}{2}) \\ & = (0,\frac{1}{2\pi} \times \arcsin(-Y)+\frac{1}{2})
\end{align*}
So, $\kappa$ preserves the multiplication map.
\begin{align*}
\kappa_1(\mathsf{unit}(X,Y)) & = \kappa_1(X,Y) \\
& = (0,\frac{1}{2\pi} \times \arcsin(-Y) + \frac{1}{2}) \\
\mathsf{unit}(\kappa_0(X,Y)) & = \mathsf{unit}(\frac{1}{2\pi} \times \arcsin(-Y) + \frac{1}{2}) \\
& = (0,\frac{1}{2\pi} \times \arcsin(-Y) + \frac{1}{2})
\end{align*}
So, $\kappa$ preserves the unit map.
\begin{align*}
\kappa_1\big(\mathsf{inv}(X,Y)\big) & = \kappa_1(X,Y) \\
& = (0,\frac{1}{2\pi} \times \arccos(X)) \\
\mathsf{inv}\big(k_1(X,Y)\big) & = \mathsf{inv}(0,\frac{1}{2\pi} \times \arcsin(-Y)+\frac{1}{2}) \\
& = (0,\frac{1}{2\pi} \times \arcsin(-Y) + \frac{1}{2})
\end{align*}
So, $\kappa$ preserves the inverse map.
\end{proof}

\section{Categorical Hopf map}
In this section, we construct the categorical Hopf map in Construction \ref{cathopf}. We compute its associated cohomology class in Observation \ref{nontri2bun} and demonstrate in Remark \ref{factorizationdiagram} how the categorical Hopf map factors through the classical Hopf map.
\begin{const}\label{cathopf}
Take the following cover of $S^2$:
\begin{align*}
V_1 & = \left\{(X,Y,Z) \in S^2 ~ \big\vert ~ X > 0\right\}, \qquad
V_2 = \left\{(X,Y,Z) \in S^2 ~ \big\vert ~ X < 0\right\}, \\
V_3 & = \left\{(X,Y,Z) \in S^2 ~ \big\vert ~ Y > 0\right\}, \qquad
V_4  = \left\{(X,Y,Z) \in S^2 ~ \big\vert ~ Y < 0\right\}, \\
V_5 & = \left\{(X,Y,Z) \in S^2 ~ \big\vert ~ Z > 0\right\}, \qquad
V_6  = \left\{(X,Y,Z) \in S^2 ~ \big\vert ~ Z < 0\right\}.
\end{align*}
In the cover image below, The axes are not oriented in the standard way. The positive $Z$-axis is directed out of the plane, the positive $Y$-axis points upward within the plane, and the $X$-axis extends horizontally from left to right.
\begin{figure}[H]
\centering
\resizebox{0.95\textwidth}{!}{
\begin{tikzpicture}
%V_5
\draw[dashed,opacity=0.25] (13,0) arc[start angle=0,end angle=360,x radius=1,y radius=1];
\shade[ball color=blue!10!white,opacity=0.2] (13,0) arc[start angle=0,end angle=360,x radius=1,y radius=1];

%V_6
\draw[dashed] (16,0) arc[start angle=0,end angle=360,x radius=1,y radius=1];
\shade[ball color=blue!10!white,opacity=0.2] (16,0) arc[start angle=0,end angle=360,x radius=1,y radius=1];

%V_1
\draw[dashed] (0.5,0) arc[start angle=0,end angle=360,x radius=0.5,y radius=1];
\shade[ball color=blue!10!white,opacity=0.2] (0,1) arc[start angle=90,end angle=270
,x radius=0.5,y radius=1]
arc[start angle=-90,end angle=90,x radius=1,y radius=1];

%V_2
\draw[dashed] (3.5,0) arc[start angle=0,end angle=360,x radius=0.5,y radius=1];
\shade[ball color=blue!10!white,opacity=0.2] (3,-1) arc[start angle=-90,end angle=90,x radius=0.5,y radius=1]
arc[start angle=90,end angle=270,x radius=1,y radius=1];

%V_3
\draw[dashed,opacity=0.25] (7,0) arc[start angle=0, end angle=180, x radius=1, y radius=0.5];
\draw[dashed] (5,0) arc[start angle=180, end angle=360, x radius=1, y radius=0.5];
\shade[ball color=blue!10!white, opacity=0.2] (7,0) arc[start angle=0, end angle=-180, x radius=1, y radius=.5]
arc[start angle=180, end angle=0, x radius=1, y radius=1];

%V_4
\draw[dashed] (10,0) arc[start angle=0, end angle=360, x radius=1, y radius=0.5];
\shade[ball color=blue!10!white, opacity=0.2] (10,0) arc[start angle=0, end angle=180, x radius=1, y radius=.5]
arc[start angle=180, end angle=360, x radius=1, y radius=1];

%labels
\node at (12,-2) {$V_5$};
\node at (15,-2) {$V_6$};
\node at (0,-2) {$V_1$};
\node at (3,-2) {$V_2$};
\node at (6,-2) {$V_3$};
\node at (9,-2) {$V_4$};

\end{tikzpicture}
}
\caption{The cover of $S^2$}
\label{coverofs2}
\end{figure}
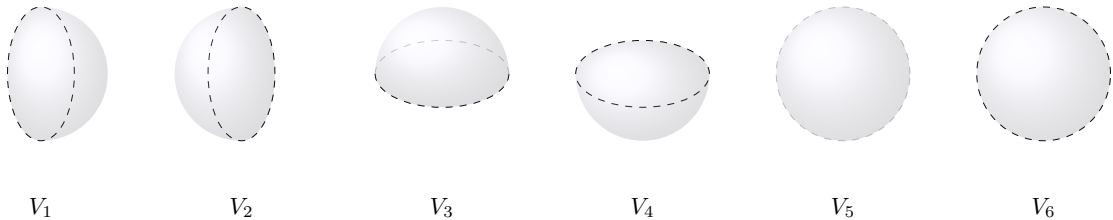
Our aim is to construct a principal $\mathcal{U}(1)$-bundle $\rho:\mathcal{P} \longrightarrow S^2$, following the formalism in \cite[Remark 2.16 on page 583]{MR2805195}. If $S^2$ is covered by the upper and the lower hemispheres, their intersection is the equator, which is homotopy equivalent to a circle. In this setting, it is not possible to define a strict morphism of categorical groups $U(1) \longrightarrow \mathcal{U}(1)$ that serves as the transition function for the prospective principal $\mathcal{U}(1)$-bundle. By instead covering $S^2$ with six opens, it suffices to consider the intersections $V_1 \cap V_5$, $V_2 \cap V_5$, $V_3 \cap V_5$, $V_4 \cap V_5$, since their union is homotopy equivalent to $S^1$. So, we require a morphism $\sigma:S^2 \setminus \left\{(0,0,1),(0,0,-1)\right\} \longrightarrow S^1$ which maps
\[
V_1 \cap V_5 ~~ \text{to} ~~ U_1, \qquad V_2 \cap V_5 ~~ \text{to} ~~ U_2, \qquad V_3 \cap V_5 ~~ \text{to} ~~ U_3, \qquad V_4 \cap V_5 ~~ \text{to} ~~ U_4.
\]
These requirements ensure that $\sigma$ maps any non-empty triple intersections of elements in the cover of $S^2$ to corresponding intersections of elements in the cover of $S^1$. For instance, $\sigma$ maps $V_1 \cap V_3 \cap V_5$ to $U_1 \cap U_3$. Explicitly, we define $\sigma$ to be given by:
$$\sigma:(X,Y,Z) \mapsto (\frac{X}{\sqrt{X^2+Y^2}},\frac{Y}{\sqrt{X^2+Y^2}})$$
Now, we define the following smooth maps using $\phi_0$, $\kappa_0$, $\sigma$:
\begin{align*}
g_{1 5} & : V_1 \cap V_5 \longrightarrow \mathbb{R}, ~~ (X,Y,Z) \mapsto \frac{1}{2\pi} \times \arcsin(\frac{Y}{\sqrt{X^2+Y^2}}) \\
g_{2 5} & : V_2 \cap V_5 \longrightarrow \mathbb{R}, ~~ (X,Y,Z) \mapsto \frac{1}{2\pi} \times \arcsin(- \frac{Y}{\sqrt{X^2+Y^2}}) + \frac{1}{2} \\
g_{3 5} & : V_3 \cap V_5 \longrightarrow \mathbb{R}, ~~ (X,Y,Z) \mapsto \frac{1}{2\pi} \times \arccos(\frac{X}{\sqrt{X^2+Y^2}}) \\
g_{4 5} & : V_4 \cap V_5 \longrightarrow \mathbb{R}, ~~ (X,Y,Z) \mapsto \frac{1}{2\pi} \times \arccos(-\frac{X}{\sqrt{X^2+Y^2}}) + \frac{1}{2} \\
g_{5 1} & : V_5 \cap V_1 \longrightarrow \mathbb{R}, ~~ (X,Y,Z) \mapsto - \frac{1}{2\pi} \times \arcsin(\frac{Y}{\sqrt{X^2+Y^2}}) \\
g_{5 2} & : V_5 \cap V_2 \longrightarrow \mathbb{R}, ~~ (X,Y,Z) \mapsto - \frac{1}{2\pi} \times \arcsin(- \frac{Y}{\sqrt{X^2+Y^2}}) - \frac{1}{2} \\
g_{5 3} & : V_5 \cap V_3 \longrightarrow \mathbb{R}, ~~ (X,Y,Z) \mapsto - \frac{1}{2\pi} \times \arccos(\frac{X}{\sqrt{X^2+Y^2}}) \\
g_{5 4} & : V_5 \cap V_4 \longrightarrow \mathbb{R}, ~~ (X,Y,Z) \mapsto - \frac{1}{2\pi} \times \arccos(-\frac{X}{\sqrt{X^2+Y^2}}) - \frac{1}{2}
\end{align*}
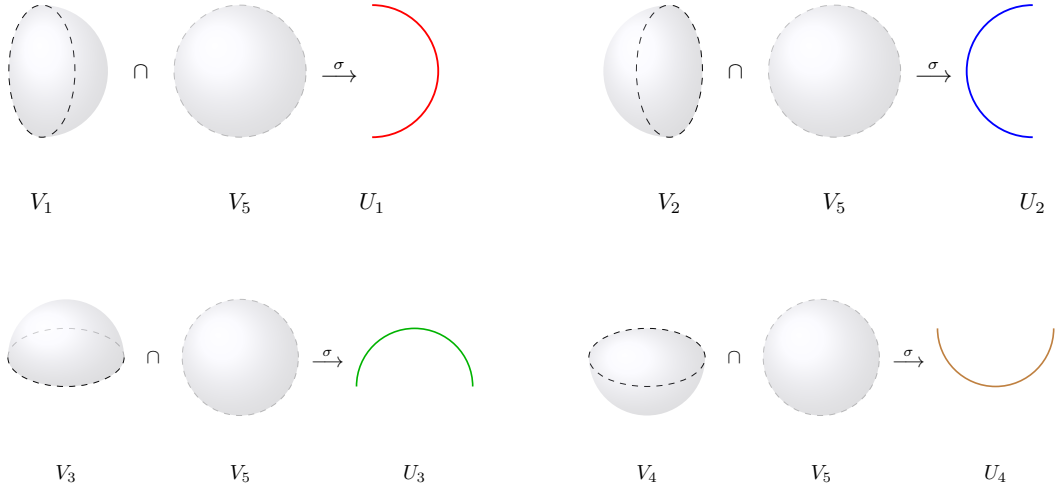
\begin{figure}[H]
\centering
\resizebox{0.9\textwidth}{!}{
\begin{tikzpicture}
%V_1
\draw[dashed] (0.5,0) arc[start angle=0,end angle=360,x radius=0.5,y radius=1];
\shade[ball color=blue!10!white,opacity=0.2] (0,1) arc[start angle=90,end angle=270,x radius=0.5,y radius=1]
arc[start angle=-90,end angle=90,x radius=1,y radius=1];

%V_5
\draw[dashed,opacity=0.25] (4,0) arc[start angle=0,end angle=360,x radius=1,y radius=1];
\shade[ball color=blue!10!white,opacity=0.2] (4,0) arc[start angle=0,end angle=360,x radius=1,y radius=1];

%U_1
\draw[red,thick] (5,-1) arc [start angle=-90,end angle=90,radius=1cm];

%V_2
\draw[dashed] (10,0) arc[start angle=0,end angle=360,x radius=0.5,y radius=1];
\shade[ball color=blue!10!white,opacity=0.2] (9.5,-1) arc[start angle=-90,end angle=90,x radius=0.5,y radius=1]
arc[start angle=90,end angle=270,x radius=1,y radius=1];

%V_5
\draw[dashed,opacity=0.25] (13,0) arc[start angle=0,end angle=360,x radius=1,y radius=1];
\shade[ball color=blue!10!white,opacity=0.2] (13,0) arc[start angle=0,end angle=360,x radius=1,y radius=1];

%U_2
\draw[blue,thick] (15,1) arc [start angle=90,end angle=270,radius=1cm];

%labels
\node at (0,-2) {$V_1$};
\node at (1.5,0) {$\cap$};
\node at (3,-2) {$V_5$};
\node at (4.5,0) {$\overset{\sigma}\longrightarrow$};
\node at (5,-2) {$U_1$};
\node at (9.5,-2) {$V_2$};
\node at (10.5,0) {$\cap$};
\node at (12,-2) {$V_5$};
\node at (13.5,0) {$\overset{\sigma}\longrightarrow$};
\node at (15,-2) {$U_2$};

\end{tikzpicture}
}
\\[1cm]
\resizebox{0.9\textwidth}{!}{
\begin{tikzpicture}
%V_3
\draw[dashed,opacity=0.25] (1.5,0) arc[start angle=0, end angle=180, x radius=1, y radius=0.5];
\draw[dashed] (-0.5,0) arc[start angle=180, end angle=360, x radius=1, y radius=0.5];
\shade[ball color=blue!10!white, opacity=0.2] (1.5,0) arc[start angle=0, end angle=-180, x radius=1, y radius=.5]
arc[start angle=180, end angle=0, x radius=1, y radius=1];

%V_5
\draw[dashed,opacity=0.25] (4.5,0) arc[start angle=0,end angle=360,x radius=1,y radius=1];
\shade[ball color=blue!10!white,opacity=0.2] (4.5,0) arc[start angle=0,end angle=360,x radius=1,y radius=1];

%U_3
\draw[black!30!green,thick] (7.5,-0.5) arc [start angle=0,end angle=180,radius=1cm];

%V_4
\draw[dashed] (11.5,0) arc[start angle=0, end angle=360, x radius=1, y radius=0.5];
\shade[ball color=blue!10!white, opacity=0.2] (11.5,0) arc[start angle=0, end angle=180, x radius=1, y radius=.5]
arc[start angle=180, end angle=360, x radius=1, y radius=1];

%V_5
\draw[dashed,opacity=0.25] (14.5,0) arc[start angle=0,end angle=360,x radius=1,y radius=1];
\shade[ball color=blue!10!white,opacity=0.2] (14.5,0) arc[start angle=0,end angle=360,x radius=1,y radius=1];

%U_4
\draw[brown,thick] (15.5,0.5) arc [start angle=180,end angle=360,radius=1cm];

%labels
\node at (0.5,-2) {$V_3$};
\node at (2,0) {$\cap$};
\node at (3.5,-2) {$V_5$};
\node at (5,0) {$\overset{\sigma}\longrightarrow$};
\node at (6.5,-2) {$U_3$};
\node at (10.5,-2) {$V_4$};
\node at (12,0) {$\cap$};
\node at (13.5,-2) {$V_5$};
\node at (15,0) {$\overset{\sigma}\longrightarrow$};
\node at (16.5,-2) {$U_4$};

\end{tikzpicture}
}
\caption{The map $\sigma$}
\end{figure}
\begin{table}[H]
\begin{center}
\resizebox{\textwidth}{!}{%
\begin{tabular}{c || c | c | c | c | c | c}
$\cap$ & $V_1$ & $V_2$ & $V_3$ & $V_4$ & $V_5$ & $V_6$ \\
\hhline{= || = | = | = | = | = | =}
& & \cellcolor{gray!20!white} & & & \cellcolor{green!10!white} & \\
$V_1$ & $0$ & \cellcolor{gray!20!white} & $0$ & $0$ & \cellcolor{green!10!white} $\frac{1}{2\pi} \times \arcsin(\frac{Y}{\sqrt{X^2+Y^2}})$ & $0$ \\[0.5cm]
\hline
& \cellcolor{gray!20!white} & & & & \cellcolor{green!10!white} & \\
$V_2$ & \cellcolor{gray!20!white} & $0$ & $0$ & $0$ & \cellcolor{green!10!white} $\frac{1}{2\pi} \times \arcsin(- \frac{Y}{\sqrt{X^2+Y^2}}) + \frac{1}{2}$ & $0$ \\[0.5cm]
\hline
& & & & \cellcolor{gray!20!white} & \cellcolor{green!10!white} & \\
$V_3$ & $0$ & $0$ & $0$ & \cellcolor{gray!20!white} & \cellcolor{green!10!white} $\frac{1}{2\pi} \times \arccos(\frac{X}{\sqrt{X^2+Y^2}})$ & $0$ \\[0.5cm]
\hline
& & & \cellcolor{gray!20!white} & & \cellcolor{green!10!white} & \\
$V_4$ & $0$ & $0$ & \cellcolor{gray!20!white} & $0$ & \cellcolor{green!10!white} $\frac{1}{2\pi} \times \arccos(-\frac{X}{\sqrt{X^2+Y^2}}) + \frac{1}{2}$ & $0$ \\[0.5cm]
\hline
& \cellcolor{green!10!white} & \cellcolor{green!10!white} & \cellcolor{green!10!white} & \cellcolor{green!10!white} & & \cellcolor{gray!20!white} \\
$V_5$ & \cellcolor{green!10!white} $-\frac{1}{2\pi} \times \arcsin(\frac{Y}{\sqrt{X^2+Y^2}})$ & \cellcolor{green!10!white} $-\frac{1}{2\pi} \times \arcsin(- \frac{Y}{\sqrt{X^2+Y^2}})-\frac{1}{2}$ & \cellcolor{green!10!white} $-\frac{1}{2\pi} \times \arccos(\frac{X}{\sqrt{X^2+Y^2}})$ & \cellcolor{green!10!white} $-\frac{1}{2\pi} \times \arccos(-\frac{X}{\sqrt{X^2+Y^2}})-\frac{1}{2}$ & $0$ & \cellcolor{gray!20!white} \\[0.5cm]
\hline
& & & & & \cellcolor{gray!20!white} & \\
$V_6$ & $0$ & $0$ & $0$ & $0$ & \cellcolor{gray!20!white} & $0$ \\
& & & & & \cellcolor{gray!20!white} &
\end{tabular}
}
\end{center}
\caption{Smooth maps on double intersections (sources and targets are sorted vertically and horizontally, respectively)}
\end{table}
All remaining smooth maps on double intersections are defined to be zero. We now define the following smooth maps constructed from $\phi_1$, $\kappa_1$, $\sigma$. We specify the integer-valued component arising from Equation \ref{cocycleg} of Definition \ref{gvaluedcocycle}, together with the second component always zero. In particular,
\[
\beta(h_{ijk}) = g_{ik} - g_{ij} - g_{jk}.
\]
We define:
\begin{align*}
h_{1 3 5} & : V_1 \cap V_3 \cap V_5 \longrightarrow \mathbb{Z} \times \mathbb{R}/\mathbb{Z}, ~~ (X,Y,Z) \mapsto \left(0,\left[0\right]\right) \\
h_{1 4 5} & : V_1 \cap V_4 \cap V_5 \longrightarrow \mathbb{Z} \times \mathbb{R}/\mathbb{Z}, ~~ (X,Y,Z) \mapsto \left(-1,\left[0\right]\right) \\
h_{2 3 5} & : V_2 \cap V_3 \cap V_5 \longrightarrow \mathbb{Z} \times \mathbb{R}/\mathbb{Z}, ~~ (X,Y,Z) \mapsto \left(0,\left[0\right]\right) \\
h_{2 4 5} & : V_2 \cap V_4 \cap V_5 \longrightarrow\mathbb{Z} \times \mathbb{R}/\mathbb{Z}, ~~ (X,Y,Z) \mapsto \left(0,\left[0\right]\right) \\
h_{3 1 5} & : V_3 \cap V_1 \cap V_5 \longrightarrow \mathbb{Z} \times \mathbb{R}/\mathbb{Z}, ~~ (X,Y,Z) \mapsto \left(0,\left[0\right]\right) \\
h_{3 2 5} & : V_3 \cap V_2 \cap V_5 \longrightarrow \mathbb{Z} \times \mathbb{R}/\mathbb{Z}, ~~ (X,Y,Z) \mapsto \left(0,\left[0\right]\right) \\
h_{4 1 5} & : V_4 \cap V_1 \cap V_5 \longrightarrow \mathbb{Z} \times \mathbb{R}/\mathbb{Z}, ~~ (X,Y,Z) \mapsto \left(1,\left[0\right]\right) \\
h_{4 2 5} & : V_4 \cap V_2 \cap V_5 \longrightarrow \mathbb{Z} \times \mathbb{R}/\mathbb{Z}, ~~ (X,Y,Z) \mapsto \left(0,\left[0\right]\right)
\end{align*}
From Equations (\ref{gii}), (\ref{cocycleg}), (\ref{hijj}) and (\ref{cocycleh}), we obtain the following pointwise equalities:
\begin{equation}\label{h231h123}
h_{3 5 1} = h_{1 3 5}.
\end{equation}
\begin{equation}\label{h312h321}
h_{5 1 3} = - h_{5 3 1}.
\end{equation}
As Equation (\ref{h231h123}) shows, placing the subscript $5$ in the middle and swapping the remaining two subscripts produces smooth maps that are equal pointwise in general. Similarly, Equation (\ref{h312h321}) indicates that positioning the subscript $5$ first and interchanging the other two subscripts yields two smooth maps that are pointwise inverses to each other. Hence, by defining the smooth maps
\begin{align*}
h_{5 1 3} & : V_5 \cap V_1 \cap V_3 \longrightarrow \mathbb{Z} \times \mathbb{R}/\mathbb{Z}, ~~ (X,Y,Z) \mapsto \left(0,[0]\right) \\
h_{5 1 4} & : V_5 \cap V_1 \cap V_4 \longrightarrow \mathbb{Z} \times \mathbb{R}/\mathbb{Z}, ~~ (X,Y,Z) \mapsto \left(-1,[0]\right) \\
h_{5 2 3} & : V_5 \cap V_2 \cap V_3 \longrightarrow \mathbb{Z} \times \mathbb{R}/\mathbb{Z}, ~~ (X,Y,Z) \mapsto \left(0,[0]\right) \\
h_{5 2 4} & : V_5 \cap V_2 \cap V_4 \longrightarrow \mathbb{Z} \times \mathbb{R}/\mathbb{Z}, ~~ (X,Y,Z) \mapsto \left(0,[0]\right)
\end{align*}
and defining all other remaining smooth maps on triple intersections to be zero maps, we have defined all the maps required for constructing our principal $\mathcal{U}(1)$-bundle $\rho:\mathcal{P} \longrightarrow S^2$. We note that all maps on triple intersections have already been specified in the preceding discussion. In particular, non-trivial smooth maps on triple intersections are defined only when one of the indices is $5$ and the other two indices are distinct; in all other cases, the maps are taken to be zero. Finally, we claim that the pair $(g_{ij},h_{ijk})$ defines a cocycle for a $\mathcal{U}(1)$-bundle over $S^2$. We verify conditions (\ref{gii}), (\ref{cocycleg}), (\ref{hijj}), and (\ref{cocycleh}) from Definition \ref{gvaluedcocycle} in some cases; the remaining cases follow similarly.
\begin{itemize}
\item Let $i=1$, $j=4$, $k=5$, and $l=4$. Condition (\ref{gii}) is automatically satisfies by the definition of the $g_{ij}$'s. We also have:
\begin{align*}
\beta\Big(h_{145}(X,Y,Z)\Big) & = -1, \\
g_{14}(X,Y,Z) & = 0, \\
g_{45}(X,Y,Z) & = \frac{\arccos(-\frac{X}{\sqrt{X^2+Y^2}})}{2\pi} + \frac{1}{2}, \\
g_{15}(X,Y,Z) & = \frac{\arcsin(\frac{Y}{\sqrt{X^2+Y^2}})}{2\pi}.
\end{align*}
Considering
\begin{align*}
\cos(\theta) & = \frac{X}{\sqrt{X^2+Y^2}}, \\
\sin(\theta) & = \frac{Y}{\sqrt{X^2+Y^2}},
\end{align*}
since the angle $\theta$ lies in the fourth quadrant, it follows that:
\begin{align*}
\arccos(-\frac{X}{\sqrt{X^2+Y^2}}) & = \theta - \pi, \\
\arcsin(\frac{Y}{\sqrt{X^2+Y^2}}) & = \theta - 2\pi.
\end{align*}
Therefore, we obtain the equality: \\
\[
-1 + 0 + \frac{\theta - \pi}{2\pi} + \frac{1}{2} = \frac{\theta - 2\pi}{2\pi},
\]
which verifies condition (\ref{cocycleg}). Condition (\ref{hijj}) is satisfied identically because of the way we defined $h_{ijk}$'s. \\
Condition (\ref{cocycleh}) is likewise satisfied as
\[
h_{145}(X,Y,Z) = - h_{154}(X,Y,Z),
\]
and $h_{454(X,Y,Z)} = (0,[0])$.
\item Let $i=1$, $j=3$, $k=5$, and $l=3$. Condition (\ref{gii}) is automatically satisfies by the definition of the $g_{ij}$'s. We also have:
\begin{align*}
\beta\Big(h_{135}(X,Y,Z)\Big) & = 0, \\
g_{13}(X,Y,Z) & = 0, \\
g_{35}(X,Y,Z) & = \frac{\arccos(\frac{X}{\sqrt{X^2+Y^2}})}{2\pi}, \\
g_{15}(X,Y,Z) & = \frac{\arcsin(\frac{Y}{\sqrt{X^2+Y^2}})}{2\pi}.
\end{align*}
Considering
\begin{align*}
\cos(\theta) & = \frac{X}{\sqrt{X^2+Y^2}}, \\
\sin(\theta) & = \frac{Y}{\sqrt{X^2+Y^2}},
\end{align*}
since the angle $\theta$ lies in the first quadrant, it follows that:
\begin{align*}
\arccos(\frac{X}{\sqrt{X^2+Y^2}}) & = \theta, \\
\arcsin(\frac{Y}{\sqrt{X^2+Y^2}}) & = \theta.
\end{align*}
Therefore, we obtain the equality: \\
\[
0 + 0 +  \theta = \theta,
\]
which verifies condition (\ref{cocycleg}). Condition (\ref{hijj}) is satisfied because of the way we defined $h_{ijk}$'s.
Condition (\ref{cocycleh}) is satisfied, since
\begin{align*}
h_{135}(X,Y,Z) & = \left(0,[0]\right), \qquad h_{153}(X,Y,Z) = \left(0,[0]\right), \\
h_{133}(X,Y,Z) & = \left(0,[0]\right), \qquad h_{353}(X,Y,Z) = \left(0,[0]\right).
\end{align*}
\end{itemize}
We define the total $2$-space $\mathcal{P}$ as follows: \\
(\rNum{1}) The space of objects of $\mathcal{P}$ is
\[
\mathsf{Obj}(\mathcal{P}) = (V_1 \times \mathbb{R}) \sqcup (V_2 \times \mathbb{R}) \sqcup (V_3 \times \mathbb{R}) \sqcup (V_4 \times \mathbb{R}) \sqcup (V_5 \times \mathbb{R}) \sqcup (V_6 \times \mathbb{R}).
\]
(\rNum{2}) The space of morphisms of $\mathcal{P}$ is 
\[
\mathsf{Mor}(\mathcal{P}) = \sqcup_{i,j} \Bigl((V_i \cap V_j) \times (\mathbb{Z} \times \mathbb{R}/\mathbb{Z}) \times \mathbb{R}\Bigr).
\]
For instance, $(V_1 \cap V_3) \times (\mathbb{Z} \times \mathbb{R}/\mathbb{Z}) \times \mathbb{R}$ forms a component of $\mathsf{Mor}(\mathcal{P})$. \\
(\rNum{3}) Take a morphism $\Bigl(v,(m,[a]),r\Bigr)$ in $(V_1 \cap V_3) \times (\mathbb{Z} \times \mathbb{R}/\mathbb{Z}) \times \mathbb{R}$. Its source and target are defined by:
\begin{align*}
\mathsf{s}\Bigl(v,(m,[a]),r\Bigr) & = (v,r) \in V_1 \times \mathbb{R}, \\
\mathsf{t}\Bigl(v,(m,[a]),r\Bigr) & = \biggl(v,r+\beta\Bigl((m,[a])\Bigr)+g_{1 3}^{-1}(v)\biggr) \in V_3 \times \mathbb{R}.
\end{align*}
Here we use that $\beta\Bigl((m,[a])\Bigr) = m$, and the relation
\begin{equation*}
\beta\bigl(h_{1 3 1}\bigl) + g_{1 3} + g_{3 1} = g_{1 1}.
\end{equation*}
ensures the simplification. \\
(\rNum{4}) Consider two composable arrows:
\[
\begin{tikzcd}[row sep=huge,column sep=huge]
(v,r) \arrow[r,"{\big(v,(m,[a]),r\big)}"] & (v,r+m)
\end{tikzcd}
\]
in $(V_1 \cap V_3) \times \mathbb{R} \times (\mathbb{Z} \times \mathbb{R}/\mathbb{Z})$, and
\[
\begin{tikzcd}[row sep=huge,column sep=huge]
\big(v,r+m\big) \arrow[rr,"{\big(v,(n,[b]),r+m+g^{-1}_{35}(v)\big)}"] & & \big(v,r+m+n+g^{-1}_{3 5}(v)\big)
\end{tikzcd}
\]
in $(V_3 \cap V_5) \times (\mathbb{Z} \times \mathbb{R}/\mathbb{Z}) \times \mathbb{R}$.
Using the relation
\begin{equation*}
\beta\bigl(h_{3 5 3}\bigl) + g_{3 5} + g_{5 3} = g_{3 3},
\end{equation*}
we identify $g^{-1}_{35} = g_{53}$. \\
The composition of the two arrows is then defined by:
\[
\begin{tikzcd}[row sep=huge,column sep=huge]
\big(v,r\big) \arrow[rrrr,"{\Bigl(v,(m,[a])+\alpha\bigl((n,[b]),g_{1 3}(v)\bigr)+h_{1 3 5}(v),r\Bigr)}"] & & & & \Biggl(v,r+m+g^{-1}_{1 5}(v)+\beta\biggl(\alpha\Bigl((n,[b]),g_{1 3}(v)\Bigl)+h_{1 3 5}(v)\biggr)\Biggr)
\end{tikzcd}
\]
Simplifying, we obtain
\[
\begin{tikzcd}[row sep=huge,column sep=huge]
\big(v,r\big) \arrow[rrr,"{\Bigl(v,(m+n,[a+b])+h_{1 3 5}(v),r\Bigr)}"] & & & \Biggl(v,r+m+n+g^{-1}_{1 5}(v)\Biggr)
\end{tikzcd}
\]
which lies in $(V_1 \cap V_5) \times (\mathbb{Z} \times \mathbb{R}/\mathbb{Z}) \times \mathbb{R}$. \\
We have defined sources, targets and compositions of arrows in examples, but the general formula would follow the similar approach \cite[Remark 2.16 on page 583]{MR2805195}. \\
The map $\rho:\mathcal{P} \longrightarrow S^2$ is defined by projection: \\
$\bullet$ on objects: $\rho(v,r) = v$ \\
$\bullet$ on morphisms: $\rho\big(v,(m,[a]),r\big) = v$. \\
The action of $\mathcal{U}(1)$ on $\mathcal{P}$ is given by
$$\mathcal{P} \times \mathcal{U}(1) \longrightarrow \mathcal{P}$$
$\bullet$ on objects: $(v,r) \cdot r^\prime \coloneq (v,r + r^\prime)$ \\
$\bullet$ on morphisms: $\big(v,(m,[a]),r\big) \cdot \big((n,[b]),r^\prime\big) \coloneq \big(v,(m+n,[a+b+rn]),r+r^\prime\big)$. \\
We now define local trivialisations $\phi_{i}:\mathcal{P} \big\vert_{V_i} \longrightarrow V_i \times \mathcal{G}$ and their weak inverses $\bar{\phi}_{i}:V_i \times \mathcal{G} \longrightarrow \mathcal{P} \big\vert_{V_i}$. Let $\bar{\phi}_{i}$ be the inclusion, and define $\phi_{i}$ by: \\
(\rNum{1}) on an object $(v,r) \in V_j \times \mathbb{R}$:
\[
\phi_{i}(v,r) = \big(v,r+g_{i j}(v)\big)
\]
(\rNum{2}) on a morphism $(V_j \cap V_k) \times (\mathbb{Z} \times \mathbb{R}/\mathbb{Z}) \times \mathbb{R}$:
\[\phi_{i}\Bigl(v,(m,[a]),r\Bigr) = \biggl(v,\alpha\Bigl((m,[a]),g_{i j}(v)\Bigr)+h_{i j k}(v),r+g_{i j}(v)\biggr)
\]
These maps make the following diagram commute:
\[
\begin{tikzcd}[row sep=huge,column sep=huge]
V_i \times \mathcal{G} \arrow[r,"\bar{\phi}_{i}"] \arrow[dr,"pr_1"'] & \mathcal{P} \big\vert_{V_i} \arrow[r,"\phi_{i}"] \arrow[d,"\rho"'] & V_i \times \mathcal{G} \arrow[dl,"pr_1"] \\
& V_i &
\end{tikzcd}
\]
Thus, we have constructed a principal $\mathcal{U}(1)$-bundle $\rho:\mathcal{P} \longrightarrow S^2$. The map $\rho$ will be referred to as the \textbf{categorical Hopf map}.
\end{const}
\begin{lem}\label{priGtoG'}
Let $F:\mathcal{G} \longrightarrow \mathcal{G}^\prime$ be a map of strict categorical groups, and let $\widetilde{F}:(G,H,\alpha,\beta) \longrightarrow (G^\prime,H^\prime,\alpha^\prime,\beta^\prime)$ denote the induced map between the crossed modules associated to $\mathcal{G}$ and $\mathcal{G}^\prime$. Let $X$ be a topological space, and let $(W_i,g_{ij},h_{ijk})$ be the $\mathcal{G}$-valued cocycle associated with a principal $\mathcal{G}$-bundle over $X$. Then, composing $$g_{ij}:W_i \cap W_j \longrightarrow G ~~~ \text{with} ~~~ \widetilde{F}_0:G \longrightarrow G^\prime,$$ and $$h_{ijk}:W_i \cap W_j \cap W_k \longrightarrow H ~~~ \text{with} ~~~ \widetilde{F}_1:H \longrightarrow H^\prime,$$ yields a principal $\mathcal{G}^\prime$-bundle over $X$.
\end{lem}
\begin{proof}
Define $$g^\prime_{ij} = \widetilde{F}_0 \circ g_{ij}, ~~~~~ h^\prime_{ijk} = \widetilde{F}_1 \circ h_{ijk}.$$
We must verify that these data satisfy the cocycle conditions for a $\mathcal{G}^\prime$-bundle:
\begin{align*}
g^\prime_{ii} & = e_{G^\prime} \\
\beta\bigl(h^\prime_{ijk}\bigr) \cdot g^\prime_{ij} \cdot g^\prime_{jk} & = g^\prime_{ik} \\
h^\prime_{ijj} = h^\prime_{jji} & = e_{H^\prime} \\
h^\prime_{ijk} \cdot h^\prime_{ikl} & = \alpha(g^\prime_{ij},h^\prime_{jkl}) \cdot h^\prime_{ijl}
\end{align*}
Since $\widetilde{F}$ is a crossed module homomorphism, these relations follow directly by applying $\widetilde{F}_0$ to Equations (\ref{gii}), (\ref{cocycleg}) and $\widetilde{F}_1$ to Equations (\ref{hijj}), (\ref{cocycleh}). Hence the claim follows.
\end{proof}
\begin{lem}\label{primathGtoG}
Let $G$ be a group, viewed as a strict categorical group $\mathcal{G}$ with objects as $G$ and only identity morphisms. Let $X$ be a topological space with an open cover $\left\{W_i\right\}_{i \in I}$. Then, constructing a principal $\mathcal{G}$-bundle over $X$ yields a principal $G$-bundle over $X$.
\end{lem}
\begin{proof}
Since both the group of objects and morphisms of $\mathcal{G}$ are identified with $G$, the crossed module associated with $\mathcal{G}$ is given by $(G,\left\{\mathbbm{1}\right\},\alpha,\beta)$, where the homomorphism $\alpha$ and $\beta$ are taken to be the trivial ones. Consider the smooth maps
\[
g_{ij}:W_i \cap W_j \longrightarrow G, ~~~~~
h_{ijk}:W_i \cap W_j \cap W_k \longrightarrow \left\{\mathbbm{1}\right\},
\]
satisfying cocycle conditions (\ref{gii}), (\ref{cocycleg}),  (\ref{hijj}) and \ref{cocycleh} in Definition \ref{gvaluedcocycle}. \\
We construct the total categorical space $\mathcal{Q}$ as follows: \\
(\rNum{1}) Objects:
\[
\mathsf{Obj}(\mathcal{Q}) = \sqcup_{i} \bigl(W_i \times G\bigr),
\]
(\rNum{2}) Morphisms:
\[
\mathsf{Mor}(\mathcal{Q}) = \sqcup_{i,j} \Bigl((W_i \cap W_j) \times \left\{\mathbbm{1}\right\} \times G\Bigr).
\]
For an arrow
\[
\begin{tikzcd}
(w,g) \arrow[rr,"{\bigl(w,\mathbbm{1},g\bigr)}"] & & (w,g^\prime)
\end{tikzcd}
\]
in $(W_i \times W_j) \times \left\{\mathbbm{1}\right\} \times G$, we have the relation $$g^\prime = g^{-1}_{ij}(w) \cdot g.$$
Equivalently, $$g = g_{ij} \cdot g^\prime,$$
which matches precisely the identification of objects in the total space of a principal $G$-bundle. Since the maps $g_{ij}$ satisfy the cocycle condition
\[
g_{ij} \cdot g_{jk} = g_{ik},
\]
the collection $(W_i,g_{ij})$ defines transition functions for a principal $G$-bundle over $X$.
\end{proof}
\begin{defn}\label{curlcirtocir}
Define $\widetilde{F}_0:\mathbb{R} \longrightarrow S^1$ and $\widetilde{F}_1:\mathbb{Z} \times \mathbb{R}/\mathbb{Z} \longrightarrow \left\{\mathbbm{1}\right\}$ to be given by
\begin{align*}
\widetilde{F}_0 & : r \mapsto q_0(r) = \bigl(\cos(2\pi r),\sin(2\pi r)\bigr) \\
\widetilde{F}_1 & : \bigl(m,[a]\bigr) \mapsto \left\{\mathbbm{1}\right\}.
\end{align*}
\end{defn}
\begin{prop}\label{giveHopf}
Take the cover of $S^2$ as in Construction \ref{cathopf}, and let $g_{ij}$ and $h_{ijk}$ be the smooth maps defined therein. Then $$g^\prime_{ij} = \widetilde{F}_0 \circ g_{ij}, ~~~~~  h^\prime_{ijk} = \widetilde{F}_1 \circ h_{ijk},$$
determines the transition data of the Hopf map $\eta:S^3 \longrightarrow S^2$.
\end{prop}
\begin{proof}
By Lemmas \ref{priGtoG'} and \ref{primathGtoG}, the maps $g^\prime_{ij}$ and $h^\prime_{ijk}$ provide the cocycle data for a principal $U(1)$-bundle over $S^2$. It therefore suffices to show that the associated total categorical space $\mathcal{Q}$, obtained using $g^\prime_{ij}$ and $h^\prime_{ijk}$, is equivalent to $S^3$. The categorical space $\mathcal{Q}$ is given by: \\
(\rNum{1}) objects:
\[
\mathsf{Obj}(\mathcal{Q}) = \sqcup_{i} (V_i \times S^1)
\]
(\rNum{2}) morphisms:
\[
\mathsf{Mor}(\mathcal{Q}) = \sqcup_{i,j} \bigl((V_i \cap V_j) \times \left\{\mathbbm{1}\right\} \times S^1\bigr)
\]
Now, if we reconstruct the Hopf map $\eta$ by regarding $U(1)$ as a strict categorical group, we obtain a categorical bundle
$$\widetilde{\eta}:\mathcal{Q}^\prime \longrightarrow S^2.$$
Take the cover $\left\{V_N,V_S\right\}$ of $S^2$, where
\begin{align*}
V_N & = S^2 \setminus \left\{Z_N = (0,0,-1)\right\}, \\
V_S & = S^2 \setminus \left\{Z_S = (0,0,1)\right\}.
\end{align*}
Then
\begin{align*}
\mathsf{Obj}(\mathcal{Q}^\prime) & = (V_N \times S^1) \sqcup (V_S \times S^1),  \\
\mathsf{Mor}(\mathcal{Q}^\prime) & = (V_N \cap V_N \times \left\{\mathbbm{1}\right\} \times S^1) \sqcup (V_N \cap V_S \times \left\{\mathbbm{1}\right\} \times S^1) (V_S \cap V_N \times \left\{\mathbbm{1}\right\} \times S^1) \sqcup (V_S \cap V_S \times \left\{\mathbbm{1}\right\} \times S^1) \\
& = (V_N \times \left\{\mathbbm{1}\right\} \times S^1) \sqcup (V_N \cap V_S \times \left\{\mathbbm{1}\right\} \times S^1) \sqcup (V_S \cap V_N\times \left\{\mathbbm{1}\right\} S^1) \sqcup (V_S \times \left\{\mathbbm{1}\right\} S^1).
\end{align*}
The categorical space $\mathcal{Q}^\prime$ is equivalent to $S^3$ as a categorical space. Thus it remains to prove that $\mathcal{Q}$ is equivalent to $\mathcal{Q}^\prime$. According to \cite[Definition 2.20 on page 589 and Theorem 2.22 on page 590]{MR2805195}, we need to show that there exists a semi-strict bundle morphism
\[
\mathcal{Q}^ \prime \longrightarrow \mathcal{Q}.
\]
The required bundle morphism is defined in the trivial manner at the level of objects and morphism. Hence the result follows, yielding the following Morita equivalence diagram
\[
\begin{tikzcd}
& Q^\prime \arrow[dl] \arrow[dr] & \\
Q & & S^3
\end{tikzcd}.
\]
\end{proof}
\begin{prop}\label{modbu}
Consider the action of $\mathbb{B}U(1)$ on the total categorical space $\mathcal{P}$ of the categorical Hopf map,
\[
\mathcal{P} \times \mathbb{B}U(1) \longrightarrow \mathcal{P}
\]
given by \\
(\rNum{1}) on objects:
\[
\mathsf{Obj}(\mathcal{P}) \times \left\{\mathbbm{1}\right\} \longrightarrow \mathcal{P}, ~~ \bigl((v,r),\mathbbm{1}\bigr) \mapsto (v,r)
\]
(\rNum{2}) on morphisms:
\[
\mathsf{Mor}(\mathcal{P}) \times \mathbb{R}/\mathbb{Z} \longrightarrow \mathsf{Mor}(\mathcal{P}), ~~ \Bigl(\bigl(v,(m,[a]),r\bigr),[b]\Bigr) \mapsto \bigl(v,(m,[a+b]),r\bigr)
\]
Then the quotient\footnote{See \cite{MR2709030} for the definition of $2$-quotient.}categorical space $\mathcal{P}/\mathbb{B}U(1)$ is equivalent, as a categorical space, to $S^3$.
\end{prop}
\begin{proof}
Write $\mathcal{Q}^{\prime\prime} = \mathcal{P}/\mathbb{B}U(1)$. The group of objects $\mathsf{Obj}(\mathcal{Q}^{\prime\prime})$ is
\[
\mathsf{Obj}(\mathcal{Q}^{\prime\prime}) = \sqcup_i (V_i \times \mathbb{R}).
\]
The group of morphims $\mathsf{Mor}(\mathcal{Q}^{\prime\prime})$ is
\[
\mathsf{Mor}(\mathcal{Q}^{\prime\prime}) = \sqcup_{i,j} \Bigl((V_i \cap V_j) \times \mathbb{Z} \times \mathbb{R}\Bigr).
\]
By Proposition \ref{giveHopf} and the fact that the action groupoid $[\mathbb{Z} \times \mathbb{R} \rightrightarrows \mathbb{R}]$ is equivalent to the groupoid $[S^1 \rightrightarrows S^1]$, it follows that $\mathcal{Q}^{\prime\prime}$ presents the same categorical space as the total space of the Hopf, namely $S^3$. Hence $\mathcal{P}/\mathbb{B}U(1) \simeq S^3$ as categorical spaces.
\end{proof}
\begin{rem}\label{spacey2andl}
As observed in Construction \ref{cathopf}, the space of objects of our categorical bundle is
\[
Y = \sqcup_i (V_i \times \mathbb{R}).
\]
From this, one can construct the basic bundle gerbe using the cover $\left\{V_i \times S^1\right\}_{i=1}^6$ of $S^3$ by six solid tori. This cover arises naturally from the construction of the Hopf map with respect to the cover $\left\{V_i\right\}_{i=1}^6$ of $S^2$. Each solid torus is embedded into $S^3$ by means of a local section of the Hopf map: if $s_i:V_i \longrightarrow S^3$ denotes a local section, then the embedding 
$$f_i:V_i \times S^1 \longrightarrow S^3$$
is given by
\[
f_i(v,z) = s_i(v) \cdot z, ~~~~~ z = \cos(t) + \sin(t) \mathsf{k}.
\]
These embeddings are compatible with the Hopf map, since
\[
\eta\left(f_i(v,z)\right) = v.
\]
We thus define the map $\pi:Y \longrightarrow S^3$ to be induced by the maps $f_i$. Consequently,
\[
Y^{[2]} = \sqcup_{i,j} \left((V_i \cap V_j) \times \mathbb{R} \times \mathbb{Z}\right).
\]
The associated $U(1)$-bundle $L$ over $Y^{[2]}$ is given by
\[
L = \sqcup_{i,j} \left((V_i \times V_j) \times \mathbb{R} \times \mathbb{Z} \times S^1\right).
\]
Hence, the categorical space $\mathcal{Q}^{\prime\prime}$ in Proposition \ref{modbu} is equivalent to the categorical space $\mathcal{Y}$, where $\mathsf{Obj}(\mathcal{Y}) = Y$ and $\mathsf{Mor}(\mathcal{Y}) = Y \times_{S^3} Y$.
\end{rem}
\begin{obs}
The categorical space $\mathcal{Q}$ in Proposition \ref{giveHopf} is equivalent to the categorical space $\mathcal{Q}^{\prime\prime}$ in Proposition \ref{modbu}. This follows from similar arguments as in the proof of Proposition \ref{giveHopf}.
\end{obs}
\begin{obs}\label{mapH}
The map\footnote{See \cite{MR2709030} for the definition of a $2$-bundle over a $2$-space.} $$\widetilde{H}:\mathcal{P} \longrightarrow \mathcal{P}/\mathbb{B}U(1)$$, defined by \\
(\rNum{1}) on objects: $\widetilde{H}_0(v,r) = (v,r)$ \\
(\rNum{2}) on morphisms: $\widetilde{H}_1\bigl((v,(m,[a])),r\bigr) = (v,m,r)$ \\
is a non-trivial principal $\mathbb{B}U(1)$-bundle.
\begin{proof}
Equivalently, by Remark \ref{butobunandconv}, we show that the associated bundle gerbe $(L,Y,S^3)$ is non-trivial. By definition, the space of objects of the bundle gerbe is
\[
Y = \mathsf{Obj}(\mathcal{P}) = \sqcup_i (V_i \times \mathbb{R}),
\]
with projection $\pi:Y \longrightarrow S^3$ as described in Remark \ref{spacey2andl}. The space of morphisms is
$$L = \mathsf{Mor}(\mathcal{P}) = \sqcup_{i,j} \bigl((V_i \cap V_j) \times (\mathbb{Z} \times \mathbb{R}/\mathbb{Z}) \times \mathbb{R}\bigr).$$
The map $\lambda:L \longrightarrow Y^{[2]}$ is given by:
\[
\lambda\bigl(v,(m,[a]),r\bigr) = \left((v,r),(v,r+m)\right).
\]
The bundle gerbe product
$$m:\pi_{12}^*L \otimes \pi_{23}^*L \longrightarrow \pi_{13}^*L$$
is defined case by case, as in Remark \ref{butobunandconv}, using the identifications of $L$ and $Y^{[2]}$ described in Remark \ref{spacey2andl}. To check the non-triviality, take the cover $\left\{V_i \times S^1\right\}_{i=1}^6$ of $S^3$, obtained from the cover $\left\{V_i\right\}_{i=1}^6$ of $S^2$. Define local sections of $\pi$ by
\begin{align*}
sec_1&:V_1 \times S^1 \longrightarrow Y, ~~ (v,e^{2\pi\theta\mathsf{i}}) \longmapsto (v,\theta), \\
sec_4&:V_4 \times S^1 \longrightarrow Y, ~~ (v,e^{2\pi\theta^\prime\mathsf{i}}) \longmapsto (v,\theta^\prime), \\
sec_5&:V_5 \times S^1 \longrightarrow Y, ~~ (v,e^{2\pi\theta^{\prime\prime}\mathsf{i}}) \longmapsto (v,\theta^{\prime\prime}).
\end{align*}
These induce maps
\begin{align*}
s_{14}&:(V_1 \times S^1) \cap (V_4 \times S^1) \longrightarrow Y^{[2]}, ~~ (v,z) \longmapsto (v,\theta,\theta^\prime-\theta), \\
s_{45}&:(V_4 \times S^1) \cap (V_5 \times S^1) \longrightarrow Y^{[2]}, ~~ (v,z^\prime) \longmapsto (v,\theta^\prime,\theta^{\prime\prime}-\theta^\prime), \\
s_{15}&:(V_1 \times S^1) \cap (V_5 \times S^1) \longrightarrow Y^{[2]}, ~~ (v,z) \longmapsto (v,\theta,\theta^{\prime\prime}-\theta).
\end{align*}
Pulling back $L$ along these maps yields line bundles $L_{14}$, $L_{45}$ and $L_{15}$. Choosing unit sections $\delta_{14}$, $\delta_{45}$ and $\delta_{15}$, the gerbe product gives the relation
\[
m(\delta_{14},\delta_{45}) = g_{145} \delta_{15}.
\]
for some $U(1)$-valued function
\[
g_{145}:(V_1 \times S^1) \cap (V_4 \times S^1) \cap (V_5 \times S^1) \longrightarrow U(1)
\]
This determines a class $[g] \in \check{H}^2\left(S^3,\underline{U(1)}\right)$. The connecting homomorphism 
$$\check{H}^2\bigl(S^3,\underline{U(1)}\bigr) \longrightarrow \check{H}^3(S^3,\underline{\mathbb{Z}})$$
is obtained via the Snake lemma applied to the diagram
\[
\begin{tikzcd}
& \frac{\check{C}^2\bigl(S^3,\underline{\mathbb{Z}}\bigr)}{\mathsf{Im}(\delta^1_{\underline{\mathbb{Z}}})} \arrow[r] \arrow[d,"{\delta^2_{\underline{\mathbb{Z}}}}"'] & \frac{\check{C}^2\bigl(S^3,\underline{\mathbb{R}}\bigr)}{\mathsf{Im}(\delta^1_{\underline{\mathbb{R}}})} \arrow[r,"p"] \arrow[d,"\delta^2_{\underline{\mathbb{R}}}"'] & \frac{\check{C}^2\bigl(S^3,\underline{U(1)}\bigr)}{\mathsf{Im}(\delta^1_{\underline{U(1)}})} \arrow[r] \arrow[d,"\delta^2_{\underline{U(1)}}"'] & 0 \\
0 \arrow[r] & \mathsf{Ker}(\delta^3_{\underline{\mathbb{Z}}}) \arrow[r,"i"'] & \mathsf{Ker}(\delta^3_{\underline{\mathbb{R}}}) \arrow[r] & \mathsf{Ker}(\delta^3_{\underline{U(1)}}) &
\end{tikzcd}
\]
and sends $[g]$ to $1 \in H^3(S^3;\mathbb{Z})$. Hence, the bundle gerbe $(L,Y,S^3)$ is non-trivial.
\\
Finally, note that the construction of $g_{145}$ depends on the morphisms and their compositions in $\mathcal{P}$, as detailed in Construction \ref{cathopf}.
\end{proof}
\end{obs}
\begin{obs}
The bundle gerbes constructed in Examples \ref{murbun} and \ref{ourbun}, together with the bundle gerbe arising in Observation \ref{mapH}, are all equivalent. Indeed, as we have verified in each case, their Dixmier–Douady classes coincide: each represents the generator $1$ in $H^3(S^3,\mathbb{Z})$.
\end{obs}
\begin{rem}\label{factorizationdiagram}
Considering the maps $\eta:S^3 \longrightarrow S^2$, $\rho:\mathcal{P} \longrightarrow S^2$, and $\widetilde{H}:\mathcal{P} \longrightarrow \mathcal{P}/\mathbb{B}U(1)$, we obtain the commutative diagram:
\begin{equation}\label{factordiag}
\begin{tikzcd}
\mathcal{P} \arrow[rr,bend left=45,"H"] \arrow[r,"\widetilde{H}"] \arrow[dr,"\rho"'] & \mathcal{P}/\mathbb{B}U(1) \arrow[d,"\widetilde{\eta}"'] \arrow[r,"\simeq"] & S^3 \arrow[dl,"\eta"] \\
& S^2 &
\end{tikzcd}
\end{equation}
Here, $\widetilde{\eta}$ denotes the categorical bundle with structure categorical group $U(1)$ associated to the Hopf map $\eta$, constructed as in Lemma \ref{primathGtoG}. By Proposition \ref{modbu}, the categorical space $\mathcal{P}/\mathbb{B}U(1)$ is equivalent to $S^3$. The map $H$ is then defined as the composition of this equivalence with $\widetilde{H}$. Consequently, both $H$ and $\widetilde{H}$ yield equivalent constructions of the basic bundle gerbe over $S^3$.
\end{rem}
\begin{obs}\label{nontri2bun}
We show that the map $\rho:\mathcal{P} \longrightarrow S^2$ defines a non-trivial principal $\mathcal{U}(1)$-bundle, using Proposition \ref{giveHopf}. \\
First, recall the exact sequence\footnote{See \cite{MR3764535}.}of categorical groups
\[
1 \longrightarrow \mathbb{B}U(1) \longrightarrow \mathcal{U}(1) \longrightarrow U(1) \longrightarrow 1,
\]
which is classified by \cite[Theorem 99 on page 662]{MR2800361} as an element of
\begin{align*}
H^3_{SM}(U(1);\mathbb{Z}) & \cong H^4(BU(1);\mathbb{Z}) \\
& \cong \mathbb{Z}.
\end{align*}
From this extension we obtain the diagram
\[
\begin{tikzcd}
B\big\vert \mathbb{B}U(1) \big\vert \arrow[r] & B \big\vert \mathcal{U}(1) \big\vert \arrow[d] \\
S^2 \arrow[ur,dashed] \arrow[r] & BU(1) \arrow[d] \\
& K(\mathbb{Z},4)
\end{tikzcd}
\]
where $\big\vert - \big\vert$ denotes the geometric realisation\footnote{See \cite{MR2597732}.}of the nerve a categorical group. \\
Note that
\begin{align*}
BU(1) & \simeq K(\mathbb{Z},2) \\
B\big\vert \mathbb{B}U(1) \big\vert & \simeq K(\mathbb{Z},3).
\end{align*}
Thus, we have the extended diagram
\[
\begin{tikzcd}
K(\mathbb{Z},3) \arrow[r,"i"] & B \big\vert \mathcal{U}(1) \big\vert \arrow[r,"f"] & K(\mathbb{Z},2) \arrow[r,"g"] & K(\mathbb{Z},4) \\
& S^2 \arrow[ul,"\alpha"] \arrow[u,dashed,"\beta"] \arrow[ur,"\gamma"] \arrow[urr,"\lambda"'] & & 
\end{tikzcd}
\]
Here: \\
$\bullet$ the map $f$ is induced by applying the functor $B(-)$ to the projection $\mathcal{U}(1) \longrightarrow U(1)$. \\
$\bullet$ the map $g$ classifies this as a $K(\mathbb{Z},3)$-bundle. \\
$\bullet$ the maps $\beta$ and $\gamma$ classify the categorical Hopf map and the Hopf map, respectively. \\
By \cite[Lemma 3 on page 17]{MR2597732}, we obtain the exact sequence
\[
[S^2 \longrightarrow K(\mathbb{Z},3)] \longrightarrow [S^2 \longrightarrow B\big\vert\mathcal{U}(1)\big\vert] \longrightarrow [S^2 \longrightarrow K(\mathbb{Z},2)] \longrightarrow [S^2 \longrightarrow K(\mathbb{Z},4)]
\]
Therefore, the categorical Hopf map is a lift of the Hopf map, and we deduce that
\[
\check{H}^1(S^2,\mathcal{U}(1)) \cong \mathbb{Z}.
\]
In particular, the associated principal $\mathcal{U}(1)$-bundle in non-trivial, and is classified by the positive generator $1$.
\end{obs}

\section{The categorical group \texorpdfstring{$String(3)$}{String(3)}}
In this section, we conjecture how the categorical group $String(3)$ can be interpreted as the categorical group of symmetries of the categorical Hopf map. To motivate this perspective, we first examine the classical setting by considering the group of symmetries of the Hopf map.
\begin{rem}
When we speak of "symmetries" of the Hopf map, we mean pairs $(A,\psi)$, where $A \in SO(3)$ and $\psi:S^3 \longrightarrow S^3$, such that the diagram
\[
\begin{tikzcd}
S^3 \arrow[r,"\psi"] \arrow[d,"\eta"'] & S^3 \arrow[d,"\eta"] \\
S^2 \arrow[r,"A"'] & S^2
\end{tikzcd}
\]
commutes and, in addition $\psi$ preserves the chosen Ehresmann connection on $\eta$.
\end{rem}
\begin{prop}\cite[Proposition 4.5.3 on page 52]{MR900587}\label{secondprop}
Let $M$ be a simply-connected smooth manifold, and let $\frac{\omega}{2\pi}$ be an integral closed $2$-form on $M$. We have the following statements: \\
(\rNum{1}) There exists a principal $U(1)$-bundle $\pi:P \longrightarrow M$ with a connection $\zeta$ whose curvature is $\omega$. \\
(\rNum{2}) Given two principal $U(1)$-bundles $\pi:P \longrightarrow M$ and $\pi^\prime:P^\prime \longrightarrow M$ with connections $\zeta$ and $\zeta^\prime$ with the same
curvature $\omega$, there exists a morphism\footnote{See \cite[on page 571]{MR2805195}.} $\psi:P \longrightarrow P^\prime$ of principal $U(1)$-bundles such that $[\psi^*\zeta^\prime] = [\zeta]$. Moreover, the morphism $\psi$ is unique up to multiplication by elements of $U(1)$.
\end{prop}
\begin{prop}\cite[Proposition 4.4.2 on page 47]{MR900587}\label{firstprop} Let $G$ be a Lie group acting smoothly on a simply-connected smooth manifold $M$, leaving an integral closed $2$-from $\frac{\omega}{2\pi}$ on $M$ invariant. Then, there exists a central extension $\widetilde{G}$ of $G$ by $U(1)$ canonically associated to the pair $(M,\omega)$.
\end{prop}
\begin{rem}
In \cite[following Propositions \ref{secondprop} and \ref{firstprop}]{MR900587}, Pressley and Segal present three equivalent constructions of the extension $\widetilde{G}$. At the end of this section, we provide the categorical analogues of these constructions.
\end{rem}
We now show that the symmetries of the Hopf map can be identified with the Lie group
\[
Spin^c(3) = Spin(3) \times_{\left\{\pm 1\right\}} U(1)
\]
by means of Propositions \ref{secondprop} and \ref{firstprop}.
\begin{const}\label{symmofhopf} We now specialise Propositions \ref{firstprop} and \ref{secondprop} to the case
$$M = S^2, ~~~~~ P = S^3, ~~~~~ G = U(1), ~~~~~ \pi = \eta,$$
where $\eta$ denotes the Hopf map.
Consider the volume form $\omega^\prime$ on $S^2$:
\begin{equation}
\omega^\prime = X dY \wedge dZ + Y dZ \wedge dX + Z dX \wedge dY.
\end{equation}
with coordinates $(X,Y,Z)$ on $S^2$. The form $\omega^\prime$ is invariant under the action of $SO(3)$ on $S^2$, since this group preserves the Euclidean volume of the sphere. \\
In spherical coordinates $(r,\theta,\phi)$, we have
\[
X = \cos(\theta) \sin(\phi), \qquad
Y = \sin(\theta) \sin(\phi), \qquad
Z = \cos(\phi),
\]
for $0 \leq \theta < 2\pi$ and $0 \leq \phi \leq \pi$.
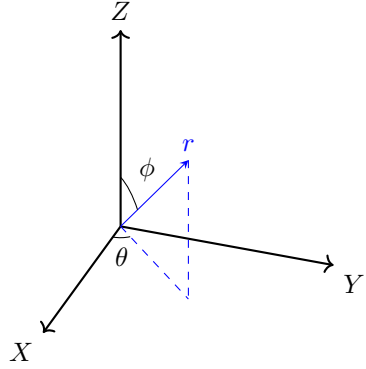
\begin{figure}[H]
\centering
\begin{tikzpicture}[scale=1.5,tdplot_main_coords]
    \coordinate (O) at (0,0,0);
    \draw[thick,->] (0,0,0) -- (2,0,0) node[anchor=north east]{$X$};
    \draw[thick,->] (0,0,0) -- (0,2,0) node[anchor=north west]{$Y$};
\draw[thick,->] (0,0,0) -- (0,0,2) node[anchor=south]{$Z$};
    \tdplotsetcoord{r}{\rvec}{\thetavec}{\phivec}
    \draw[-stealth,color=blue] (O) -- (r) node[above] {$r$};
    \draw[dashed, color=blue] (O) -- (rxy);
    \draw[dashed, color=blue] (r) -- (rxy);
    \tdplotdrawarc{(O)}{0.2}{0}{\thetavec}{anchor=north}{$\theta$}
    \tdplotsetthetaplanecoords{\thetavec}
    \tdplotdrawarc[tdplot_rotated_coords]{(0,0,0)}{0.5}{0}%
        {\phivec}{anchor=south west}{$\phi$}
\end{tikzpicture}
\caption{Spherical coordinate system $(r,\theta,\phi)$}
\end{figure}
\noindent
Differentiating gives
\begin{align*}
dX & = -\sin(\theta) \sin(\phi) d\theta + \cos(\theta) \cos(\phi) d\phi, \\
dY & = \cos(\theta) \sin(\phi) d\theta + \sin(\theta) \cos(\phi) d\phi, \\
dZ & = -\sin(\phi) d\phi.
\end{align*}
Substituting into $\omega^\prime$ yields
\begin{align*}
\omega^\prime & = \cos(\theta) \sin(\phi)\left(\cos(\theta) \sin(\phi) d\theta + \sin(\theta) \cos(\phi) d\phi\right) \wedge \left(-\sin(\phi) d\phi\right) \\
& + \sin(\theta) \sin(\phi) \left(-\sin(\phi) d\phi\right) \wedge \left(-\sin(\theta) \sin(\phi) d\theta + \cos(\theta) \cos(\phi) d\phi\right) \\
& + \cos(\phi)\left(-\sin(\theta) \sin(\phi) d\theta + \cos(\theta) \cos(\phi) d\phi\right) \wedge \left(\cos(\theta) \sin(\phi) d\theta + \sin(\theta) \cos(\phi) d\phi\right) \\
& = -\cos^2(\theta)\sin^3(\phi) d\theta \wedge d\phi + \sin^2(\theta) \sin^3(\phi) d\phi \wedge d\theta \\
& - \sin^2(\theta)\sin(\phi)\cos^2(\phi) d\theta \wedge d\phi + \cos^2(\theta)\sin(\phi)\cos^2(\phi) d\phi \wedge d\theta \\
& = \sin^3(\phi) d\phi \wedge d\theta + \sin(\phi) \cos^2(\phi) d\phi \wedge d\theta \\
& =  \sin(\phi) d\phi \wedge d\theta.
\end{align*}
Integrating,
\begin{align*}
\int\limits_0^{2\pi}\int\limits_0^{\pi} \sin(\phi) d\phi d\theta & = \int\limits_0^{2\pi}[-\cos(\phi)]_0^\pi d\theta \\
& = 4\pi.
\end{align*} 
Since the Hopf map is classified by $1 \in H^2(S^2;\mathbb{Z})$, in de Rham cohomology this corresponds to the normalised volume form $\omega^\prime/4\pi$. Thus we take an Ehresmann connection $\zeta$ on the Hopf map with curvature
\[
\omega = \frac{\omega^\prime}{2}.
\]
so that $\omega/2\pi$ is an integral closed $2$-form representing the generator of $H^2(S^2;\mathbb{Z})$. Notice that the form $\omega/2\pi$ represents the image of the generator of the Hopf map in real cohomology. \\
We now define the connection $1$-form $\zeta$ on the Hopf map as follows\footnote{See \cite[on page 269]{MR3012377}. See also Remark \ref{connection_on_Hopf_just}.}:
\begin{equation}\label{connonHopf}
\zeta = -ydx + xdy -wdz + zdw,
\end{equation}
where $(x,y,z,w)$ denote the standard coordinates on $S^3 \subset \mathbb{R}^4$. \\
Let $\widetilde{SO(3)}$ denote the central extension of $SO(3)$ by $U(1)$ obtained via Proposition \ref{firstprop}. An equivalent construction of $\widetilde{SO(3)}$ is given by Proposition \ref{secondprop}. In fact, for a general compact Lie group $G$, the two constructions agree; see \cite[Page 54]{MR900587} for details. \\
For each $A \in SO(3)$, consider the diffeomorphism $A:S^2 \longrightarrow S^2$ given by left multiplication by $A$. Pulling back the Hopf map $\eta:S^3 \longrightarrow S^2$ along $A$ yields a principal $U(1)$-bundle
$$A^* \eta:A^* S^3 \longrightarrow S^2$$
equipped with the pulled-back connection $A^\ast \zeta$, whose curvature is again $\omega$. \\
By Proposition \ref{secondprop}, we may define $\widetilde{SO(3)}$ as the group of pairs $(A,\psi)$, where $A \in SO(3)$ and $\psi:S^3 \longrightarrow A^\ast S^3$ is a $U(1)$-bundle morphism. For each $A$, there is a circle’s worth of such morphisms $\psi$, arising from the fibrewise action of $U(1)$ and the connections on these bundles.
\[
\begin{tikzcd}
S^3 \arrow[rr,"\psi"] \arrow[dr,"\eta"'] & & A^\ast S^3 \arrow[dl] \\
& S^2 &
\end{tikzcd}
\]
Each such $\psi$ can be reinterpreted as a map $\phi:S^3 \longrightarrow S^3$ satisfying 
$$\phi^* \zeta = \zeta,$$
so that the diagram
\[
\begin{tikzcd}
S^3 \arrow[r,"\phi"] \arrow[d,"\eta"'] & S^3 \arrow[d,"\eta"] \\
S^2 \arrow[r,"A"'] & S^2
\end{tikzcd}
\]
commutes. \\
To see this equivalence, suppose we are given $\psi:S^3 \longrightarrow A^\ast S^3$. Consider the pullback diagram 
\[
\begin{tikzcd}
A^\ast S^3 \arrow[r,"\lambda"] \arrow[d,"\theta"'] & S^3 \arrow[d,"\eta"] \\
S^2 \arrow[r,"A"'] & S^2
\end{tikzcd}
\]
where $\lambda$ is the canonical projection. Define $\phi \coloneq \lambda \circ \psi$. Then
\begin{align*}
\phi^\ast\zeta & = \left(\lambda \circ \psi\right)^\ast\zeta \\
& = \psi^\ast \circ \lambda^\ast \zeta \\
& = \psi^\ast \zeta^\prime \\
& = \zeta.
\end{align*}
where $\zeta^\prime$ denotes the connection $1$-form on $A^\ast S^3$. Thus $\phi$ preserves the connection and exhibits the required compatibility with $A$. \\
Now, given a map $\phi:S^3 \longrightarrow S^3$ with the above properties, the universal property of pullbacks provides a map $\psi:S^3 \longrightarrow A^\ast S^3$. Concretely, we have the pullback diagram
\[
\begin{tikzcd}
S^3 \arrow[drr,bend left=45,"\phi"] \arrow[ddr,bend right=45,"\eta"'] \arrow[dr,dashed,"\phi"] & & \\
& A^\ast S^3 \arrow[r,"\lambda"] \arrow[d,"\theta"'] & S^3 \arrow[d,"\eta"] \\
& S^2 \arrow[r,"A"'] & S^2
\end{tikzcd}
\]
so that
\begin{align*}
\psi^\ast\zeta^\prime & = \psi^\ast\left(\lambda^\ast\zeta\right) \\
& = \left(\lambda \circ \psi\right)^\ast \zeta \\
& = \phi^\ast\zeta \\
& = \zeta.
\end{align*}
Fix the base-point $Y_N = (0,1,0) \in S^2$ and choose $a^\prime = \bigl(0,\tfrac{\sqrt{2}}{2},\tfrac{\sqrt{2}}{2},0\bigr) \in \eta^{-1}(Y_N)$ as a base-point of $S^3$. Different choices of base-points lead to equivalent constructions. After fixing these choices, a morphism $\psi$ with the above conditions is uniquely determined by $A \in SO(3)$ together with $\psi(a^\prime) \in S^3$. \\
Thus, as a topological space, $\widetilde{SO(3)}$ may be identified with the fibre product
\[
\begin{tikzcd}
\widetilde{SO(3)} \arrow[r] \arrow[d] & S^3 \arrow[d,"\eta"] \\
SO(3) \arrow[r,"\mu"'] & S^2
\end{tikzcd}
\]
where $\mu(A) = A\eta(a^\prime) = AY_N$. Equivalently,
\[
\widetilde{SO(3)} = \left\{(A,b^\prime) \in SO(3) \times S^3 ~ \big\vert ~ \eta(b^\prime)=AY_N\right\}.
\]
Central extensions of $SO(3)$ by $U(1)$ are classified by the Segal–Mitchison\footnote{See for example, \cite[Section 4]{MR2800361}.}cohomology group $H^2_{SM}\bigl(SO(3);U(1)\bigr)$. \\
By \cite[Corollary 97]{MR2800361},
\begin{align*}
H^2_{SM}\bigl(SO(3);U(1)\bigr) & \cong H^3_{SM}\bigl(SO(3);\mathbb{Z}\bigr) \\
&\cong H^3\bigl(BSO(3);\mathbb{Z}\bigr) \\
& \cong \mathbb{Z}_2.
\end{align*}
Hence there exist exactly two isomorphism classes of principal $U(1)$-bundles over $SO(3)$: the trivial bundle $SO(3)\times U(1)$ and the non-trivial extension, which is $Spin^c(3)$. \\
There are two ways to show that $SO(3)\times U(1)$ is not the same as $\widetilde{SO(3)}$: \\
$(1)$ they are not homeomorphic as topological spaces, and \\
$(2)$ they are not homomorphic as groups. \\
For ($1$), consider the fibre sequence
\[
0 \longrightarrow U(1) \simeq SO(2) \longrightarrow  SO(3) \overset{\mu}\longrightarrow S^2 \longrightarrow 0.
\]
The Leray–Serre spectral sequence associated to this fibration (see, for example, \cite[Chapter 3 on page 177]{MR658304}) has second page
\[
E^{p,q}_2 = H^P\left(S^2,H^q\left(U(1);\mathbb{Z}\right)\right) \Longrightarrow H^{p+q}\left(SO(3);\mathbb{Z}\right).
\]
Therefore, the $E_2$-page has the form
\[
\begin{tikzpicture}
  \matrix (m) [matrix of math nodes,
    nodes in empty cells,nodes={minimum width=5ex,
    minimum height=5ex,outer sep=-5pt},
    column sep=1ex,row sep=1ex]{
                &      &     &     & \\
          1     &  \mathbb{Z} &  0  & \mathbb{Z} & \\
          0     &  \mathbb{Z}  & 0 &  \mathbb{Z}  & \\
    \quad\strut &   0  &  1  &  2  & \strut \\};
 \draw[-stealth] (m-2-2.south east) -- (m-3-4.north west);
\draw[thick,->] (m-4-1.east) -- (m-1-1.east) ;
\draw[thick,->] (m-4-1.north) -- (m-4-5.north) ;
\draw (2.2,-1.1) node[below] {p};
\draw (-1.5,1.9) node[left] {q};
\end{tikzpicture}
\]
and all other rows and columns vanish. The only non-trivial differential is
$$d_2^{0,1}:\mathbb{Z} \cong H^0(S^2;\mathbb{Z}) \longrightarrow H^2(S^2;\mathbb{Z}) \cong \mathbb{Z}.$$
From this, we obtain the exact sequence
\[
\cdots \longrightarrow 0 \cong H^1\bigl(SO(3);\mathbb{Z}\bigr) \longrightarrow \mathbb{Z} \cong H^0(S^2;\mathbb{Z}) \longrightarrow  H^2(S^2;\mathbb{Z}) \cong \mathbb{Z} \overset{\mu^*} \longrightarrow H^2\bigl(SO(3);\mathbb{Z}\bigr) \longrightarrow 0 \longrightarrow \cdots
\]
which is precisely the Gysin sequence associated to the fibration $SO(3)\overset{\mu}{\longrightarrow}S^2$ (see, for instance, \cite[Chapter 3 on page 178]{MR658304}). \\
Thus, the cohomology class of the principal $U(1)$-bundle $\widetilde{SO(3)}$ is given by the image of the Hopf class under $\mu^*$. This shows that $\widetilde{SO(3)}$ is not homeomorphic to the product $SO(3)\times U(1)$. \\
In case ($2$), we aim to show that the group $\widetilde{SO(3)}$ is not homomorphic to the group $SO(3) \times U(1)$. To do this, we examine the group structure of $\widetilde{SO(3)}$ as constructed in Proposition \ref{firstprop}. \\
First, recall the definition of the map
\begin{align}
C : LS^2 & \longrightarrow U(1) \nonumber \\
l ~~ & \mapsto e^{2\pi(\int\limits_\delta \omega)\mathsf{i}} \label{themapC}
\end{align}
where $LS^2$ denotes the loop space of $S^2$, and $\delta$ is a surface in $S^2$ with boundary $\partial \delta = l$. The map $C$ is well-defined, since if $\delta$ and $\delta'$ are two such surfaces with the same boundary, then $\int_\delta \omega - \int_{\delta'} \omega$ is an integer multiple of $2\pi$, making the exponential independent of the choice of $\delta$.
Using $C$, we may describe $\widetilde{SO(3)}$ as the set of equivalence classes of triples $(A,\gamma,z)$, where $A\in SO(3)$, $\gamma$ is a path from $Y_N=(0,1,0)$ to $AY_N$, and $z\in U(1)$. Two triples $(A,\gamma,z)$ and $(A',\gamma',z')$ are equivalent if: \\
$\bullet$ $A = A^\prime$, \\
$\bullet$ $z = C(\gamma * {\gamma^\prime}^{-1}) z^\prime$ \\
where $\gamma * \gamma'^{-1}$ denotes the concatenation of $\gamma$ with the reverse path $\gamma'^{-1}$. \\
The multiplication in $\widetilde{SO(3)}$ is given by
\[
(A,\gamma,z) \cdot (A^\prime,\gamma^\prime,z^\prime) \coloneq (AA^\prime,\gamma * A\gamma^\prime,zz^\prime).
\]
where $A\gamma'$ is the path from $AY_N$ to $AA'Y_N$ obtained by applying $A$ pointwise to $\gamma'$. Thus $\gamma * A\gamma'$ is a path from $Y_N$ to $AA'Y_N$. \\
To see that $\widetilde{SO(3)}$ is not homomorphic to $SO(3)\times U(1)$, it suffices to compare their finite subgroups. Indeed, in $\widetilde{SO(3)}$ one can exhibit non-trivial central extensions of finite subgroups of $SO(3)$. \\
As an example, consider the dihedral group $D_8$ embedded in $SO(3)$, consisting of the following eight matrices:
\[
I = \begin{bmatrix} 1 & 0 & 0 \\ 0 & 1 & 0 \\ 0 & 0 & 1  \end{bmatrix}, ~ A_1=\begin{bmatrix} 0 & -1 & 0 \\ 1 & 0 & 0 \\ 0 & 0 & 1  \end{bmatrix}, ~ A_2=\begin{bmatrix} -1 & 0 & 0 \\ 0 & -1 & 0 \\ 0 & 0 & 1  \end{bmatrix}, ~ A_3=\begin{bmatrix} 0 & 1 & 0 \\ -1 & 0 & 0 \\ 0 & 0 & 1  \end{bmatrix},
\]
\[
A_4=\begin{bmatrix} -1 & 0 & 0 \\ 0 & 1 & 0 \\ 0 & 0 & -1  \end{bmatrix}, ~ A_5=\begin{bmatrix} 0 & 1 & 0 \\ 1 & 0 & 0 \\ 0 & 0 & -1  \end{bmatrix}, ~ A_6=\begin{bmatrix} 1 & 0 & 0 \\ 0 & -1 & 0 \\ 0 & 0 & -1  \end{bmatrix}, ~ A_7=\begin{bmatrix} 0 & -1 & 0 \\ -1 & 0 & 0 \\ 0 & 0 & -1  \end{bmatrix}
\]
For consistency with existing literature, we denote the elements $I, A_1,\dots,A_7$ by
$$1, r, r^2, r^3, s, rs, r^2s, r^3s$$
with the following interpretations: \\
$\bullet$ $1$: the identity element, \\
$\bullet$ $r$: the rotation by $\frac{\pi}{2}$ (counterclockwise) in the $XY$-plane (of order 4), \\
$\bullet$ $r^2$: the rotation by $\pi$ (counterclockwise) in the $XY$-plane (of order 2), \\
$\bullet$ $r^3$: the rotation by $\frac{3\pi}{2}$ (counterclockwise) in the $XY$-plane (of order 4), \\
$\bullet$ $s$: the reflection across the $Y$-axis (of order 2), \\
$\bullet$ $rs$: the reflection across the line $Y=-X$ (of order 2), \\
$\bullet$ $r^2s$: the reflection across the $X$-axis (of order 2), \\
$\bullet$ $r^3s$: the reflection across the line $Y=X$ (of order 2). \\
Thus $D_8$ admits the familiar presentation
\[
D_8 = \langle r,s ~ \big\vert ~ r^4 = 1, s^2 = 1, srs^{-1} = r^{-1} \rangle.
\]
By the universal coefficient theorem, one computes
\[
H^2(D_8,U(1)) \cong \mathbb{Z}_2.
\]
Indeed, $H_1(D_8) \cong \mathbb{Z}_2$ and $H_2(D_8) \cong \mathbb{Z}_2$, see for example \cite[Example 6.8.4 on page 196 and Example 6.8.5 on page 197]{MR1269324}. \\
Therefore $D_8$ admits exactly two central extensions by $U(1)$: \\
$\bullet$ the trivial extension $D_8 \times U(1)$, which sits naturally inside $SO(3)\times U(1)$, and \\
$\bullet$ a non-trivial extension $\widetilde{D_8}$, realized as a subgroup of $\widetilde{SO(3)}$ constructed in Proposition \ref{firstprop}. \\
Our aim is to demonstrate that the subgroup $\widetilde{D_8}$ of $\widetilde{SO(3)}$, as obtained in Proposition \ref{firstprop}, is not homomorphic to the product group $D_8 \times U(1)$. \\
The binary dihedral group $BD_8$ of order $16$ is defined as follows \cite{MR1863996}
\begin{align*}
BD_8 & = \left\{e^{\mathsf{k}\frac{\pi l}{4}} ~ \big\vert ~ l = 0,1,2,\dots,7\right\} \cup ~ \mathsf{j} \cdot \left\{e^{\mathsf{k}\frac{\pi l}{4}} ~ \big\vert ~ l = 0,1,2,\dots,7\right\} \\
& = \left\{1,\frac{\sqrt{2}}{2} + \frac{\sqrt{2}}{2} \mathsf{k},\mathsf{k},-\frac{\sqrt{2}}{2} + \frac{\sqrt{2}}{2} \mathsf{k},-1,-\frac{\sqrt{2}}{2} - \frac{\sqrt{2}}{2} \mathsf{k},-\mathsf{k},\frac{\sqrt{2}}{2} - \frac{\sqrt{2}}{2} \mathsf{k}\right\} \\
& \cup \left\{\mathsf{j},\frac{\sqrt{2}}{2} \mathsf{j} + \frac{\sqrt{2}}{2} \mathsf{i},\mathsf{i},-\frac{\sqrt{2}}{2} \mathsf{j} + \frac{\sqrt{2}}{2} \mathsf{i},-\mathsf{j},-\frac{\sqrt{2}}{2} \mathsf{j} - \frac{\sqrt{2}}{2} \mathsf{i},-\mathsf{i},\frac{\sqrt{2}}{2} \mathsf{j} - \frac{\sqrt{2}}{2} \mathsf{i}\right\}
\end{align*}
A standard group presentation for $BD_8$ is
\[
BD_8 = \langle x,y ~ \big\vert ~ x^8 = 1, y^4 = 1, y^{-1}xy=x^{-1} \rangle.
\]
For reference, a $2$-cocycle\footnote{For an introduction to group cohomology, see \cite[Section 1 in Chapter 3]{MR672956}.}associated with the binary dihedral group is provided in Table \ref{bindihco}.
\begin{table}[H]
\centering
\begin{tabular}{ c | c | c | c | c | c | c | c | c }
& $1$ & $r$ & $r^2$ & $r^3$ & $s$ & $rs$ & $r^2s$ & $r^3s$ \\
\hline
$1$ & $1$ & $1$ & $1$ & $1$ & $1$ & $1$ & $1$ & $1$ \\
\hline
$r$ & $1$ & $1$ & $1$ & $-1$ & $1$ & $1$ & $1$ & $-1$ \\
\hline
$r^2$ & $1$ & $1$ & $-1$ & $-1$ & $1$ & $1$ & $-1$ & $-1$ \\
\hline
$r^3$ & $1$ & $-1$ & $-1$ & $-1$ & $1$ & $-1$ & $-1$ & $-1$ \\
\hline
$s$ & $1$ & $-1$ & $-1$ & $-1$ & $-1$ & $1$ & $1$ & $1$ \\
\hline
$rs$ & $1$ & $1$ & $-1$ & $-1$ & $-1$ & $-1$ & $1$ & $1$ \\
\hline
$r^2s$ & $1$ & $1$ & $1$ & $-1$ & $-1$ & $-1$ & $-1$ & $1$ \\
\hline
$r^3s$ & $1$ & $1$ & $1$ & $1$ & $-1$ & $-1$ & $-1$ & $-1$
\end{tabular}
\caption{A $2$-cocycle for binary dihedral group $BD_8$}
\label{bindihco}
\end{table}
We adopt a multiplicative convention for the group $\mathbb{Z}_2$, namely 
\[
\mathbb{Z}_2 = \left\{\pm{1}\right\},
\]
as this formulation is more natural within our framework. \\
To obtain the relevant $2$-cocycle, we define a set-theoretic section
\[
sec:D_8 \longrightarrow BD_8
\]
defined by
\begin{align*}
sec:1 & \longmapsto 1, \qquad
r  \longmapsto e^{\mathsf{k}\frac{\pi}{4}}, \\
r^2 & \longmapsto e^{\mathsf{k}\frac{\pi}{2}}, \qquad
r^3  \longmapsto e^{\mathsf{k}\frac{3\pi}{4}}, \\
s & \longmapsto \mathsf{j}, \qquad
rs  \longmapsto \mathsf{j}e^{\mathsf{k}\frac{\pi}{4}}, \\
r^2s & \longmapsto \mathsf{j}e^{\mathsf{k}\frac{\pi}{2}} \qquad
r^3s  \longmapsto \mathsf{j}e^{\mathsf{k}\frac{3\pi}{4}}.
\end{align*}
Using the section $sec$, we define a $2$-cocycle
\[
c:D_8 \times D_8 \longrightarrow \mathbb{Z}_2
\]
by the formula
$$c(A,B) \coloneq \frac{sec(A) \times sec(B)}{sec(AB)}.$$
This map $c$ satisfies the $2$-cocycle condition
\begin{align}\label{dihco}
\frac{A \cdot c(A_2,A_3) \times c(A_1,A_2A_3)}{c(A_1A_2,A_3) \times c(A_1,A_2)}.
\end{align}
for all $A_1$, $A_2$ and $A_3$ in $D_8$. \\
In the context of central extensions, the action of $D_8$ on $\mathbb{Z}_2$ is trivial. Accordingly, Equation (\ref{dihco}) reduces to
\begin{align*}
\frac{c(A_2,A_3) \times c(A_1,A_2A_3)}{c(A_1A_2,A_3) \times c(A_1,A_2)}.
\end{align*}
\noindent
Finally, a $2$-cocycle
\[
c^\prime:D_8 \times D_8 \longrightarrow U(1),
\]
corresponding to the group $\widetilde{D_8}$, is provided in Table \ref{tildedihco}.
\begin{table}[H]
\centering
\begin{tabular}{ c | c | c | c | c | c | c | c | c }
& $1$ & $r$ & $r^2$ & $r^3$ & $s$ & $rs$ & $r^2s$ & $r^3s$ \\
\hline
$1$ & $1$ & $1$ & $1$ & $1$ & $1$ & $1$ & $1$ & $1$ \\
\hline
$r$ & $1$ & $1$ & $1$ & $-1$ & $1$ & $1$ & $1$ & $-1$ \\
\hline
$r^2$ & $1$ & $1$ & $-1$ & $-1$ & $1$ & $1$ & $-1$ & $-1$ \\
\hline
$r^3$ & $1$ & $-1$ & $-1$ & $-1$ & $1$ & $-1$ & $-1$ & $-1$ \\
\hline
$s$ & $1$ & $-1$ & $-1$ & $-1$ & $1$ & $-1$ & $-1$ & $-1$ \\
\hline
$rs$ & $1$ & $1$ & $-1$ & $-1$ & $1$ & $1$ & $-1$ & $-1$ \\
\hline
$r^2s$ & $1$ & $1$ & $1$ & $-1$ & $1$ & $1$ & $1$ & $-1$ \\
\hline
$r^3s$ & $1$ & $1$ & $1$ & $1$ & $1$ & $1$ & $1$ & $1$
\end{tabular}
\caption{The $2$-cocycle for $\widetilde{D_8}$ (base-point is taken to be $Y_N \in S^2$)}
\label{tildedihco}
\end{table}
We recall that the group $\widetilde{D_8}$ arises from Proposition \ref{firstprop}, with the base-point in $S^2$ chosen to be
\[
Y_N = (0,1,0).
\]
To analyse its cohomological structure, consider the commutative diagram
\[
\begin{tikzcd}[scale=0.95,transform shape]
0 \arrow [r] & \mathsf{Ext}^1\bigl(H_1(D_8),\mathbb{Z}_2\bigr) \arrow[r,"i"] \arrow[d] & H^2(D_8,\mathbb{Z}_2) \arrow[r,"p"] \arrow[d,"\bar{f}"'] & \mathsf{Hom}\bigl(H_2(D_8),\mathbb{Z}_2\bigr) \arrow[r] \arrow[d,"\cong"] & 0 \\
0 \arrow[r] & 0 \cong \mathsf{Ext}^1\bigl(H_1(D_8),U(1)\bigr) \arrow[r,hook] & H^2\bigl(D_8,U(1)\bigr) \arrow[r,"\cong"'] & \mathsf{Hom}\bigl(H_2(D_8),U(1)\bigr) \arrow[r] & 0
\end{tikzcd}
\]
arising from the universal coefficient theorem, together with the non-trivial group homomorphism 
$$f:\mathbb{Z}_2 \longrightarrow U(1).$$
From this diagram it follows that 
$$\mathsf{Ker}(p) \cong \mathsf{Ker}(\bar{f}).$$
Recall that the group $H_1(D_8)$ is isomorphic to the Klein four-group:
\[
H_1(D_8) \cong K_4 \cong \mathbb{Z}_2 \oplus \mathbb{Z}_2.
\]
Consider the following presentation of the Klein four-group:
$$K_4 = \langle a, b ~ \big\vert ~ a^2 = b^2 = (ab)^2 = 1\rangle.$$
Equivalently, the group consists of four elements,
$$\left\{1,a,b,c\right\},$$
with multiplication given by the Cayley table:
\begin{table}[H]
\centering
\begin{tabular}{ c | c   c   c   c}
& $1$ & $a$ & $b$ & $c$ \\
\hline
$1$ & $1$ & $a$ & $b$ & $c$ \\
$a$ & $a$ & $1$ & $c$ & $b$ \\
$b$ & $b$ & $c$ & $1$ & $a$ \\
$c$ & $c$ & $b$ & $a$ & $1$
\end{tabular}
\end{table}
We now define four group homomorphisms $$f_1, f_2, f_3, f_4:K_4 \longrightarrow \mathbb{Z}_2$$ specified by their values on the generators $a$ and $b$:
\[
f_1:a \longmapsto 1, ~~ b \longmapsto 1, ~~~~~ f_2:a  \longmapsto 1, ~~ b \longmapsto -1,
\]
\[
f_3:a  \longmapsto -1, ~~ b \longmapsto 1, ~~~~~ f_4:a \longmapsto -1, ~~ b \longmapsto -1.
\]
hese homomorphisms give rise to four distinct $2$-cocycles
\[
c_1, c_2, c_3, c_4: K_4 \times K_4 \longrightarrow \mathbb{Z}_2
\]
each corresponding to an abelian extension of $K_4$ by $\mathbb{Z}_2$. Collectively, these extensions form the group $$\mathsf{Ext}^1(K_4,\mathbb{Z}_2).$$
The explicit forms of the cocycles are presented in Table \ref{k4abext}.
\begin{table}[H]
\centering
\begin{tabular}{ c | c   c   c   c}
& $1$ & $a$ & $b$ & $c$ \\
\hline
$1$ & $1$ & $1$ & $1$ & $1$ \\
$a$ & $1$ & $1$ & $1$ & $1$ \\
$b$ & $1$ & $1$ & $1$ & $1$ \\
$c$ & $1$ & $1$ & $1$ & $1$
\end{tabular} ~~~~~
\begin{tabular}{ c | c   c   c   c}
& $1$ & $a$ & $b$ & $c$ \\
\hline
$1$ & $1$ & $1$ & $1$ & $1$ \\
$a$ & $1$ & $1$ & $1$ & $1$ \\
$b$ & $1$ & $1$ & $-1$ & $-1$ \\
$c$ & $1$ & $1$ & $-1$ & $-1$
\end{tabular} ~~~~~
\begin{tabular}{ c | c   c   c   c}
& $1$ & $a$ & $b$ & $c$ \\
\hline
$1$ & $1$ & $1$ & $1$ & $1$ \\
$a$ & $1$ & $-1$ & $1$ & $-1$ \\
$b$ & $1$ & $1$ & $1$ & $1$ \\
$c$ & $1$ & $-1$ & $1$ & $-1$
\end{tabular} ~~~~~
\begin{tabular}{ c | c   c   c   c}
& $1$ & $a$ & $b$ & $c$ \\
\hline
$1$ & $1$ & $1$ & $1$ & $1$ \\
$a$ & $1$ & $-1$ & $-1$ & $1$ \\
$b$ & $1$ & $-1$ & $-1$ & $1$ \\
$c$ & $1$ & $1$ & $1$ & $1$
\end{tabular}
\caption{The $2$-cocycles $c_1$, $c_2$, $c_3$ and $c_4$}
\label{k4abext}
\end{table}
The quotient map $$q:D_8 \longrightarrow K_4$$ is defined by
\begin{align*}
q:1 & \longmapsto 1, \qquad
r \longmapsto a, \\
r^2 & \longmapsto 1, \qquad
r^3 \longmapsto a, \\
s & \longmapsto b, \qquad
rs \longmapsto c, \\
r^2s & \longmapsto b, \qquad
r^3s \longmapsto c.
\end{align*}
Hence, the kernel of $q$ is
$$\mathsf{Ker}(1) = \left\{1, r^2\right\}.$$
By composing each of the cocycles $c_1$, $c_2$, $c_3$ and $c_4$ with $q \times q$, we obtain four $2$-cocyles
$$c^\prime_1 = c_1 \circ (q \times q):D_8 \times D_8 \longrightarrow \mathbb{Z}_2,$$
$$c^\prime_2 = c_2 \circ (q \times q):D_8 \times D_8 \longrightarrow \mathbb{Z}_2,$$
$$c^\prime_3 = c_3 \circ (q \times q):D_8 \times D_8 \longrightarrow \mathbb{Z}_2,$$
$$c^\prime_4 = c_4 \circ (q \times q):D_8 \times D_8 \longrightarrow \mathbb{Z}_2.$$
These four cocycles lie in $$\mathsf{Im}(i) \cong \mathsf{Ker}(p),$$
and they are not cohomologous to the $2$-cocycle $c^\prime$ associated with the binary dihedral group $BD_8$.
These four $2$-cocycles are shown in Tables \ref{comingk4-1} and \ref{comingk4-2}.
\begin{table}[H]
\centering
\begin{tabular}{ c | c   c   c   c   c   c   c   c}
& $1$ & $r$ & $r^2$ & $r^3$ & $s$ & $rs$ & $r^2s$ & $r^3s$ \\
\hline
$1$ & $1$ & $1$ & $1$ & $1$ & $1$ & $1$ & $1$ & $1$ \\
$r$ & $1$ & $1$ & $1$ & $1$ & $1$ & $1$ & $1$ & $1$ \\
$r^2$ & $1$ & $1$ & $1$ & $1$ & $1$ & $1$ & $1$ & $1$ \\
$r^3$ & $1$ & $1$ & $1$ & $1$ & $1$ & $1$ & $1$ & $1$ \\
$s$ & $1$ & $1$ & $1$ & $1$ & $1$ & $1$ & $1$ & $1$ \\
$rs$ & $1$ & $1$ & $1$ & $1$ & $1$ & $1$ & $1$ & $1$ \\
$r^2s$ & $1$ & $1$ & $1$ & $1$ & $1$ & $1$ & $1$ & $1$ \\
$r^3s$ & $1$ & $1$ & $1$ & $1$ & $1$ & $1$ & $1$ & $1$
\end{tabular} ~~~
\begin{tabular}{ c | c   c   c   c   c   c   c   c}
& $1$ & $r$ & $r^2$ & $r^3$ & $s$ & $rs$ & $r^2s$ & $r^3s$ \\
\hline
$1$ & $1$ & $1$ & $1$ & $1$ & $1$ & $1$ & $1$ & $1$ \\
$r$ & $1$ & $1$ & $1$ & $1$ & $1$ & $1$ & $1$ & $1$ \\
$r^2$ & $1$ & $1$ & $1$ & $1$ & $1$ & $1$ & $1$ & $1$ \\
$r^3$ & $1$ & $1$ & $1$ & $1$ & $1$ & $1$ & $1$ & $1$ \\
$s$ & $1$ & $1$ & $1$ & $1$ & $-1$ & $-1$ & $-1$ & $-1$ \\
$rs$ & $1$ & $1$ & $1$ & $1$ & $-1$ & $-1$ & $-1$ & $-1$ \\
$r^2s$ & $1$ & $1$ & $1$ & $1$ & $-1$ & $-1$ & $-1$ & $-1$ \\
$r^3s$ & $1$ & $1$ & $1$ & $1$ & $-1$ & $-1$ & $-1$ & $-1$
\end{tabular}
\caption{The $2$-cocycles $c^\prime_1$ and $c^\prime_2$}
\label{comingk4-1}
\end{table}
\begin{table}[H]
\centering
\begin{tabular}{ c | c   c   c   c   c   c   c   c}
& $1$ & $r$ & $r^2$ & $r^3$ & $s$ & $rs$ & $r^2s$ & $r^3s$ \\
\hline
$1$ & $1$ & $1$ & $1$ & $1$ & $1$ & $1$ & $1$ & $1$ \\
$r$ & $1$ & $-1$ & $1$ & $-1$ & $1$ & $-1$ & $1$ & $-1$ \\
$r^2$ & $1$ & $1$ & $1$ & $1$ & $1$ & $1$ & $1$ & $1$ \\
$r^3$ & $1$ & $-1$ & $1$ & $-1$ & $1$ & $-1$ & $1$ & $-1$ \\
$s$ & $1$ & $1$ & $1$ & $1$ & $1$ & $1$ & $1$ & $1$ \\
$rs$ & $1$ & $-1$ & $1$ & $-1$ & $1$ & $-1$ & $1$ & $-1$ \\
$r^2s$ & $1$ & $1$ & $1$ & $1$ & $1$ & $1$ & $1$ & $1$ \\
$r^3s$ & $1$ & $-1$ & $1$ & $-1$ & $1$ & $-1$ & $1$ & $-1$
\end{tabular} ~~~
\begin{tabular}{ c | c   c   c   c   c   c   c   c}
& $1$ & $r$ & $r^2$ & $r^3$ & $s$ & $rs$ & $r^2s$ & $r^3s$ \\
\hline
$1$ & $1$ & $1$ & $1$ & $1$ & $1$ & $1$ & $1$ & $1$ \\
$r$ & $1$ & $-1$ & $1$ & $-1$ & $-1$ & $1$ & $-1$ & $1$ \\
$r^2$ & $1$ & $1$ & $1$ & $1$ & $1$ & $1$ & $1$ & $1$ \\
$r^3$ & $1$ & $-1$ & $1$ & $-1$ & $-1$ & $1$ & $-1$ & $1$ \\
$s$ & $1$ & $-1$ & $1$ & $-1$ & $-1$ & $1$ & $-1$ & $1$ \\
$rs$ & $1$ & $1$ & $1$ & $1$ & $1$ & $1$ & $1$ & $1$ \\
$r^2s$ & $1$ & $-1$ & $1$ & $-1$ & $-1$ & $1$ & $-1$ & $1$ \\
$r^3s$ & $1$ & $1$ & $1$ & $1$ & $1$ & $1$ & $1$ & $1$
\end{tabular}
\caption{The $2$-cocycles $c^\prime_3$ and $c^\prime_4$}
\label{comingk4-2}
\end{table}
By combining the $2$-cocycles $c^\prime_1$, $c^\prime_2$, $c^\prime_3$ and $c^\prime_4$ with the $2$-cocycle $c^\prime$, we obtain four additional extensions of $D_8$ by $\mathbb{Z}_2$. \\
The corresponding $2$-cocycles associated with these extensions are given by:
\begin{table}[H]
\centering
\begin{tabular}{ c | c   c   c   c   c   c   c   c}
& $1$ & $r$ & $r^2$ & $r^3$ & $s$ & $rs$ & $r^2s$ & $r^3s$ \\
\hline
$1$ & $1$ & $1$ & $1$ & $1$ & $1$ & $1$ & $1$ & $1$ \\
$r$ & $1$ & $1$ & $1$ & $-1$ & $1$ & $1$ & $1$ & $-1$ \\
$r^2$ & $1$ & $1$ & $-1$ & $-1$ & $1$ & $1$ & $-1$ & $-1$ \\
$r^3$ & $1$ & $-1$ & $-1$ & $-1$ & $1$ & $-1$ & $-1$ & $-1$ \\
$s$ & $1$ & $-1$ & $-1$ & $-1$ & $-1$ & $1$ & $1$ & $1$ \\
$rs$ & $1$ & $1$ & $-1$ & $-1$ & $-1$ & $-1$ & $1$ & $1$ \\
$r^2s$ & $1$ & $1$ & $1$ & $-1$ & $-1$ & $-1$ & $-1$ & $1$ \\
$r^3s$ & $1$ & $1$ & $1$ & $1$ & $-1$ & $-1$ & $-1$ & $-1$
\end{tabular} ~~~
\begin{tabular}{ c | c   c   c   c   c   c   c   c}
& $1$ & $r$ & $r^2$ & $r^3$ & $s$ & $rs$ & $r^2s$ & $r^3s$ \\
\hline
$1$ & $1$ & $1$ & $1$ & $1$ & $1$ & $1$ & $1$ & $1$ \\
$r$ & $1$ & $1$ & $1$ & $-1$ & $1$ & $1$ & $1$ & $-1$ \\
$r^2$ & $1$ & $1$ & $-1$ & $-1$ & $1$ & $1$ & $-1$ & $-1$ \\
$r^3$ & $1$ & $-1$ & $-1$ & $-1$ & $1$ & $-1$ & $-1$ & $-1$ \\
$s$ & $1$ & $-1$ & $-1$ & $-1$ & $1$ & $-1$ & $-1$ & $-1$ \\
$rs$ & $1$ & $1$ & $-1$ & $-1$ & $1$ & $1$ & $-1$ & $-1$ \\
$r^2s$ & $1$ & $1$ & $1$ & $-1$ & $1$ & $1$ & $1$ & $-1$ \\
$r^3s$ & $1$ & $1$ & $1$ & $1$ & $1$ & $1$ & $1$ & $1$
\end{tabular}
\end{table}
\begin{table}[H]
\centering
\begin{tabular}{ c | c   c   c   c   c   c   c   c}
& $1$ & $r$ & $r^2$ & $r^3$ & $s$ & $rs$ & $r^2s$ & $r^3s$ \\
\hline
$1$ & $1$ & $1$ & $1$ & $1$ & $1$ & $1$ & $1$ & $1$ \\
$r$ & $1$ & $-1$ & $1$ & $1$ & $1$ & $-1$ & $1$ & $1$ \\
$r^2$ & $1$ & $1$ & $-1$ & $-1$ & $1$ & $1$ & $-1$ & $-1$ \\
$r^3$ & $1$ & $1$ & $-1$ & $1$ & $1$ & $1$ & $-1$ & $1$ \\
$s$ & $1$ & $-1$ & $-1$ & $-1$ & $-1$ & $1$ & $1$ & $1$ \\
$rs$ & $1$ & $-1$ & $-1$ & $1$ & $-1$ & $1$ & $1$ & $-1$ \\
$r^2s$ & $1$ & $1$ & $1$ & $-1$ & $-1$ & $-1$ & $-1$ & $1$ \\
$r^3s$ & $1$ & $-1$ & $1$ & $-1$ & $-1$ & $1$ & $-1$ & $1$
\end{tabular} ~~~
\begin{tabular}{ c | c   c   c   c   c   c   c   c}
& $1$ & $r$ & $r^2$ & $r^3$ & $s$ & $rs$ & $r^2s$ & $r^3s$ \\
\hline
$1$ & $1$ & $1$ & $1$ & $1$ & $1$ & $1$ & $1$ & $1$ \\
$r$ & $1$ & $-1$ & $1$ & $1$ & $-1$ & $1$ & $-1$ & $-1$ \\
$r^2$ & $1$ & $1$ & $-1$ & $-1$ & $1$ & $1$ & $-1$ & $-1$ \\
$r^3$ & $1$ & $1$ & $-1$ & $1$ & $-1$ & $-1$ & $1$ & $-1$ \\
$s$ & $1$ & $1$ & $-1$ & $1$ & $1$ & $1$ & $-1$ & $1$ \\
$rs$ & $1$ & $1$ & $-1$ & $-1$ & $-1$ & $-1$ & $1$ & $1$ \\
$r^2s$ & $1$ & $-1$ & $1$ & $1$ & $1$ & $-1$ & $1$ & $1$ \\
$r^3s$ & $1$ & $1$ & $1$ & $1$ & $-1$ & $-1$ & $-1$ & $-1$
\end{tabular}
\end{table}
By comparing the four cocycle tables constructed above with the cocycle table for $\widetilde{D_8}$ (Table \ref{tildedihco}), we observe that the cocycle corresponding to $\widetilde{D_8}$ coincides with the second table. \\
Consequently, the $2$-cocycle $$c^\prime:D_8 \times D_8 \longrightarrow U(1)$$
is non-trivial in $H^2\bigl(D_8,U(1)\bigr)$. It follows that the group $\widetilde{D_8}$ cannot be homomorphic to the product space $D_8 \times U(1)$.
\end{const}
\begin{lem}
The connection $\zeta$ defined in Equation (\ref{connonHopf}) s invariant under multiplication by elements of $S^3$.
\end{lem}
\begin{proof}
Let $(x_0,y_0,z_0,w_0) \in S^3$. Identifying $S^3$ with the group of unit quaternions, the product of $(x,y,z,w)$ with $(x_0,y_0,z_0,w_0)$ is
\begin{align*}
(x\mathsf{i}+y\mathsf{j}+z\mathsf{k}+w) \cdot (x_0\mathsf{i}+y_0\mathsf{j}+z_0\mathsf{k}+w_0) & = -xx_0 + xy_0\mathsf{k}-xz_0\mathsf{j}+xw_0\mathsf{i} \\
& - yx_0\mathsf{k} - yy_0 + yz_0\mathsf{i} + yw_0\mathsf{j} \\
& + zx_0\mathsf{j} - zy_0\mathsf{i} -zz_0 +zw_0\mathsf{k} \\
& + wx_0\mathsf{i} +wy_0\mathsf{j} + wz_0\mathsf{k} + ww_0 \\
& = (xw_0+yz_0-zy_0+wx_0,-xz_0+yw_0+zx_0+wy_0, \\
& xy_0-yx_0+zw_0+wz_0,-xx_0-yy_0-zz_0+ww_0).
\end{align*}
Accordingly, the connection one-form $\zeta$ becomes
\begin{align*}
& - (-xz_0+yw_0+zx_0+wy_0)(w_0dx+z_0dy-y_0dz+x_0dw) \\ & + (xw_0+yz_0-zy_0+wx_0)(-z_0dx+w_0dy+x_0dz+y_0dw) \\ & - (-xx_0-yy_0-zz_0+ww_0)(y_0dx-x_0dy+w_0dz+z_0dw) \\ & + (xy_0-yx_0+zw_0+wz_0)(-x_0dx-y_0dy-z_0dz-w_0dw).
\end{align*}
Collecting coefficients of $dx$, we find
\begin{align*}
& -xx_0y_0-yx_0^2-zx_0w_0+wx_0z_0+xx_0y_0-yy_0^2-zy_0z_0-wy_0w_0 \\ 
& -xz_0w_0-yz_0^2+zy_0z_0-wx_0z_0+xz_0w_0-yw_0^2+zx_0w_0+wy_0w_0 \\
& = -y.
\end{align*}
The coefficients of $dy$, $dz$ and $dw$ can be verified by similar calculations. Thus, the form $\zeta$ is preserved under multiplication by elements of $S^3$.
\end{proof}
\begin{rem}\label{connection_on_Hopf_just}
The connection $\zeta$ on the Hopf map $\eta:S^3 \longrightarrow S^2$ is well known in the literature; see for example, \cite[on page 269]{MR3012377}. Here, we justify why $\zeta$ arises naturally as a connection on $\eta$.
\\
In Construction \ref{cathopf}, we took the six-open cover $\left\{V_1, V_2, \cdots, V_6\right\}$ of $S^2$. On $S^2$, consider the $2$-form
\[
\omega = \frac{1}{2}\left(XdY \wedge dZ + YdZ \wedge dX + ZdX \wedge dY\right).
\]
Define a map
\[
f:\mathbb{R}^2 \longrightarrow S^2, ~~~~~ (X,Y) \longmapsto \left(X,Y,\sqrt{1 - X^2 - Y^2}\right).
\]
The image of $f$ is precisely the upper hemisphere of $S^2$, i.e. the region with $Z \geq 0$. \\
Pulling back the restriction of $\omega$ on hemisphere along the map $f$, we obtain:
\begin{align*}
f^\ast\omega & = \frac{1}{2} \times XdY \wedge \left(-\frac{XdX}{\sqrt{1-X^2-Y^2}} - \frac{YdY}{\sqrt{1-X^2-Y^2}}\right) \\
& + \frac{1}{2} \times Y\left(-\frac{XdX}{\sqrt{1-X^2-Y^2}} - \frac{YdY}{\sqrt{1-X^2-Y^2}}\right) \wedge dX + \frac{1}{2} \times \sqrt{1-X^2-Y^2}dX \wedge dY \\
& = -\frac{X^2dY \wedge dX}{2\sqrt{1-X^2-Y^2}} - \frac{Y^2dY \wedge dX}{2\sqrt{1-X^2-Y^2}} + \frac{1}{2} \times \sqrt{1-X^2-Y^2}dX \wedge dY \\
& = \frac{(X^2+Y^2)dX \wedge dY}{2\sqrt{1-X^2-Y^2}} + \frac{1}{2} \times \sqrt{1-X^2-Y^2}dX \wedge dY \\
& = \frac{dX \wedge dY}{2\sqrt{1-X^2-Y^2}}.
\end{align*}
This is a closed $2$-form on $\mathbb{R}^2$. Since $\mathbb{R}^2$ is contractible, By Poincar\'{e} lemma\footnote{See \cite[on page 94]{MR209411}.}there exists a $1$-form $\alpha$ such that:
\[
d\alpha = f^\ast\omega.
\]
Write $r(X,Y) = \frac{1}{2\sqrt{1-X^2-Y^2}}$. Then
\begin{align*}
\alpha & = \left(\int\limits_0^1 t \times r(tX,tY)dt\right)XdY - \left(\int\limits_0^1 t \times r(tX,tY)dt\right)YdX.
\end{align*}
It remains to evaluate
\[
\int\limits_0^1 \frac{tdt}{2\sqrt{1-t^2X^2-t^2Y^2}}.
\]
With $u = t^2$, $du = 2tdt$,
\[
\int\limits_0^1 \frac{tdt}{2\sqrt{1-t^2X^2-t^2Y^2}} = \frac{1}{4}\int\limits_0^1 \frac{du}{\sqrt{1-u(X^2+Y^2)}}.
\]
Let $v = 1- u(x^2+y^2)$. Then
\begin{align*}
\frac{1}{4}\int\limits_0^1 \frac{du}{\sqrt{1-u(X^2+Y^2)}} & = -\frac{1}{4(X^2+Y^2)}\int\limits_1^{1-(X^2+Y^2)} \frac{dv}{\sqrt{v}} \\
& = -\frac{1}{4(X^2+Y^2)} \left[2v^{1/2}\right]^{1-(X^2+Y^2)}_1 \\
& = \frac{1 - \sqrt{1-(X^2+Y^2)}}{2(X^2+Y^2)} \\
& = \frac{1}{2} \times \frac{1}{1+\sqrt{1-(X^2+Y^2)}}.
\end{align*}
Hence
\[
\alpha = \frac{1}{2} \times \frac{1}{1+\sqrt{1-(X^2+Y^2)}}(-YdX + XdY).
\]
Since $f$ is a diffeomorphism onto its image (the upper hemisphere $Z \geq 0$), we may identify $\alpha$ with a $1$-form on that hemisphere and then pull back along the Hopf map $\eta$. Using the standard coordinate expression for the Hopf map
\[
dX = 2zdx+2wdy+2xdz+2ydw, ~~~~~ dY = -2wdx+2zdy+2ydz-2xdw.
\]
Therefore
\begin{align*}
\eta^\ast\alpha & = \eta^\ast\left(\frac{1}{2} \times \frac{1}{1+\sqrt{1-(X^2+Y^2)}}(-YdX + XdY)\right) \\
& = \frac{1}{2} \times \frac{-2(yz-xw)(2zdx+2wdy+2xdz+2ydw) + 2(xz+yw)(-2wdx+2zdy+2ydz-2xdw)}{1 + \sqrt{1 - 4(xz+yw)^2 - 4(yz-xw)^2}}.
\end{align*}
We calculate:
\begin{align*}
-2(yz-xw)(2zdx+2wdy+2xdz+2ydw) & = (-2yz+2xw)(2zdx+2wdy+2xdz+2ydw) \\
& = -4yz^2dx-4yzwdy-4xyzdz-4y^2zdw \\
& + 4xzwdx+4xw^2dy+4x^2wdz+4xywdw. \\
2(xz+yw)(-2wdx+2zdy+2ydz-2xdw) & = (2xz+2yw)(-2wdx+2zdy+2ydz-2xdw) \\
& = -4xzwdx+4xz^2dy+4xyzdz-4x^2zdw \\
& -4yw^2dx+4yzwdy+4y^2wdz-4xywdw. \\
1 - 4(xz+yw)^2 - 4(yz-xw)^2 & = 1 - 4(x^2z^2+y^2w^2+2xyzw) - 4(y^2z^2+x^2w^2-2xyzw) \\
& = 1-4x^2z^2-4y^2w^2-4y^2z^2-4x^2w^2
\end{align*}
In the numerator, we get:
\[
-4yz^2dx-4yw^2dx+4xw^2dy+4xz^2dy+4x^2wdz+4y^2wdz-4y^2zdw-4x^2zdw
\]
which is equal to:
\[
4(z^2+w^2)(-ydx+xdy)+4(x^2+y^2)(wdz-zdw).
\]
In the denominator, we get:
\begin{align*}
1 - 4(xz+yw)^2 - 4(yz-xw)^2 & = 1 - 4(x^2z^2+y^2w^2+2xyzw) - 4(y^2z^2+x^2w^2-2xyzw) \\
& = 1-4x^2z^2-4y^2w^2-4y^2z^2-4x^2w^2.
\end{align*}
Set $A \coloneq x^2z^2+y^2w^2+y^2z^2+x^2w^2$. We want to compute $1-4A$. We have:
\begin{equation}\label{1-4a}
(x^2+y^2+z^2+w^2)^2 = x^4+y^4+z^4+w^4+2(x^2y^2+x^2z^2+x^2w^2+y^2z^2+y^2w^2+z^2w^2),
\end{equation}
and
\[
(x^2+y^2-z^2-w^2)^2 = x^4+y^4+z^4+w^4+2(x^2y^2+z^2w^2)-2(x^2z^2+x^2w^2+y^2z^2+y^2w^2).
\]
Therefore,
\begin{align*}
1 + (x^2+y^2-z^2-w^2)^2 & = 2(x^4+y^4+z^4+w^4)+4(x^2y^2+z^2w^2)
\end{align*}
By Equation (\ref{1-4a}), we have:
\[
2A = 1 - (x^4+y^4+z^4+w^4) - 2(x^2y^2+z^2w^2).
\]
So,
\[
1 - 4A = -1 + 2(x^4+y^4+z^4+w^4) + 4(x^2y^2+z^2w^2) = (x^2+y^2-z^2-w^2)^2.
\]
The pullback $\eta^\ast\alpha$ becomes:
\begin{align*}
\eta^\ast\alpha & =  \frac{2(z^2+w^2)(-ydx+xdy)+2(x^2+y^2)(wdz-zdw)}{1 + x^2+y^2-z^2-w^2} \\
& = \frac{2(1-x^2-y^2)(-ydx+xdy)+2(x^2+y^2)(wdz-zdw)}{2(x^2+y^2)} \\
& = \frac{(-ydx+xdy)}{x^2+y^2} + (ydx-xdy) + (wdz-zdw).
\end{align*}
We observe that $\eta^\ast \alpha$ defers from $\zeta$ by the exact $1$-form
\[
\frac{(-ydx+xdy)}{x^2+y^2}.
\]
Hence, $[\zeta] = [\eta^\ast \alpha]$.
\end{rem}
\begin{prop}\label{symmofhopfprop}
The group of symmetries of the Hopf map $\eta:S^3 \longrightarrow S^2$ is isomorphic to $Spin^c(3)$.
\end{prop}
\begin{proof}
This follows directly from Construction \ref{symmofhopf} and statement ($\rNum{2}$) in \ref{secondprop}.
\end{proof}
In the following, we recall the categorical analogue of Proposition \ref{secondprop}. For the definitions of anafunctors, transformations between anafunctors, and pullbacks along anafunctors, see \cite[Definition 2.3.1 on page 7 and Section 4.3]{MR3894086}. Following the terminology of \cite{MR3089401}, and as discussed similarly in Remark~\ref{butobunandconv}, we use the notions of principal $\mathbb{B}U(1)$-bundles and bundle gerbes interchangeably in what follows.
\begin{prop}\cite[Lemma 1.3.5 on page 17 and Theorem 1.3.6 on page 18]{Wal2007phd} \cite[Theorem 2.3 on page 101]{MR3013040} \cite[Corollary 2.3 on page 263]{MR2318389} \label{secondpropcat}
Let $M$ be a $2$-connected smooth manifold on which we have an integral closed $3$-form $\nu/{2\pi^2}$. We have the following statements: \\
$1)$ There exists a principal $\mathbb{B}U(1)$-bundle $J:\mathcal{Q} \longrightarrow M$ with a connection $\Omega$ whose curvature is $\nu$. \\
$2)$ Given two principal $\mathbb{B}U(1)$-bundles $J:\mathcal{Q} \longrightarrow M$ and $J^\prime:\mathcal{Q}^\prime \longrightarrow M$ with connections $\Omega$ and $\Omega^\prime$ with the
same curvature $\nu$, there exists an anafunctor $\Psi:\mathcal{Q} \longrightarrow \mathcal{Q}^\prime$ such that $[\Psi^\ast\Omega^\prime] = [\Omega]$. Moreover, $\Psi$ is unique up to a transformation, which is unique up to multiplication by elements of $U(1)$.
\end{prop}
\begin{proof}
Both statements follow from the classification of bundle gerbes over $M$ by $H^3(S^3;\mathbb{Z})$ and the fact that $v/2\pi^2$ represents the image of a class of third integral cohomology in real cohomology. The uniqueness of a transformation between two such anafunctors $\mathcal{Q} \longrightarrow \mathcal{Q}$ is established using \cite[Corollary 2.3 on page 263]{MR2318389}. Specifically, the set of isomorphism classes of $1$-morphsims forms a torsor over the set of isomorphism classes of line bundles over $M$. Since $H^2(M;\mathbb{Z}) = 0$, the result follows.
\end{proof}
\begin{cor}\label{corsecondpropcat}
Let $M$ be a $2$-connected smooth manifold on which we have an integral closed $3$-form $\nu/{2\pi^2}$. Let $G$ be a Lie group acting on $M$, leaving $\nu/{2\pi^2}$ invariant. Given a principal $\mathbb{B}U(1)$-bundle $J:\mathcal{Q} \longrightarrow M$ with a connection $\omega$ whose curvature is $\nu$, there exists an anafunctor $\Phi:\mathcal{Q} \longrightarrow \mathcal{Q}$ covering the action of $G$ on $M$, such that $$[\Phi^\ast\Omega] = [\Omega].$$
\end{cor}
\begin{proof}
Let $g \in G$. The action of $g$ defines a map
\[
g:M \longrightarrow M,
\]
and hence induces the pullback map $g^\ast J:g^\ast \mathcal{Q} \longrightarrow M$ of $J$ along $g$. By Proposition \ref{secondpropcat}, we obtain a connection-preserving anafunctor
\[
\Psi:\mathcal{Q} \longrightarrow g^\ast \mathcal{Q}
\]
which fits into a commutative diagram
\[
\begin{tikzcd}
\mathcal{Q} \arrow[dr,"J"'] \arrow[rr,"\Psi"] & & g^\ast \mathcal{Q} \arrow[dl] \\
& M &
\end{tikzcd}.
\]
Now consider the pullback\footnote{See \cite[page 54]{MR2709030} and \cite[page 258]{MR2318389} for the definition of categorical pullbacks.}diagram
\[
\begin{tikzcd}
g^\ast \mathcal{Q} \arrow[r,"\lambda"] \arrow[d,"\theta"'] & \mathcal{Q} \arrow[d,"J"] \\
M \arrow[r,"g"'] & M
\end{tikzcd}
\]
and define $\Phi = \lambda \circ \Psi$. Then
\begin{align*}
(\lambda \circ \Psi)^\ast \Omega & = \Psi^\ast \circ \lambda^\ast (\Omega) \\
& = \Psi^\ast (g^\ast \Omega) \\
& = \Omega.
\end{align*}
Moreover, $\Phi$ makes the diagram
\[
\begin{tikzcd}
\mathcal{Q} \arrow[d,"J"'] \arrow[r,"\Phi"] & \mathcal{Q} \arrow[d,"J"] \\
M \arrow[r,"g"'] & M
\end{tikzcd}
\]
commute.
\end{proof}
\begin{rem}\label{hopesecondpropproof}
By Proposition \ref{secondpropcat}, we obtain a principal $\mathbb{B}U(1)$-bundle $J:\mathcal{Q} \longrightarrow M$. Following arguments analogous to those in \cite[Propositions 4.4.2 on page 47 and 4.5.3 on page 52]{MR900587} and Construction \ref{symmofhopf}, we define a categorical group $CG$ as follows. \\
$\bullet$ $\mathsf{Obj}(CG)$: pairs $(g,\Psi)$, where $g \in G$ and $\Psi:\mathcal{Q} \longrightarrow \mathcal{Q}$ is an anafunctor obtained via Corollary \ref{corsecondpropcat}. \\
$\bullet$ $\mathsf{Mor}(CG)$: maps $F:(g,\Psi) \longrightarrow (g,\Psi^\prime)$, where $F$ is a transformation
\[
\begin{tikzcd}
\mathcal{Q} \arrow[rr,bend left=45,"\Phi",""{name=u1,below}] \arrow[rr,bend right=45,"\Psi^\prime"',""{name=u2,above}] \arrow[Rightarrow,from=u1,to=u2,"F"] & & \mathcal{Q}
\end{tikzcd}
\]
between two anafunctors $\Psi$ and $\Psi^\prime$. \\
The categorical group structure in $CG$ is as follows: \\
$\bullet$ On objects, the tensor product is given by composition of anafunctors:
\[
(g,\Psi) \otimes (g^\prime,\Psi^\prime) = (g^\prime \cdot g,\Psi^\prime \circ \Psi).
\]
\[
\begin{tikzcd}
\mathcal{Q} \arrow[r,"\Psi"] \arrow[d,"J"'] & \mathcal{Q} \arrow[r,"\Psi^\prime"] \arrow[d,"J"] & \mathcal{Q} \arrow[d,"J"] \\
M \arrow[r,"g"'] & M \arrow[r,"{g^\prime}"'] & M
\end{tikzcd}
\]
$\bullet$ On morphisms, the tensor product is given by composition of transformations:
\[
\begin{tikzcd}
\mathcal{Q} \arrow[rr,bend left=45,"\Phi",""{name=u1,below}] \arrow[rr,bend right=45,"\Psi^\prime"',""{name=u2,above}] \arrow[Rightarrow,from=u1,to=u2,"F"] & & \mathcal{Q}
\end{tikzcd} ~~ \otimes ~~
\begin{tikzcd}
\mathcal{Q} \arrow[rr,bend left=45,"\Phi^\prime",""{name=u1,below}] \arrow[rr,bend right=45,"\Psi^{\prime\prime}"',""{name=u2,above}] \arrow[Rightarrow,from=u1,to=u2,"F^\prime"] & & \mathcal{Q}
\end{tikzcd} =
\begin{tikzcd}
\mathcal{Q} \arrow[rr,bend left=45,"\Phi",""{name=u1,below}] \arrow[rr,bend right=45,"\Psi^{\prime\prime}"',""{name=u2,above}] \arrow[Rightarrow,from=u1,to=u2,"F^\prime \circ F"] & & \mathcal{Q}
\end{tikzcd}
\]
In \cite{MR2366945}, authors constructed a Lie $2$-group model for $String(n)$. For a simply connected, compact, simple Lie group $G$, they defined a categorical group $\mathcal{P}_k G$ for each $k \in \mathbb{Z}$. In our case, we take $k=1$, yielding a categorical group $\mathcal{P}G$ with objects $P_0 G$ and morphisms $P_0 G \ltimes \widehat{\Omega G}$. Here, $P_0 G$ denotes the space of smooth paths in $G$ starting at $1 \in G$, and $\widehat{\Omega G}$ denotes the central extension of the loop group $\Omega G$ by $U(1)$, as constructed in \cite{MR2366945}. \\
Now we define a map from $\mathcal{P}G$ to $CG$ as follows: \\
$\bullet$ On objects: a path $\gamma:[0,1] \to G$ with $\gamma(0) = 1$ and $\gamma(1) = g$ is sent to $(g,\Psi)$, where $\Psi$ is an anafunctor $\mathcal{Q} \longrightarrow \mathcal{Q}$ that we expect can be obtained via the parallel $2$-transport formalism developed in \cite{MR3917427} and \cite{MR3084724}. To clarify this, we describe what this formalism provides and what remains to be constructed. \\
For each path $\gamma:1 \longrightarrow g$ in $G$ and for each $m \in M$, we obtain a path
$$\gamma \cdot m:1 \cdot m \longrightarrow g \cdot m$$
in $M$. Consequently, by \cite{MR3917427}, we obtain an anafunctor
\[
\tau_m:\mathcal{Q}_m \longrightarrow \mathcal{Q}_{g \cdot m}.
\]
In general, this yields a family of anafunctors $(\tau_m)_{m \in M}$. What remains is to determine whether these local anafunctors can be assembled coherently into a single anafunctor
$$\Psi:\mathcal{Q} \longrightarrow \mathcal{Q}.$$
$\bullet$ On morphisms: a pair $(\gamma,\hat{l})$ is sent to the transformation $F$, which we again conjecture can be defined using the parallel $2$-transport formalism of \cite{MR3917427} applied to the map $(\gamma,\hat{l})$. The source and target of $(\gamma,\hat{l})$ are:
\begin{align*}
\mathsf{s}\left((\gamma,\hat{l})\right) & = \gamma, \\
\mathsf{t}\left((\gamma,\hat{l})\right) & = \partial(\hat{l})\gamma
\end{align*}
as in \cite[Proposition 4.1 on page 117]{MR2366945}, where $\partial$ is defined by the composition
\[
\begin{tikzcd}
\widehat{\Omega G} \arrow[r] & \Omega G \arrow[r,hook] & P_0 G
\end{tikzcd}.
\]
The construction of $\mathcal{P}G$ yields a multiplicative\footnote{See \cite[page 314]{MR2610397} and \cite[Proposition 5.2 on page 597]{MR2174418}}bundle gerbe over G, as discussed in \cite{MR2366945}.
\end{rem}
In what follows, we provide the conjectural construction of $String(3)$:
\begin{const}\label{catextensionofspin}
We now specialise Proposition \ref{secondpropcat} to the case $$G = Spin(3), ~~~~~ M = S^3, ~~~~~ J = H:\mathcal{P} \longrightarrow S^3$$
where $S^3$ is equipped with the volume form
\[
\nu=xdy \wedge dz \wedge dw + ydz \wedge dw \wedge dx + zdw \wedge dx \wedge dy + wdx \wedge dy \wedge dz.
\]
The integral of $\nu$ over $S^3$ is
\[
\int\limits_{S^3} \nu = 2\pi^2.
\]
Hence the normalised $3$-form $\nu/2\pi^2$ represents the positive generator of $$H^3(S^3;\mathbb{Z}) \cong \mathbb{Z},$$
and corresponds to the basic bundle gerbe over $S^3$. \\
Equipping this bundle gerbe with a connection $\Omega$ whose curvature is $\nu$, one may adapt the arguments of Construction \ref{symmofhopf} to to obtain a new interpretation of $String(3)$. The topological and categorical structure of the categorical group $CSpin(3)$, arising from Proposition \ref{secondpropcat} and Remark \ref{hopesecondpropproof}, can be described as follows: \\
$\bullet$ Topological structure can be identified with topology of $\mathcal{P}Spin(3)$, as discussed in \cite{MR2366945}. \\
$\bullet$ Categorical group structure can be regarded as a categorical group with the group of objects as pairs $(g,\Psi)$, where $g \in Spin(3)$ and $\Psi:\mathcal{P} \longrightarrow \mathcal{P}$ is a an anafunctor covering the action of $g$ and preserving the connection on $H$, and the group of morphisms as transformations $F:(g,\Psi) \longrightarrow (g,\Psi^\prime)$:
\[
\begin{tikzcd}
\mathcal{P} \arrow[rr,bend left=45,"{\Psi}",""{name=u1,below}] \arrow[rr,bend right=45,"{\Psi^\prime}"',""{name=u2,above}] \arrow[Rightarrow,from=u1,to=u2,"F"] & & \mathcal{P}
\end{tikzcd}
\]
The monoidal structure of $CSpin(3)$ can be obtained similarly as in Remark \ref{hopesecondpropproof}.
\end{const}
\begin{rem}
We note that in our conjectural construction, Construction \ref{catextensionofspin}, our aim is to generalise the constructions given in \cite{MR900587}. Taking into account the approaches developed in \cite{MR3535115} and \cite{MR4268834}, we instead focus on providing a description of the categorical group $String(3)$ as categorical group of symmetries of the categorical Hopf map.
\end{rem}

\section{Questions and future work}
Below, we outline several questions and directions for further investigation.
\begin{ques}
For the connection on our basic bundle gerbe $H:\mathcal{P} \longrightarrow S^3$, one may employ the explicit formulas given in \cite{MR3894086}. Once a connection and curving are specified on $H$, this construction yields a connection on the categorical Hopf map. As shown in Propositions \ref{secondprop} and \ref{firstprop}, our results do not depend on the explicit expression for the connection itself. More precisely, what is required is only the associated $2$-form, namely, the curvature of a potential connection. The same observation holds in the categorical setting. \\
Nevertheless, it would be worthwhile to compute and record an explicit formula for the connection on the categorical Hopf map.
\end{ques}
\begin{ques}
One can work through the map $H:\mathcal{P} \longrightarrow S^3$ to obtain finite categorical subgroups of $String(3)$. This may be viewed as a potential categorical analogue of the McKay correspondence. We have finite subgroups of $Spin(3)$. one may attempt to extend these by suitable finite categorical subgroups of $\mathbb{B}U(1)$ to realise corresponding finite categorical subgroups of $String(3)$. \\
These finite categorical subgroups can be thought of as categorical groups of symmetries of categorical analogues of the Platonic solids. Related ideas and examples appear in the work of Epa and Ganter \cite{MR3912053}.  
\end{ques}
\begin{rem}
As it has been discussed in \cite{MR3089401} and \cite{MR2805195}, there are certain differences between Waldorf's and Wockel's formalism for constructing principal $2$-bundles. Although we have followed Wockel's approach, the categorical Hopf map could alternatively be reconstructed using the framework of \cite{MR3089401}. Since our construction does not depend heavily on these differences, we have employed Waldorf’s formalism whenever it proved convenient.
\end{rem}

% Appendices
\begin{appendices}

\section{Calculations and constructions}
\begin{rem}\label{kernelofhopf}
The Hopf map $\eta:S^3 \longrightarrow S^2$ is given by
\[
\eta(x,y,z,w) = \left(2(xz+yw),2(yz-xw),-x^2-y^2+z^2+w^2\right).
\]
This expression arises from the definition
\[
q\mathsf{k}\bar{q},
\]
where $q=x+y\mathsf{i}+z\mathsf{j}+w\mathsf{k}$. The kernel of the Hopf map is therefore
\begin{align*}
K & = \left\{(0,0,\sin(t),\cos(t)) \in S^3 ~ \big\vert ~ t \in \mathbb{R}\right\} \\
& = \cos(t) + \sin(t)\mathsf{k},
\end{align*}
with the base-point in $S^2$ is chosen to be $(0,0,1)$.
\end{rem}
\begin{rem}
We consider the $2$-form
\[
\omega = \frac{1}{2}\left(XdY \wedge dZ + YdZ \wedge dX + ZdX \wedge dY\right)
\]
on $S^2$, together with the $1$-form
\[
\zeta = -ydx+xdy-wdz+zdw
\]
on $S^3$. Let $\eta$ denote the Hopf map
\[
\eta(x,y,z,w) = \left(2(xz+yw),2(yz-xw),-x^2-y^2+z^2+w^2\right).
\]
as introduced in the background section. We anticipate a relation between $\eta^{\ast}\omega$ and $d\zeta$. To this end, we first compute the pullback $2$-form $\eta^{\ast}\omega$.
Define
\begin{align*}
X & = 2(xz+yw) \\
Y & = 2(yz-xw) \\
Z & = -x^2-y^2+z^2+w^2.
\end{align*}
Differentiating, we obtain
\begin{align*}
dX & = 2zdx+2wdy+2xdz+2ydw, \\
dY & = -2wdx+2zdy+2ydz-2xdw, \\
dZ & = -2xdx-2ydy+2zdz+2wdw.
\end{align*}
Substituting into $\omega$, we find
\begin{align*}
\eta^\ast\omega & = (xz+yw)(-2wdx+2zdy+2ydz-2xdw) \wedge (-2xdx-2ydy+2zdz+2wdw) \\
& + (yz-xw)(-2xdx-2ydy+2zdz+2wdw) \wedge (2zdx+2wdy+2xdz+2ydw) \\
& + \frac{1}{2}(-x^2-y^2+z^2+w^2)(2zdx+2wdy+2xdz+2ydw) \wedge (-2wdx+2zdy+2ydz-2xdw).
\end{align*}
Collecting coefficients, we obtain: \\
$\bullet$ Coefficient of $dx \wedge dy$ in $\eta^\ast\omega$ is
\begin{align*}
& (xz+yw)(4yw+4xz) + (yz-xw)(-4xw+4yz) + (-x^2-y^2+z^2+w^2)(2z^2+2w^2) \\
& = 4xyzw+4x^2z^2+4y^2w^2+4xyzw-4xyzw+4y^2z^2+4x^2w^2-4xyzw-2x^2z^2 \\
& -2x^2w^2-2y^2z^2-2y^2w^2+2z^4+2z^2w^2+2z^2w^2+2w^4 = 2x^2z^2+2y^2w^2+2y^2z^2 \\
& +2x^2w^2+2z^4+2z^2w^2+2z^2w^2+2w^4 \\
& = 2(z^2+w^2).
\end{align*}
$\bullet$ Coefficient of $dx \wedge dz$ in $\eta^\ast\omega$ is
\begin{align*}
& (xz+yw)(-4zw+4xy) + (yz-xw)(-4x^2-4z^2) + (-x^2-y^2+z^2+w^2)(2yz+2xw) = \\
& -4xz^2w+4x^2yz-4yzw^2+4xy^2w-4x^2yz-4yz^3+4x^3w+4xz^2w-2x^2yz-2x^3w \\
& -2y^3z-2xy^2w+2yz^3+2xz^2w+2yzw^2+2xw^3 = -2yzw^2+2xy^2w-2yz^3+2x^3w \\
& -2x^2yz-2y^3z+2xz^2w+2xw^3 = 2z^2(xw-yz)+2y^2(xw-yz)+2x^2(xw-yz)+2w^2(xw-yz) \\
& = 2(xw-yz).
\end{align*}
$\bullet$ Coefficient of $dx \wedge dw$ in $\eta^\ast\omega$ is
\begin{align*}
& (xz+yw)(-4w^2-4x^2)+(yz-xw)(-4xy-4zw) + (-x^2-y^2+z^2+w^2)(-2xz+2yw) \\
& = -4xzw^2-4x^3z-4yw^3-4x^2yw-4xy^2z-4yz^2w+4x^2yw+4xzw^2+2x^3z-2x^2yw \\
& +2xy^2z-2y^3w-2xz^3+2yz^2w-2xzw^2+2yw^3 = -2x^2(xz+yw) - 2y^2(xz+yw) - 2z^2(xz+yw) \\
& -2w^2(xz+yw) \\
& = -2(xz+yw).
\end{align*}
$\bullet$ Coefficient of $dy \wedge dz$ in $\eta^\ast\omega$ is
\begin{align*}
& (xz+yw)(4z^2+4y^2) + (yz-xw)(-4xy-4zw) + (-x^2-y^2+z^2+w^2)(2yw-2xz) \\
& = 4xz^3+4xy^2z+4yz^2w+4y^3w-4xy^2z-4yz^2w+4x^2yw+4xzw^2-2x^2yw \\
& +2x^3z-2y^3w+2xy^2z+2yz^2w-2xz^3+2yw^3-2xzw^2 \\
& = 2(xz+yw).
\end{align*}
$\bullet$ Coefficient of $dy \wedge dw$ in $\eta^\ast\omega$ is
\begin{align*}
& (xz+yw)(4zw-4xy)+(yz-xw)(-4y^2-4w^2)+(-x^2-y^2+z^2+w^2)(-2xw-2yz) \\
& = 4xz^2w-4x^2yz+4yzw^2-4xy^2w-4y^3z-4yzw^2+4xy^2w+4xw^3+2x^3w+2x^2yz \\
& +2xy^2w+2y^3z-2xz^2w-2yz^3-2xw^3-2yzw^2 = 2x^2(xw-yz) + 2y^2(xw-yz) + 2z^2(xw-yz) \\
& +2w^2(xw-yz) \\
& = 2(xw-yz).
\end{align*}
$\bullet$ Coefficient of $dz \wedge dw$ in $\eta^\ast\omega$ is
\begin{align*}
& (xz+yw)(4yw+4xz) + (yz-xw)(4yz-4xw) + (-x^2-y^2+z^2+w^2)(-2x^2-2y^2) \\
& = 4xyzw+4x^2z^2+4xyzw+4y^2w^2+4y^2z^2-4xyzw-4xyzw+4x^2w^2+2x^4+2x^2y^2 \\
& +2x^2y^2+2y^4-2x^2z^2-2y^2z^2-2x^2w^2-2y^2w^2 = 2x^2z^2+2y^2w^2+2y^2z^2 \\
& +2x^2w^2+2x^4+2x^2y^2+2x^2y^2+2y^4 \\
& = 2(x^2+y^2).
\end{align*}
Therefore, we obtain:
\begin{align*}
\eta^\ast\omega & = 2(z^2+w^2)dx \wedge dy + 2(xw-yz)dx \wedge dz - 2(xz+yw)dx \wedge dw \\
& + 2(xz+yw)dy \wedge dz + 2(xw-yz)dy \wedge dw + 2(x^2+y^2)dz \wedge dw.
\end{align*}
On $S^3$, we have the identity:
\[
xdx+ydy+zdz+wdw=0.
\]
Using this relation to simplify $\eta^\ast\omega$, we get:
\begin{align*}
\eta^\ast\omega & = 2(x^2+y^2+z^2+w^2)(dx \wedge dy + dz \wedge dw) \\
& = 2dx \wedge dy + 2dz \wedge dw \\
& = d\zeta.
\end{align*}
\end{rem}
\begin{const}\label{fouropens}
In Construction \ref{cathopf}, we considered the cover of $S^2$ consisting of six opens. This cover is compatible with the action of the octahedral group $\mathsf{Oct}$ on $S^2$ in the sense that $\mathsf{Oct}$ permutes the open sets in this cover. \\
We now introduce an alternative cover of $S^2$ consisting of four opens, obtained from the projections of the faces of a tetrahedron inscribed in $S^2$, to $S^2$. This cover is compatible with the action of the tetrahedral group $\mathsf{Tet}$ on $S^2$ in the sense that elements of $Tet$ permute the open sets in this cover. \\
Fix $\epsilon > 0$ and define the cover $\left\{X_1,X_2,X_3,X_4\right\}$ of $S^2$ as follows:
\[
X_1 = \left\{(X,Y,Z) \in S^2 ~ \big\vert ~ z < \epsilon\right\}
\]
The remaining opens $X_2$, $X_3$ and $X_4$ correspond to small neighbourhoods of the spherical projections of three faces of the tetrahedron. In spherical coordinates, we may describe them by
\begin{align*}
X_2 & = \left\{\bigl(\cos(\theta)\sin(\phi),\sin(\theta)\sin(\phi),\cos(\phi)\bigr) \in S^2 ~ \big\vert ~ -\epsilon < \theta < \frac{2\pi}{3} + \epsilon ~ \text{and} ~ -\epsilon < \phi < \frac{2\pi}{3} + \epsilon \right\}, \\
X_3 & = \left\{\bigl(\cos(\theta)\sin(\phi),\sin(\theta)\sin(\phi),\cos(\phi)\bigr) \in S^2 ~ \big\vert ~ \frac{2\pi}{3} - \epsilon < \theta < \frac{4\pi}{3} + \epsilon ~ \text{and} ~ -\epsilon < \phi < \frac{2\pi}{3} + \epsilon\right\}, \\
X_4 & = \left\{\bigl(\cos(\theta)\sin(\phi),\sin(\theta)\sin(\phi),\cos(\phi)\bigr) \in S^2 ~ \big\vert ~ \frac{4\pi}{3} - \epsilon < \theta < 2\pi + \epsilon ~ \text{and} ~ -\epsilon < \phi < \frac{2\pi}{3} + \epsilon\right\}.
\end{align*}
\end{const}

\end{appendices}

%bibliography
\newpage
\bibliographystyle{alpha}
\bibliography{biblio}
\end{document}